\documentclass[12pt,a4paper,english,reqno]{amsart}
\usepackage[latin1]{inputenc}
\usepackage{graphicx}
\usepackage{amssymb}

\makeatletter
 \theoremstyle{plain}    
 \newtheorem{thm}{Theorem}[section]
 \numberwithin{equation}{section} %% Comment out for sequentially-numbered
 \numberwithin{figure}{section} %% Comment out for sequentially-numbered
 \theoremstyle{plain}
 \theoremstyle{definition}
 \newtheorem{defn}[thm]{Definition}
 \theoremstyle{remark}
 \newtheorem{rem}[thm]{Remark}
 \theoremstyle{plain}    
 \newtheorem{lem}[thm]{Lemma} %%Delete [thm] to re-start numbering
 \theoremstyle{plain}    
 \newtheorem{prop}[thm]{Proposition} %%Delete [thm] to re-start numbering
 \theoremstyle{plain}    
 \newtheorem{cor}[thm]{Corollary} %%Delete [thm] to re-start numbering
 \usepackage{verbatim}
 \theoremstyle{remark}    
 \newtheorem{notation}[thm]{Notation} 

\usepackage{amscd, amssymb, %, nologo
}

\usepackage[all %
]{xy}

\def\al{\alpha}

\def\top{\operatorname{top}} 
\def\soc{\operatorname{soc}}
\def\Ker{\operatorname{Ker}}
\def\Cok{\operatorname{Coker}}
\def\Im{\operatorname{Im}}

\def\Hom{\operatorname{Hom}}
\def\rad{\operatorname{rad}}
\def\Aut{\operatorname{Aut}}
\def\End{\operatorname{End}}
\def\Ext{\operatorname{Ext}}
\def\Tor{\operatorname{Tor}}

\def\mod{\operatorname{mod}}

\def\supp{\operatorname{supp}}

\newcommand{\udim}{\operatorname{\underline{dim}}\nolimits}

\def\ind{\operatorname{ind}}

\def\Irr{\operatorname{Irr}}

\def\Proj{\operatorname{Proj}}

\def\rank{\operatorname{rank}}
\def\chr{\operatorname{char}}

\def\ad{\operatorname{ad}}
\def\injdim{\operatorname{injdim}}

\def\Mat{\operatorname{Mat}}
\def\org{\operatorname{org}}
\def\Rt{\operatorname{Rt}}
\def\Db{{\mathcal D}^{\text{\rm b}}}

\def\dsm#1,#2..#3{\bigoplus_{{#1}={#2}}^{#3}}
\def\sm#1,#2..#3{\sum_{{#1}={#2}}^{#3}}

\def\id{1\kern-.25em{\text{{\rm l}}}} %1\!\!{\text{{\rm l}}}   is good for 12pt

\def\isoto{\ \raise.8ex\hbox{$^{\sim}$}\kern-.7em\hbox{$\to$}\ } 

\def\ya#1{\overset{#1}{\longrightarrow}}
\def\blank{\text{ - }}

\def\Ltimes{\overset{\mathbf{L}}{\otimes}}

\def\bg{%
\family{cmr}\size{20}{12pt}\selectfont}

\def\bigzerou{%
\smash{\lower1.7ex\hbox{\bg 0}}}

\def\repr[#1;#2;#3;#4;#5]{
\left(
\begin{matrix}#1\\#2\end{matrix}
#3
\begin{matrix}#4\\#5\end{matrix}
\right)}
\let\sbjcls=\subjclass
\renewcommand{\subjclass}{\sbjcls[2000]}

\numberwithin{equation}{section}

\makeatletter
\def\@secnumfont{\bfseries}
\makeatother

\usepackage{babel}
\makeatother
\begin{document}

\title{Domestic canonical algebras and simple Lie algebras}

\author{Hideto Asashiba}

\email{shasash@ipc.shizuoka.ac.jp}



\begin{abstract}
For each simply-laced Dynkin graph $\Delta$ we realize the simple
complex Lie algebra of type $\Delta$ as a quotient algebra of the
complex degenerate composition Lie algebra $L(A)_{1}^{\mathbb{C}}$
of a domestic canonical algebra $A$ of type $\Delta$ by some ideal
$I$ of $L(A)_{1}^{\mathbb{C}}$ that is defined via the Hall algebra
of $A$, and give an explicit form of $I$. Moreover, we show that
each root space of $L(A)_{1}^{\mathbb{C}}/I$ has a basis given by
the coset of an indecomposable $A$-module $M$ with root easily computed
by the dimension vector of $M$.
\end{abstract}

\dedicatory{Dedicated to Professor Claus Michael Ringel on the occasion of his
60th birthday}

\maketitle

\section*{\textbf{Introduction}}

Let $A$ be a finite-dimensional algebra over a finite field $k$
with $q$ elements, and consider the free abelian group $\mathcal{H}(A)$
with basis the isoclasses of finite $A$-modules. Then by Ringel \cite{Ri90-Banach}
$\mathcal{H}(A)$ turns out to be an associative ring with identity,
called the \emph{integral} \emph{Hall algebra} of $A$, with respect
to the multiplication whose structure constants are given by the numbers
of filtrations of modules with factors isomorphic to modules that
are multiplied (see \ref{sub:Hall-alg}). The free abelian subgroup
$\overline{L}(A)$ of $\mathcal{H}(A)$ with basis the isoclasses
of finite indecomposable $A$-modules becomes a Lie subalgebra modulo
$q-1$ whose Lie bracket is given by the commutator of the Hall multiplication.
We call this Lie bracket the \emph{Hall commutator}. It would be interesting
to realize all types of simple (complex) Lie algebras using this Hall
commutator.

Along this line, Ringel \cite{Ri90-poly} realized the positive part
of the simple Lie algebra $\frak{g}(\Delta)$ for each Dynkin type
$\Delta$. Further Peng and Xiao \cite{PX97} realized all types of
simple Lie algebras by the so-called root categories of finite-dimensional
representation-finite hereditary algebras. But the Lie bracket was
not completely given by the Hall commutator, because the root category
$\mathcal{R}$ provides only the positive and the negative parts.
The Cartan subalgebra $\mathfrak{h}$ was given by a subgroup of the
Grothendieck group of $\mathcal{R}$ over the field $\mathbb{Q}$
of rational numbers. The Hall commutator was used to define the Lie
bracket only inside $\mathcal{R}$, and when the bracket should not
be closed in $\mathcal{R}$, namely when we deal with an indecomposable
object $X$ in $\mathcal{R}$ of a root $\alpha$ and an indecomposable
object $Y$ in $\mathcal{R}$ of the root $-\alpha$, the definition
of the bracket $[X,Y]$ was changed in order to have $[X,Y]\in\mathfrak{h}$.
In \cite{Asa02} we succeeded to realize general linear algebras and
special linear algebras (see also Iyama \cite{Iya}) by the Hall commutator
defined on cyclic quiver algebras. In this realization also the Cartan
subalgebra was naturally provided together with the positive and the
negative parts. In \cite{Asa02-2} we gave a way how to realize all
types of simple Lie algebras by the Hall commutator using tame hereditary
algebras, in particular we gave an explicit realization of simple
Lie algebras of type $D_{n}$.

However this realization needed some rational constants to define
a necessary ideal of the Lie algebra. In this paper we give another
realization by using domestic canonical algebras. Here we do not use
a surjective Lie algebra homomorphism from an affine Lie algebra (\cite{Kac90-book})
that was an essential tool in \cite{Asa02-2}. In the realization
using a tame hereditary algebra we had to choose some orientation
for the quiver of the algebra. But if we use a canonical algebra we
are free from choosing an orientation of the quiver of the algebra
except for the $A_{n}$ case because the orientation is given from
the beginning. For simplicity we deal only with simply-laced cases.
Non-simply-laced cases may be treated using the generalized definition
of canonical algebras by Ringel \cite{Ri90-can}. We expect that the
same approach works to realize affine Kac-Moody algebras by using
tubular canonical algebras instead of domestic ones. (In fact, Zhengxin
Chen is carrying out this plan, the primary version \cite{ChenZh}
contained a similar error as in the first version of this paper.)
It should be pointed out that in the realization using canonical algebras
the preprojective (resp.$\,$preinjective) component contains only
basis vectors of the positive (resp.$\,$negative) part (see Remarks
\ref{rem:pos-neg} and \ref{sub:Basis-vectors}), in contrast, in
the realization using hereditary (non-canonical) algebras the preprojective
component and the preinjective component contain basis elements of
both positive and negative parts. Finally we mention that there is
a possibility to construct representations of simple Lie algebras
by the form of our realization using infinite-dimensional modules
as done in \cite{Iya}.

The first version contained a serious error that the constructed Lie
algebra may turn out to be zero because the relations required on
it was too much. This problem was fixed in the present version.

The paper is organized as follows. After preliminaries in Sect.$\,$1
we collect necessary facts on Lie algebra constructions using Hall
algebras, and domestic canonical algebras in sections 2 and 3, respectively.
In Sect.$\,$4 we state our main theorem, and Sect.$\,$5 is devoted
to preparations of our proof of the main theorem. We give a proof
of the main theorem in Sect.$\,$6. We next examine root spaces of
the Lie algebra constructed here to prove the remaining theorem in
Sect.$\,$7. Finally in the last section we exhibit an example of
basis vectors of the realization of simple Lie algebra of type $D_{5}$,
and an example that shows an error in the first version of the paper.

\section{\textbf{Preliminaries}}

\subsection{Notation\label{sub:notations}}

Throughout this paper $k$ is a finite field of cardinality $q\ge3$.
When we deal with domestic canonical algebras of type $E_{8}$ (see
Sect.\ \ref{sub:dfn-dom-can} for definition) we assume that $\chr k\ne2$.
For a (finite-dimensional) $k$-algebra $A$, we denote by $\mod A$
the category of finite-dimensional (left) $A$-modules, and by $\ind A$
the full subcategory of $\mod A$ consisting of indecomposable modules.
For an $A$-module $M$, $\top M:=M/\rad M$, $\soc M$, $l(M)$ and
$[M]$ denote the top, the socle, the composition length and the isoclass
of $M$, respectively. For $\mathcal{C}=\mod A,\ind A$ we denote
by $[\mathcal{C}]$ the set of isoclasses of objects in $\mathcal{C}$.
For a field extension $K$ of $k$, we set $V^{K}:=V\otimes_{k}K$
for all $k$-vector spaces $V$. For a set $E$, $|E|$ denotes the
cardinality of $E$. The set of positive integers and the set of non-negative
integers are denoted by $\mathbb{N}$ and by $\mathbb{N}_{0}$, respectively.
For a ring $R$, $R^{\times}$ denotes the set of invertible elements
of $R$. By $\delta_{ij}$ we denote the Kronecker symbol, i.e., $\delta_{ij}=1$
if $i=j$, and $\delta_{ij}=0$ if $i\neq j$. For an abelian group
$L$, we set $L^{\mathbb{C}}:=L\otimes_{\mathbb{Z}}\mathbb{C}$ and
$L^{\mathbb{Q}}:=L\otimes_{\mathbb{Z}}\mathbb{Q}$. For elements $x_{1},\ldots,x_{n}$
of a Lie algebra, we set

\begin{equation}
[x_{1},\ldots,x_{n}]:=[[\cdots[[x_{1},x_{2}],x_{3}],\cdots],x_{n}].\label{eq:multibracket}\end{equation}

For Auslander-Reiten theory we refer to \cite{ARS,Ga831,Ri84} and
for tilting theory to \cite{HR,Ri84,Ha88}. We set $\Gamma_{A}$ to
be the Auslander-Reiten quiver of $A$. $D=\Hom_{k}(-,k)$ and $\tau=\tau_{A}$
denote the usual $k$-duality $\mod A\to\mod A^{\text{op}}$ and the
Auslander-Reiten translation of $A$, respectively.

\subsection{Hall numbers\label{sub:Hall-numbers}}

For $X,Y,Z\in\mod A$, we set \begin{eqnarray*}
\mathcal{F}_{X,*}^{Z} & = & \{ M\mid M\textrm{\text{ is a submodule of }}Z\text{ with }Z/M\cong X\},\\
\mathcal{F}_{*,Y}^{Z} & = & \{ M\mid M\text{ is a submodule of }Z\text{ with }M\cong Y\},\\
\mathcal{F}_{XY}^{Z} & = & \mathcal{F}_{X,*}^{Z}\cap\mathcal{F}_{*,Y}^{Z}\end{eqnarray*}
and the cardinalities of these are denoted by $F_{X,*}^{Z},F_{*,Y}^{Z},$
and $F_{XY}^{Z}$, respectively. $F_{XY}^{Z}$ is called a \emph{Hall
number}. If $X\cong X'$, $Y\cong Y'$ and $Z\cong Z'$ in $\mod A$,
then we clearly have $F_{XY}^{Z}=F_{X'Y'}^{Z'}$. Therefore we may
set $F_{[X][Y]}^{[Z]}:=F_{XY}^{Z}$. Recall the following well-known
formula (the Riedtmann formula) for $A$-modules $X$, $Y$ and $Z$
(see \cite[4.1, 4.3]{Rie94}, \cite[Lemma 3.1]{Pe97}):

\begin{equation}
F_{XY}^{Z}=\frac{|\Ext_{A}^{1}(X,Y)_{Z}|\cdot|\Aut_{A}Z|}{|\Hom_{A}(X,Y)|\cdot|\Aut_{A}X|\cdot|\Aut_{A}Y|}\textrm{ in $\mathbb{Z}$},\label{eq:Hall-coefficient}\end{equation}
where $\Ext_{A}^{1}(X,Y)_{Z}$ is the set of equivalence classes in
$\Ext_{A}^{1}(X,Y)$ of extensions with the middle term $Z$. To compute
the number $F_{XY}^{Z}$ we will use the number\[
W_{XY}^{Z}:=|\{(f,g)\in\Hom_{A}(Y,Z)\times\Hom_{A}(Z,X)\mid0\to Y\ya{f}Z\ya{g}X\to0\text{ is exact}\}|.\]
As easily seen we have the following relationship between $F_{XY}^{Z}$
and $W_{XY}^{Z}$:\[
F_{XY}^{Z}=\frac{W_{XY}^{Z}}{|\Aut_{A}X|\cdot|\Aut_{A}Y|}.\]

\subsection{Representations of quivers}

\begin{defn}
(1) Recall that a \emph{quiver} is a quadruple $Q=(Q_{0},Q_{1},$
$t_{Q},h_{Q})$, where $Q_{0},Q_{1}$ are sets (or classes) and $t_{Q},h_{Q}$
are maps from $Q_{1}$ to $Q_{0}$. Elements of $Q_{0},Q_{1}$ are
called \emph{vertices} and \emph{arrows} of $Q$, respectively, and
for each $\alpha\in Q_{1}$ the vertices $t_{Q}(\alpha),h_{Q}(\alpha)$
are called the \emph{tail} and the \emph{head} of $\alpha$, respectively.
By drawing an arrow $t_{Q}(\alpha)\ya{\alpha}h_{Q}(\alpha)$ for each
$\alpha\in Q_{1}$ we can express $Q$ as an oriented graph. We can
regard categories as quivers by forgetting compositions.

(2) A \emph{morphism} from a quiver $Q$ to a quiver $Q'$ is a pair
$f=(f_{0},f_{1})$ of maps $f_{i}\colon Q_{i}\to Q'_{i}$ for $i=0$
and 1 such that $f_{0}t_{Q}=t_{Q'}f_{1}$ and $f_{0}h_{Q}=h_{Q'}f_{1}$.
This is also written as $f=(f(x),f(\alpha))_{x\in Q_{0},\alpha\in Q_{1}}$,
where we put $f(x):=f_{0}(x),f(\alpha):=f_{1}(\alpha)$ for each $x\in Q_{0}$
and $\alpha\in Q_{1}$.
\end{defn}
(3) A $k$-\emph{representation} of a quiver $Q$ is just a morphism
$V=(V(x),V(\alpha))_{x,\alpha}$ from $Q$ to the category $\mod k$
of finite-dimensional $k$-vector spaces regarded as a quiver. The
definition of \emph{morphisms} between representations of $Q$ is
similar to that of natural transformations between functors.

\begin{rem}
Unless otherwise stated we only deal with \emph{finite} quivers, i.e.
quivers with only finitely many vertices and arrows.
\end{rem}
The vector space $kQ$ with basis the set of all paths in $Q$ turns
out to be a $k$-algebra with identity via the multiplication given
by concatenation of paths. We refer to \cite{Ga831} for details.
When an algebra $A$ is defined by a quiver $Q$ with relations $\rho_{1},\ldots,\rho_{t}$,
say $A=kQ/I$, where $I$ is the admissible ideal of $kQ$ generated
by $\rho_{1},\ldots,\rho_{t}$, we identify $\mod A$ with the category
of representations of $Q$ satisfying the relations $\rho_{1},\ldots,\rho_{t}$
as in \cite{Ga831}. Thus for an $A$-module $M$, regarded as a $k$-representation,
we write $M=(M(x),M(\alpha))_{x\in Q_{0},\alpha\in Q_{1}}$. In fact,
$M(x)=\mathbf{e}_{x}M$ and $M(\alpha)$ is given by the left multiplication
by $\alpha+I\in A$.

\begin{rem}
For simplicity we assume throughout the rest of this paper that $A$
is a finite-dimensional $k$-algebra defined by a quiver $Q=(Q_{0},Q_{1},t_{Q},h_{Q})$
with relations.
\end{rem}
\begin{defn}
\label{sub:supp-alg}For each vertex $x$ of $Q$, we denote by $\mathbf{e}_{x}$
the idempotent of $A$ corresponding to $x$. The \emph{support algebra}
of an $A$-module $M$, denoted by $\supp M$, is defined by\[
\supp M:=A/A\mathbf{e}_{M}A,\]
where $\mathbf{e}_{M}:=\sum_{M(x)=0}\mathbf{e}_{x}$. Note that $\mod\,\supp M$
forms a full subcategory of $\mod A$ closed under extensions.
\end{defn}

\subsection{Grothendieck group\label{sub:Groth-group-der-cat}}

\begin{defn}
\label{sub:K0}(1) The image of an $M\in\mod A$ in the Grothendieck
group $K_{0}(A)$ of $A$ is denoted by $\udim M$ and is called the
\emph{dimension vector} of $M$.

(2) For each vertex $x\in Q_{0}$ we set $S_{x}:=A\mathbf{e}_{x}/\rad A\mathbf{e}_{x}$
to be the simple $A$-module corresponding to $x$, and $e_{x}:=\udim S_{x}$.

(3) For each $v,w\in K_{0}(A)$ we write $v\le w$ if $v_{x}\le w_{x}$
for all $x\in Q_{0}$. This defines a partial order on $K_{0}(A)$.
Further we write $v<w$ if $v\le w$ but $v\ne w$.
\end{defn}
\begin{rem}
(1) The set $\{ e_{x}\mid x\in Q_{0}\}$ forms a basis of $K_{0}(A)$,
by which we regard each element $v=\sum_{x\in Q_{0}}v_{x}e_{x}$ in
$K_{0}(A)$ as a row vector $(v_{x})_{x\in Q_{0}}\in\mathbb{Z}^{Q_{0}}$,
and identify $K_{0}(A)$ with $\mathbb{Z}^{Q_{0}}$. Under this identification
we have $\udim M=(\dim_{k}M(x))_{x\in Q_{0}}$. Note that since we
deal with row vectors, each $\mathbb{Z}$-endomorphism $f$ of $K_{0}(A)$
is expressed by the right multiplication by the corresponding matrix
$F$ as $f(v)=vF$ for all $v\in K_{0}(A)$.

(2) Let $K$ be an arbitrary field extension of $k$ and consider
the $K$-algebra $A^{K}:=A\otimes_{k}K$. By identifying $K_{0}(A^{K})$
with $\mathbb{Z}^{Q_{0}}$ by the same way as above we also have $\udim M^{K}=(\dim_{K}M^{K}(x))_{x\in Q_{0}}$.
Then since $\dim_{k}M(x)=\dim_{K}M^{K}(x)$ for all $x\in Q_{0}$,
we have \[
\udim M^{K}=\udim M.\]

\end{rem}
\begin{defn}
Since the derived category $\Db(\mod A)$ of bounded complexes in
$\mod A$ is a triangulated category, the Grothendieck group $K_{0}(\Db(\mod A))$
is defined by using triangles (\cite{Gr}, \cite[1.1]{Ha88}). For
each $X\in\Db(\mod A)$ we denote by $\udim^{D}X$ the image of $X$
in $K_{0}(\Db(\mod A))$.
\end{defn}
\begin{rem}
\label{rem:der-udim}We regard $\mod A$ as a subcategory of $\Db(\mod A)$
by the canonical embedding $\mod A\to\Db(\mod A)$ sending modules
to complexes concentrated in degree zero, which induces an isomorphism
$K_{0}(A)\to K_{0}(\Db(\mod A))$, $\udim X\mapsto\udim^{D}X$ with
the inverse $\udim^{D}X\mapsto\sum_{i=1}^{\infty}(-1)^{i}\udim X^{i}$(\cite{Gr}).
By this isomorphism we identify $K_{0}(A)$ with $K_{0}(\Db(\mod A))$.
Therefore for all $X\in\Db(\mod A)$ we may write $\udim X=\udim^{D}X$,
and we have
\end{rem}
\[
\udim X=\sum_{i=1}^{\infty}(-1)^{i}\udim X^{i}.\]
In particular, for each $X\in\mathcal{D^{\textrm{b}}}(\mod A)$ and
$i\in\mathbb{Z}$ we have\[
\udim X[i]=(-1)^{i}\udim X.\]

\subsection{Bilinear form and quadratic form}

Let $C$ be the \emph{Cartan matrix} of $A$, namely the matrix whose
$(i,j)$-entry is given by $\dim\mathbf{e}_{i}A\mathbf{e}_{j}$ for
all $i,j\in Q_{0}$ (Definition\ref{sub:supp-alg}).

\begin{defn}
\label{bilinear} If the global dimension of $A$ is finite, say at
most $d\in\mathbb{N}$, then $C$ is invertible and we can define
a bilinear form $B_{A}$ by\[
B_{A}(v,w)=vC^{-T}w^{T}\]
for all $v,w\in K_{0}(A)\cong\mathbb{Z}^{Q_{0}}$ ($C^{-T}$ denotes
the inverse matrix of the transposed matrix $C^{T}$ of $C$).
\end{defn}
\begin{rem}
In the setting above the following is well-known:\[
B_{A}(\udim X,\udim Y)=\sum_{i=0}^{d}(-1)^{i}\dim\Ext_{A}^{i}(X,Y)\]
for all $A$-modules $X$, $Y$ (\cite[Lemma 2.4]{Ri84}).
\end{rem}
\begin{defn}
\label{quadratic}(1) We denote by $\chi_{A}$ the corresponding quadratic
form, namely\[
\chi_{A}(v):=B_{A}(v,v)\]
for all $v\in K_{0}(A)$.

(2) An element $v\in K_{0}(A)$ is called a \emph{root} (resp.\ 
a \emph{radical}) of $\chi_{A}$ if $\chi_{A}(v)=1$ (resp.\  $\chi_{A}(v)=0$).

(3) We set $\rad\chi_{A}:=\{ v\in K_{0}(A)\mid\chi_{A}(v)=0\}$ and
call it the \emph{radical} of $\chi_{A}$. 
\end{defn}

\subsection{Exceptional modules\label{sub:exceptional}}

Recall that an $A$-module $X$ is called \emph{exceptional} if $X$
is indecomposable and $\Ext_{A}^{1}(X,X)=0$. We take an algebraic
closure $\overline{k}$ of $k$, and set $\Omega=\Omega_{A}$ to be
the set of all finite field extensions $K$ of $k$ contained in $\overline{k}$
such that $(\End_{A}X)^{K}$ is a field for all exceptional $A$-modules
$X$. We set $\mathcal{E}_{ex}(A):=\{\End_{A}(X)\mid X\textrm{ is exceptional}\}/\cong$
and $\mathcal{E}(A):=\{\End_{A}(X)\mid X\textrm{ is simple}\}/\cong$.
(In the simply-laced cases our domestic canonical algebras $A$ defined
in the next section are defined by quivers with relations and we always
have $\mathcal{E}(A)=\{ k\}$. Therefore in our case we can omit this
notation, but we keep it here because it is needed in the non-simply-laced
cases and it tells us how to generalize our argument.)

\begin{lem}
\label{lem:endo-exc}If $A$ is an algebra derived equivalent to a
hereditary algebra $H$, then $\mathcal{E}(A)\subseteq\mathcal{E}_{ex}(A)\subseteq\mathcal{E}_{ex}(H)=\mathcal{E}(H)$.
Therefore in particular, $\Omega_{A}$ is an infinite set.
\end{lem}
\begin{proof}
Since simple modules are exceptional, both $\mathcal{E}(A)\subseteq\mathcal{E}_{ex}(A)$
and $\mathcal{E}(H)\subseteq\mathcal{E}_{ex}(H)$ are trivial. In
the hereditary case it is known that the converse inclusion is also
true (\cite{Ri94-braid}), thus we have $\mathcal{E}(H)=\mathcal{E}_{ex}(H)$.
We only have to show that $\mathcal{E}_{ex}(A)\subseteq\mathcal{E}_{ex}(H)$.
Let $F\colon\Db(\mod A)\to\Db(\mod H)$ be a triangle-equivalence,
and $X$ an exceptional $A$-module. Then $FX[i]\in\mod H$ for some
$i\in\mathbb{Z}$ because $FX$ is an indecomposable complex and $H$
is hereditary. It is obvious from the construction that $FX[i]$ is
an exceptional $H$-module. The algebra isomorphisms $\End_{A}(X)\cong\Db(\mod A)(X,X)\cong\Db(\mod H)(FX,FX)\cong\End_{H}(FX[i])$
show that $\mathcal{E}_{ex}(A)\subseteq\mathcal{E}_{ex}(H)$.
\end{proof}

\section{\textbf{Lie algebras defined by the Hall multiplication}}

\subsection{Hall algebras}

Since $A$ is a finite-dimensional $k$-algebra with $k$ a finite
field, $A$ is a \emph{finitary} ring as shown in Ringel \cite{Ri90-Banach},
i.e., $\Ext_{A}^{1}(X,Y)$ is a finite group for all $X,Y\in\mod A$.

\begin{defn}
\label{sub:Hall-alg}The free abelian group $\mathcal{H}(A)$ with
basis $\{ u_{[X]}\}_{[X]\in[\mod A]}$ together with the multiplication
defined by \[
u_{[X]}u_{[Y]}:=\sum_{[Z]\in[\mod A]}F_{[X][Y]}^{[Z]}u_{[Z]}\]
is called the \emph{integral Hall algebra} of $A$.
\end{defn}
Ringel \cite{Ri90-Banach} proved the following.

\begin{lem}
\emph{$\mathcal{H}(A)$} is an associative ring with the identity
$1=u_{[0]}$.\qed
\end{lem}

\subsection{Lie algebras\label{Hall}}

\begin{defn}
Let $\overline{L}(A)$ be the free abelian subgroup of $\mathcal{H}(A)$
with basis $\{ u_{\alpha}\}_{\alpha\in[\ind A]}$. We set $L/(a):=L/aL$
for all $\mathbb{Z}$-modules $L$ and $a\in\mathbb{Z}$, and denote
elements $x+a\overline{L}(A)$ of $\overline{L}(A)/(a)$ ($x\in\overline{L}(A)$)
simply by $x$.
\end{defn}
We have the following by Ringel \cite[Proposition 3]{Ri92-Tsukuba}
(see also Ringel \cite[Proposition 1]{Ri92-From}).

\begin{lem}
The free $\mathbb{Z}/(q-1)\mathbb{Z}$-module $\overline{L}(A)/(q-1)$
is a Lie subalgebra of $\mathcal{H}(A)/(q-1)$ with the Lie bracket
\[
[u_{[X]},u_{[Y]}]=\sum_{[Z]\in[\ind A]}(F_{[X][Y]}^{[Z]}-F_{[Y][X]}^{[Z]})u_{[Z]}\]
 for each $[X],[Y]\in[\ind A]$.\emph{\hfil\qed}
\end{lem}
Note that in the right hand side of the formula above the sum may
be taken only over $[Z]\in[\ind A]$ such that $\udim Z=\udim X+\udim Y$.
Hence if we put $(\overline{L}(A)/(q-1))_{d}$ to be the free $\mathbb{Z}/(q-1)\mathbb{Z}$-submodule
with the basis $\{ u_{[X]}\mid[X]\in[\ind A],\udim X=d\}$ for all
$d\in K_{0}(A)$, we have $[(\overline{L}(A)/(q-1))_{d},(\overline{L}(A)/(q-1))_{e}]\subseteq(\overline{L}(A)/(q-1))_{d+e}$
for all $d,e\in K_{0}(A)$. Thus we have the following.

\begin{prop}
$\overline{L}(A)/(q-1)=\bigoplus_{d\in K_{0}(A)}(\overline{L}(A)/(q-1))_{d}$
\emph{}is a $K_{0}(A)$-graded Lie algebra. \hfil\qed
\end{prop}

\subsection{Composition Lie algebras}

\begin{defn}
For $A$-modules $X$, $Y$ and $Z$ a polynomial $\varphi_{ZX}^{Y}(T)\in\mathbb{Z}[T]$
in an indeterminate $T$ with integral coefficients is called a \emph{Hall
polynomial} for the triple $(X,Y,Z)$ if $F_{Z^{K}X^{K}}^{Y^{X}}=\varphi_{ZX}^{Y}(|K|)$
for all $K\in\Omega_{A}$.
\end{defn}
Note that when the set $\Omega_{A}$ is infinite, a Hall polynomial
for a triple is uniquely determined if it exists.

\begin{defn}
Assume that $\Omega_{A}$ is an infinite set. By Lemma \ref{Hall},
$\overline{L}(A^{K})/(|K|-1)$ is a Lie subalgebra of $\mathcal{H}(A^{K})/(|K|-1)$
over $\mathbb{Z}/(|K|-1)\mathbb{Z}$ for each $K\in\Omega$. Consider
the Lie algebra over $\mathbb{Z}$ given by the direct product of
Lie algebras:\[
\Pi=\Pi_{A}:=\prod_{K\in\Omega}\overline{L}(A^{K})/(|K|-1).\]

(1) An $A$-module $X$ is called \emph{absolutely indecomposable}
(with respect to $\Omega$) if\begin{equation}
X^{K}\textrm{ is indecomposable for all }K\in\Omega.\label{comp-ind}\end{equation}
We write $\mathbf{u}_{[X]}:=(u_{[X^{K}]})_{K\in\Omega}\in\Pi$ if
$X$ is absolutely indecomposable. Note that all simple modules are
absolutely indecomposable.

(2) The Lie subalgebra of $\Pi$ generated by $\{\mathbf{u}_{[S]}\mid S\textrm{ is simple}\}$
is denoted by $L(A)_{1}$ and is called the \emph{degenerate composition
Lie algebra} of $A$. The Lie algebra $L(A)_{1}$ is not a torsion
$\mathbb{Z}$-module because $\Omega$ is an infinite set.
\end{defn}
\begin{lem}
\label{eq:bracket-Hall-poly}Let $X$, $Y$, $Z$ be absolutely indecomposable
$A$-modules with $\udim X+\udim Z=\udim Y$ such that $Y$ is the
unique indecomposable $A$-module with dimension vector $\udim X+\udim Z$
up to isomorphisms. If there exist Hall polynomials $\varphi_{ZX}^{Y}$
and $\varphi_{XZ}^{Y}$, then\[
[\mathbf{u}_{Z},\mathbf{u}_{X}]=(\varphi_{ZX}^{Y}(1)-\varphi_{XZ}^{Y}(1))\mathbf{u}_{Y}\]
in $\Pi$.
\end{lem}
\begin{proof}
This follows from $F_{Z^{K}X^{K}}^{Y^{K}}=\varphi_{ZX}^{Y}(1)$ in
$\mathbb{Z}/(|K|-1)\mathbb{Z}$ for all $K\in\Omega_{A}$.
\end{proof}
The following seems to be well-known.

\begin{prop}
\label{lem:rep-fin-hered}Let $\Delta$ be a simply-laced Dynkin graph.
If $A$ is a connected representation-finite hereditary algebra of
type $\Delta$, and $M$ an indecomposable $A$-module. Then $\mathbf{u}_{M}\in L(A)_{1}$.
\end{prop}
\begin{proof}
Since in this case $\Omega_{A}$ is an infinite set, $L(A)_{1}$ is
defined. We prove the assertion by induction on $\dim M$. If $\dim M=1$,
then $M$ is simple, and $\mathbf{u}_{M}\in L(A)_{1}$. Assume that
$\dim M>1$. Then as easily seen there exists a simple $A$-module
$S$ and an indecomposable $A$-module $N$ such that $\udim M=\udim S+\udim N$.
By putting $a:=\varphi_{SN}^{M}(1)-\varphi_{NS}^{M}(1)$ we have $[\mathbf{u}_{S},\mathbf{u}_{N}]=a\mathbf{u}_{M}$
in $\Pi$ by Lemma \ref{eq:bracket-Hall-poly}. Here by \cite{Ri90-poly}
precisely one of the two values $\varphi_{SN}^{M}(1)$ and $\varphi_{NS}^{M}(1)$
is nonzero, and the nonzero value is in $\{\pm1,\pm2,\pm3\}$. Therefore
$a\in\{\pm1,\pm2,\pm3\}$. Since $\Delta$ is simply-laced, we have
$a=\pm1$. Hence $\mathbf{u}_{M}=\frac{1}{a}[\mathbf{u}_{S},\mathbf{u}_{N}]\in L(A)_{1}$
because by induction hypothesis $\mathbf{u}_{N}\in L(A)_{1}$.
\end{proof}

\section{\textbf{Canonical algebras}}

\subsection{Canonical algebras\label{sub:dfn-dom-can}}

Among canonical algebras we consider, in this paper, only domestic
canonical algebras given by quivers with relations. Namely, a domestic
canonical algebra $A$ is given by the quiver $Q$ in Figure \ref{quiver-dom-can},%
\begin{figure}
\begin{center}\includegraphics[%
  angle=90]{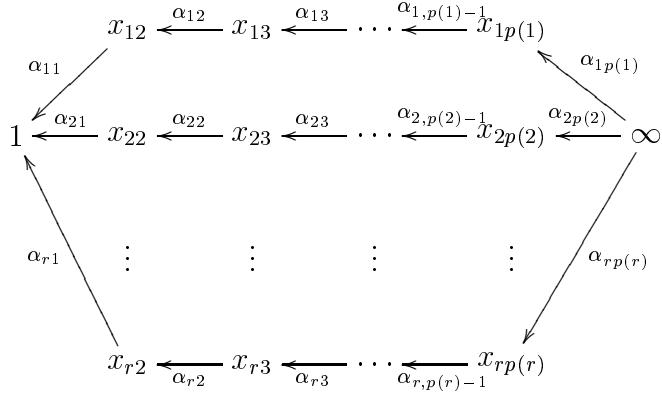}\end{center}

\caption{\label{quiver-dom-can}Quiver of a canonical algebra}
\end{figure}
 where $r\in\{2,3\}$, $p(i)\geq p(i+1)\geq1$ for all $1\leq i\leq r-1$,
with no relation when $r=2$; and with the relation $\sum_{i=1}^{r}\alpha_{i1}\cdots\alpha_{ip(i)}=0$
when $r=3$. Further when $r=3$ it is assumed that \[
(p(1),p(2),p(3))\in\{(d,2,2),(3,3,2),(4,3,2),(5,3,2)\}\]
 for some $d\geq2$. For convenience we set $x_{11}=x_{21}=\cdots=x_{r1}=1$,
$x_{1,p(1)+1}=x_{2,p(2)+1}=\cdots=x_{r,p(r)+1}=\infty$, and give
a partial order on $Q_{0}$ by setting $x_{ij}<x_{i,j+1}$ for all
$1\leq i\leq r$ and $1\leq j\leq p(i)$. Denote by $Q^{l}$ the quiver
obtained from $Q$ by deleting the vertex $\infty$. Note that the
underlying graph $\Delta$ of $Q^{l}$ is a (simply-laced) Dynkin
graph, which is called the \emph{type} of $A$. Conversely every simply-laced
Dynkin graph $\Gamma$ is obtained in this way, and the canonical
algebra of the type $\Gamma$ is uniquely determined if $\Gamma$
is not of type $A_{n}$. We set $\Delta_{0}:=Q_{0}^{l}=Q_{0}\setminus\{\infty\}$
to be the set of vertices of $\Delta$, $n:=|\Delta_{0}|$, and denote
by $(a_{xy})_{x,y\in\Delta_{0}}$ the Cartan matrix expressed by the
graph $\Delta$, namely

\begin{equation}
a_{xy}=\left\{ \begin{array}{ll}
2 & \textrm{if $x=y$;}\\
-1 & \textrm{if $x\neq y$, and $x,y$ are neighbors in $\Delta$; and}\\
0 & \textrm{if $x\neq y$, and $x,y$ are not neighbors in $\Delta$,}\end{array}\right.\label{eq:a-xy}\end{equation}
 where vertices $x,y\in\Delta_{0}$ are said to be \emph{neighbors}
in $\Delta$ if they are connected by an edge in $\Delta$.

Throughout the rest of this paper we assume that $A$ \emph{is a domestic
canonical algebra}.

\subsection{Domestic canonical algebras as tame concealed algebras\label{sub:tame-concealed}}

Note that $\Gamma_{A}$ has a \emph{preprojective} component (\cite[p.80]{Ri84}),
which contains a unique \emph{complete slice} $\mathcal{S}$ (\cite[7.1]{HR},
cf.$\,$\cite[p.180]{Ri84}) with $P_{\infty}:=A\mathbf{e}_{\infty}$
the unique source. Let $T$ be the corresponding \emph{slice module}
(\cite[p.183]{Ri84}), which is a tilting module for $A$. Then $H:=\End_{A}(T)^{\textrm{op}}$
is a tame hereditary algebra, whose quiver is obtained by giving an
orientation to the affine graph $\Delta^{(1)}$ corresponding to the
type $\Delta$ of $A$. Thus $A$ is a \emph{tilted} algebra (\cite{HR},
\cite[4.2]{Ri84}) or more precisely a \emph{tame concealed} algebra
(\cite[4.3]{Ri84}), and hence the global dimension of $A$ is at
most 2, and it is derived equivalent to the hereditary algebra $H$
by \cite[Theorem 2.10]{Ha88} or \cite[Theorem 6.4]{Ric89}. Denote
by $\mathcal{F}$ and $\mathcal{T}$ (resp.$\,$$\mathcal{Y}$ and
$\mathcal{X}$) the torsion-free class and the torsion class in $\mod A$
(resp.$\,$$\mod H$), respectively, defined by the tilting module
$T$. Note that the torsion pair $(\mathcal{T},\mathcal{F})$ splits,
i.e., we have a disjoint union $\ind A=(\ind A\cap\mathcal{T})\sqcup(\ind A\cap\mathcal{F})$,
whereas in general the torsion pair $(\mathcal{X},\mathcal{Y})$ does
not split, thus $\ind H\supsetneq(\ind H\cap\mathcal{X})\sqcup(\ind H\cap\mathcal{Y})$.
Set\[
F:=\Hom_{A}(T,-),F':=\Ext_{A}^{1}(T,-),\hat{F}:=\mathbf{R}\Hom_{A}^{\centerdot}(T,-),\]
\[
G:=T\otimes_{H}-,G':=\Tor_{1}^{H}(T,-),\hat{G}:=T\Ltimes_{H}-\]
Then as well-known we have quasi-inverse pairs of equivalences and
triangle-equiva\-lences$$
\xymatrix{
\mathcal{T}\ar@<1ex>[r]^F&\mathcal{Y}\ar@<1ex>[l]^G
},
\xymatrix{
\mathcal{F}\ar@<1ex>[r]^{F'}&\mathcal{X}\ar@<1ex>[l]^{G'}
},\text{ and}
\xymatrix{
\Db(\mod A)\ar@<1ex>[r]^{\hat{F}}&\Db(\mod H)\ar@<1ex>[l]^{\hat{G}}
}.
$$

Since $H$ is given by a quiver, we have $\mathcal{E}(H)=\{ k\}$.
Then by Lemma \ref{lem:endo-exc} we have the following.

\begin{lem}
\label{sub:Endo-algebras-of}\emph{$\mathcal{E}_{ex}(A)=\mathcal{E}(A)=\{ k\}$,}
and $\Omega_{A}$ is an infinite set.
\end{lem}
\qed

\subsection{Bilinear form, quadratic form and rank}

Since the global dimension of $A$ is finite (Sect.$\:$\ref{sub:tame-concealed}),
the bilinear form $B:=B_{A}$ is defined (Definition \ref{bilinear}).
Denote by $r_{T}\colon K_{0}(A)\to K_{0}(H)$ the isomorphism defined
by $r_{T}(\udim X)=\udim\hat{F}X$ for all $X\in\Db(\mod A)$. Then
as well-known $B_{A}(x,y)=B_{H}(r_{T}(x),r_{T}(y))$ for all $x,y\in K_{0}(A)$,
in particular, we have $\chi_{A}(x)=\chi_{H}(r_{T}(x))$ for all $x\in K_{0}(A)$
(\cite[Proposition III.1.5]{Ha88}). Thus $\rad\chi_{A}$ is isomorphic
to $\rad\chi_{H}$, which is well-known to be a free abelian group
of rank 1. Then since $\delta:=(1,1,\ldots,1)\in\rad\chi_{A}$, we
have $\rad\chi_{A}=\mathbb{Z}\delta$. Thus $\delta$ is the minimal
positive radical vector of $\chi_{A}$.

We set $\rho:=(1,0,\ldots,0,-1)\in\mathbb{Z}^{n+1}$. For an element
$v\in K_{0}(A)\cong\mathbb{Z}^{n+1}$ we set

\[
\rank v:=v_{1}-v_{\infty}=v\rho^{T}\]
and call it the \emph{rank} of $v$, and for an $A$-module $M$ we
set

\[
\rank M:=\rank(\udim M)=\dim M(1)-\dim M(\infty)=(\udim M)\rho^{T},\]
which is called the \emph{rank} of $M$.

A direct calculation shows that

\[
B(v,w)=\left\{ \begin{array}{ll}
\sum_{x\in Q_{0}}v_{x}w_{x}-\sum_{x\to y}v_{x}w_{y} & \textrm{if }\Delta\in\{ A_{n}|n\in\mathbb{N}\}\\
\sum_{x\in Q_{0}}v_{x}w_{x}-\sum_{x\to y}v_{x}w_{y}+v_{\infty}w_{1} & \textrm{if }\Delta\not\in\{ A_{n}|n\in\mathbb{N}\}\end{array}\right.\]
for all $v,w\in K_{0}(A)$, where the sum $\sum_{x\to y}$ is taken
over all pairs $(x,y)\in Q_{0}\times Q_{0}$ such that there exists
an arrow from $x$ to $y$ in $Q$. This immediately yields

\begin{eqnarray}
B(\delta,v) & = & -\rank v\nonumber \\
B(v,\delta) & = & \rank v.\label{eq:B-delta-ex}\end{eqnarray}
for all $v\in K_{0}(A)$.

\subsection{Lost indecomposable modules}

Since $H$ is hereditary, each indecomposable complex in $\Db(\mod H)$
is isomorphic to a complex concentrated in one degree. In other words
each indecomposable complex in $\Db(\mod H)$ is regarded as an $H$-module
up to shifts. But the corresponding statement does not hold for $A$
in general. An indecomposable $H$-module $X$ is sent by $\hat{G}$
to a complex of $A$-modules that cannot be isomorphic to an $A$-module
up to shifts if and only if $X\not\in\mathcal{X}\cup\mathcal{Y}$.
Thus, when we pass from $\overline{L}(H)_{(q-1)}$ to $\overline{L}(A)_{(q-1)}$
we lose the basis $u_{[X]}$ for such an $X$. Therefore $L(A)_{1}^{\mathbb{C}}$
would not realize the positive part of the affine Kac-Moody algebra
of type $\Delta^{(1)}$, which was realized as $L(H)_{1}^{\mathbb{C}}$
by a part of \cite[Theorem 4.7]{PX2000} (see also \cite[Theorems 2 and 3]{Ri92-From}).
In this connection it would be interesting to know which indecomposable
complex of $A$-modules can be an $A$-module up to shifts. This is
the case if and only if positive and negative entries are not \emph{mixed}
in its dimension vector. Namely, we have the following.

\begin{lem}
\label{lem:cpx-mod}Let $X\in\mathcal{D^{\textrm{b}}}(\mod A)$ be
indecomposable. Then $X[i]\in\ind A$ for some $i\in\mathbb{Z}$ if
and only if $\udim X>0$ or $\udim X<0$. 
\end{lem}
\begin{proof}
($\Rightarrow$). If $X[i]\in\ind A$ for some $i\in\mathbb{Z}$,
then $0<\udim X[i]=(-1)^{i}\udim X$ by \ref{sub:Groth-group-der-cat}.
Hence $\udim X>0$ or $\udim X<0$. 

($\Leftarrow$). Since the torsion pair $(\mathcal{F},\mathcal{T})$
in \ref{sub:tame-concealed} splits, $X[i]\in\ind A$ if and only
if $X[i]\in\mathcal{F}$ or $X[i]\in\mathcal{T}$ for all indecomposable
complexes $X\in\Db(\mod A)$ and $i\in\mathbb{Z}$.

Now assume that $X[i]\not\in\ind A$ for all $i\in\mathbb{Z}$. It
is enough to show that $\udim X\not>0$ and $\udim X\not<0$. Since
$H$ is hereditary, there exists some $i\in\mathbb{Z}$ such that
$Y:=\hat{F}X[i]\in\mod H$. It follows from the assumption that $X[i]\not\in\mathcal{F}$
and $X[i]\not\in\mathcal{T}$. Therefore $Y\not\in\mathcal{Y}$ and
$Y\not\in\mathcal{X}$. Let $\mathcal{P}_{H}$, $\mathcal{R}_{H}$,
and $\mathcal{I}_{H}$ be the preprojective component, the tubular
family, and the preinjective component of the Auslander-Reiten quiver
$\Gamma_{H}$ of $H$, respectively. Since $\mathcal{P}_{H},\mathcal{R}_{H}\subseteq\mathcal{Y}$,
we have $Y\in\mathcal{I}_{H}$. Consider the canonical exact sequence\[
0\to Y'\ya{\mu}Y\ya{\varepsilon}Y''\to0\]
with $Y'\in\mathcal{X}$ and $Y''\in\mathcal{Y}$. Then\begin{eqnarray*}
\udim X[i] & = & \udim\hat{G}Y\\
 & = & \udim\hat{G}Y'+\udim\hat{G}Y''\\
 & = & \udim GY'-\udim G'Y'+\udim GY''-\udim G'Y''\\
 & = & \udim GY''-\udim G'Y'.\end{eqnarray*}
Hence \begin{equation}
\udim X[i]=\udim GY''-\udim G'Y'.\label{eq:udimX[i]}\end{equation}
We first show that $(\udim X[i])_{1}<0$. Let $I_{1}$ be the injective
hull of $S_{1}$. Suppose that $\Hom_{H}(Y'',FI_{1})\ne0$. Then since
$\Hom_{H}(\varepsilon,FI_{1})\colon\Hom_{H}(Y,FI_{1})\to\Hom_{H}(Y'',FI_{1})$
is an epimorphism, we have $\Hom_{H}(Y,FI_{1})\ne0$. This shows that
$Y$ is a predecessor of $FI_{1}$ in $\mathcal{I}_{H}$. Since $FI_{1}\in\mathcal{Y}$
and $\mathcal{Y}\cap\mathcal{I}_{H}$ is closed under predecessors
in $\mathcal{I}_{H}$, we have $Y\in\mathcal{Y}$, a contradiction.
Therefore we must have $\Hom_{H}(Y'',FI_{1})=0$. Then since $Y'',FI_{1}\in\mathcal{Y}$,
we have $\Hom_{H}(GY'',I_{1})=0$, which shows that $(\udim GY'')_{1}=0$.
Whereas since $G'Y'\in\mathcal{F}\subseteq\mathcal{P}$, we have $(\udim G'Y')_{1}>(\udim G'Y')_{\infty}\ge0$.
Hence by (\ref{eq:udimX[i]}) we have $(\udim X[i])_{1}<0$. We next
show that $(\udim X[i])_{\infty}>0$. Since $Y''\in\mathcal{Y}\cap\mathcal{I}_{H}$,
we have $GY''\in\mathcal{I}$, and hence $(\udim GY'')_{\infty}>(\udim GY'')_{1}\ge0$.
Further since $G'Y'$ is not a successor of $P_{\infty}$, we have
$(\udim G'Y')_{\infty}=0$. Hence by (\ref{eq:udimX[i]}) we have
$(\udim X[i])_{\infty}>0$. As a consequence we have $\udim X[i]\not>0$
and $\udim X[i]\not<0$, which implies that $\udim X\not>0$ and $\udim X\not<0$
by Remark \ref{rem:der-udim}.
\end{proof}

\subsection{Indecomposable modules of dimension vector $\boldsymbol{\delta}$\label{sub:inds-of-dimvec-delta}}

Let $K\in\Omega$. We list indecomposable $A^{K}$-modules with dimension
vector $\delta$ for later use.

\begin{enumerate}
\item For each $c\in K$ we define a $A^{K}$-module $W_{c}(K)$ as follows.
Let $W_{c}(K)(x)=K$ for all $x\in Q_{0}$; and\[
W_{c}(K)(\alpha_{ij})=\left\{ \begin{array}{rl}
c\id & \textrm{if $(i,j)=(2,1)$;}\\
-(1+c)\id & \textrm{if $(i,j)=(3,1)$; and}\\
\id & \textrm{otherwise.}\end{array}\right.\]

\item For each arrow $\alpha=\alpha_{ij}\in Q_{1}$ we define a $A^{K}$-module
$X_{\alpha}(K)=X_{ij}(K)$ as follows. Let $X_{\alpha}(K)(x)=X_{ij}(K)(x)=k$
for all $x\in Q_{0}$; and\[
X_{ij}(K)(\alpha_{st})=\left\{ \begin{array}{rl}
0 & \textrm{if $(s,t)=(i,j)$};\\
-\id & \textrm{if $(i,s,t)\in\{(1,3,1),(2,3,1),(3,2,1)\}$};and\\
\id & \textrm{otherwise.}\end{array}\right.\]

\end{enumerate}
Note that $W_{0}(K)=X_{21}(K)$ and that when $A$ is not of type
$A_{n}$, we have $W_{-1}(K)=X_{31}(K)$. For $K=k$ we simply write
$W_{c}=W_{c}(k)$ and $X_{ij}=X_{ij}(k)$ for all $c\in k$ and $\alpha_{ij}\in Q_{1}$.
Then clearly we have $X_{ij}(K)\cong X_{ij}^{K}$ for all $\alpha_{ij}\in Q_{1}$
and $K\in\Omega$. The following is well-known (\cite{Ri84}).

\begin{prop}
\label{ind-delta}The set $\{ W_{c}(K),X_{ij}^{K}\mid c\in K\setminus E_{\Delta},\alpha_{ij}\in Q_{1}\}$
forms a complete set of representatives of isoclasses of indecomposable
$A^{K}$-modules with dimension vector $\delta$, where \begin{equation}
E_{\Delta}:=\left\{ \begin{array}{ll}
\{0\} & \;\textrm{if $A$ is of type $A_{n}$; and}\\
\{0,-1\} & \;\textrm{otherwise.}\end{array}\right.\label{eq:zyogai}\end{equation}

\end{prop}

\subsection{The Auslander-Reiten quiver}

Recall that the set of isoclasses of simple regular representations
of the Kronecker algebra $\left(\begin{smallmatrix}k & 0\\ k^2 & k\end{smallmatrix}\right)$
over $k$ is identified with the projective line $\mathbb{P}^{1}(k)=\Proj k[x_{0},x_{1}]$
of the ring $k$, which is needed to apply general results in \cite{Ri90-can}.
We obtain the following by \cite[Theorem 4.3]{Ri84}, \cite[Theorem 1]{Ri90-can},
and \cite[Theorem 3.7]{Ri84} (see \cite{Ri84} for definitions of
orbit quivers, tubular families and so on):

\begin{thm}
\label{thm:AR-quiver}Let $A$ be a domestic canonical algebra. Then

\emph{(1)} $\Gamma_{A}$ consists of a unique preprojective component
$\mathcal{P}$ with orbit quiver of type $\Delta^{(1)}$ containing
all projective indecomposables, a unique preinjective component $\mathcal{I}$
with orbit quiver of type $\Delta^{(1)}$ containing all injective
indecomposables and a stable separating tubular $\mathbb{P}^{1}(k)$-family
$\mathcal{R}=(\mathcal{T}_{c})_{c\in\mathbb{P}^{1}(k)}$ of type $(p(1),\ldots,p(r))$
separating $\mathcal{P}$ from $\mathcal{I}$ \emph{(see Definition
\ref{dfn:separating} for definition);}

\emph{(2)} An indecomposable $A$-module $M$ is preprojective, i.e.,
$M\in\mathcal{P}$ $($resp. preinjective, i.e., $M\in\mathcal{I}$$)$
if and only if $\rank M>0$ $($resp. $<0$$)$, if and only if all
maps $M(\alpha)$, $\alpha\in Q_{1}$ are monomorphisms $($resp.
epimorphisms$)$ and there is some non-isomorphism among them; and
$M$ is regular, i.e., $M\in\mathcal{R}$ if and only if $\rank M=0$,
if and only if either all maps $M(\alpha)$, $\alpha\in Q_{1}$ are
isomorphisms, or there is some non-monomorphism and some non-epimorphism
among them; and

\emph{(3)} $\mod A$ is \emph{controlled} by $\chi_{A}$. Namely,

\emph{$\quad$(a)} $\{\chi_{A}(\udim X)|X\in\ind A\}=\{0,1\}$;

\emph{$\quad$(b)} for any positive root $v$ of $\chi_{A}$ $|\{[X]\in[\ind A]|v=\udim X\}|=1$;
and

\emph{$\quad$(c)} for any positive radical vector $v$ of $\chi_{A}$
$|\{[X]\in[\ind A]|v=\udim X\}|\ge|k|+1$.\emph{\hfil\qed}
\end{thm}
More detailed account on the tubular $\mathbb{P}^{1}(k)$-family $\mathcal{R}$
will be given in Sect.$\:$\ref{sub:tubular-family} below. 

\begin{defn}
\label{dfn:separating}(1) For each positive root $v$ of $\chi_{A}$
we denote by $m(v)$ the unique element of $\{[X]\in[\ind A]\mid v=\udim X\}$
and choose an indecomposable $A$-module $M(v)\in m(v)$. For each
$K\in\Omega$ we set $m(v)^{K}:=[M(v)^{K}]$.

(2) We here recall the definition for $\mathcal{R}$ to be \emph{separating}
$\mathcal{P}$ from $\mathcal{I}$. First for a translation subquiver
$\mathcal{T}$ of $\Gamma_{A}$ we denote by $\langle\mathcal{T}\rangle$
the full subcategory of $\mod A$ consisting of the modules in $\mathcal{T}$
(sometimes we simply write $\mathcal{T}$ for $\langle\mathcal{T}\rangle$
if there seems to be no confusion). Then $\langle\mathcal{T}\rangle$
is said to be \emph{standard} if it is isomorphic to the \emph{mesh
category} $k(\mathcal{T})$ of $\mathcal{T}$ (\cite[p. 51]{Ri84}).
Now $\mathcal{R}$ is said to be \emph{separating} $\mathcal{P}$
from $\mathcal{I}$ if

$\quad$(a) $\langle\mathcal{R}\rangle$ is \emph{standard} (thus
there are no nonzero morphisms between distinct tubes, and $\langle\mathcal{T}_{c}\rangle\cong k(\mathcal{T}_{c})$
for all $c\in\mathbb{P}^{1}(k)$);

$\quad$(b) $\Hom_{A}(\mathcal{I},\mathcal{P})=\Hom_{A}(\mathcal{I},\mathcal{T})=\Hom_{A}(\mathcal{T},\mathcal{P})=0$;
and

$\quad$(c) For each $f\in\Hom_{A}(\mathcal{P},\mathcal{I})$ and
each $c\in\mathbb{P}^{1}(k)$, $f$ can be factored through $\mathcal{T}_{c}$.
\end{defn}
\begin{cor}
\label{sub:min-dim}Let $M$ be an indecomposable $A$-module. Then
\end{cor}
\[
\min_{x\in Q_{0}}\dim M(x)=\left\{ \begin{array}{ll}
\dim M(\infty) & \textrm{if $M$ is preprojective;}\\
\dim M(1) & \textrm{if $M$ is preinjective.}\end{array}\right.\]

\begin{proof}
This is immediate from Theorem \ref{thm:AR-quiver}(2).
\end{proof}
\begin{cor}
\label{sub:delta-strings}Let $v$ be a root of $\chi_{A}$ and $t\in\mathbb{Z}$.
Then $v+t\delta$ is a root of $\chi_{A}$. In particular, $|\rank X|\le6$
for all $X\in\ind A$.
\end{cor}
\begin{proof}
By the formula (\ref{eq:B-delta-ex}) we have $\chi_{A}(v+t\delta)=\chi_{A}(v)+tB(v,\delta)+tB(\delta,v)+t^{2}\chi_{A}(\delta)=1+t\rank v-t\rank v+0=1$.
Thus $v+t\delta$ is a root of $\chi_{A}$. Now let $X\in\ind A$.
If $X$ is regular, then the assertion is trivial because $\rank X=0$.
If $X$ is preprojective, then $\rank X>0$ and $\udim X$ is a positive
root, and hence $w:=\udim X-\dim X(\infty)\delta$ is a positive root
by Corollary \ref{sub:min-dim}. Thus there exists some $Y\in\ind A$
such that $\udim Y=w$. Then $\rank X=\rank Y=\dim Y(1)$ because
$\dim Y(\infty)=0$ by construction. Here $Y$ is regarded as an indecomposable
module over $\supp Y$ that is a representation-finite hereditary
algebra defined by a quiver. Hence $\dim Y(1)\le6$ by Gabriel's Theorem
\cite{Ga72} on the classification of representation-finite quivers
(or Ovsienko's Theorem \cite{Ov} explained in \cite[1.0 Theorem 1]{Ri84}).
If $X$ is preinjective, then the similar argument works to have $-6\le\rank X<0$.
\end{proof}

\subsection{Tubular family\label{sub:tubular-family}}

We describe the tubular $\mathbb{P}^{1}(k)$-family $\mathcal{R}=(\mathcal{T}_{c})_{c\in\mathbb{P}^{1}(k)}$
in Theorem \ref{thm:AR-quiver} in more detail following \cite{Ri84,Ri90-can}.
Recall first that as a set of points, $\mathbb{P}^{1}(k)$ decomposes
into a disjoint union $\mathbb{P}^{1}(k)=\bigsqcup_{d\in\mathbb{N}}\mathbb{P}^{1}(k)_{d}$
of the subsets \[
\mathbb{P}^{1}(k)_{d}:=\{\langle p\rangle\in\mathbb{P}^{1}(k)|p\in k[x_{0},x_{1}]\textrm{ is homogeneous, irreducible, and }\deg p=d\},\]
where $d\in\mathbb{N}$. In \cite{Asa02-2}, to parameterize indecomposable
modules with dimension vector the minimal positive imaginary root
of a simple Lie algebra considered there, we used the set $\mathbb{P}_{k}^{1}:=(k\times k\setminus\{(0,0)\})/\sim$,
where for each $(a,b),(a',b')\in k\times k\setminus\{(0,0)\}$ we
define $(a,b)\sim(a',b')$ if and only if $(a,b)=t(a',b')$ for some
$t\in k^{\times}$, which is an equivalence relation on $k\times k\setminus\{(0,0)\}$.
We here identify $\mathbb{P}_{k}^{1}$ with the subset $\mathbb{P}^{1}(k)_{1}$
of $\mathbb{P}^{1}(k)$ by the bijection $(a:b)\mapsto(ax_{0}+bx_{1})$,
where $(a:b)$ denotes the equivalence class in $\mathbb{P}_{k}^{1}$
containing $(a,b)\in k\times k\setminus\{(0,0)\}$. We also identify
$\mathbb{P}_{k}^{1}$ with the set $k\cup\{\infty\}$ by the bijection
$(a:1)\mapsto a$ for $a\in k$ and $(1:0)\mapsto\infty$. For each
$c\in\mathbb{P}^{1}(k)$, $\mathcal{T}_{c}$ has the following shape:
If $c\not\in E_{\Delta}\cup\{\infty\}$ (see (\ref{eq:zyogai})),
then $\mathcal{T}_{c}$ is a \emph{homogeneous tube}, i.e., is isomorphic
to the translation quiver $\mathbb{Z}A_{\infty}/\langle\tau\rangle$
(see \cite[Chap. 3]{Ri84}, it is denoted by $\mathbb{Z}A_{\infty}/1$
there). The module $W_{c}$ defined in Sect.$\:$\ref{sub:inds-of-dimvec-delta}
is the unique module on the \emph{mouth} of $\mathcal{T}_{c}$ (\cite[3.1]{Ri84}).
Each module in $\mathcal{T}_{c}$ is uniquely determined by $W_{c}$
and by its \emph{quasi-length} (= the number of modules in the shortest
path from the mouth to it) $m$, and thus we denote it by $W_{c}[m]$.
(Since $W_{c}$ is of quasi-length 1, we can write $W_{c}=W_{c}[1]$.)
The set of modules in $\mathcal{T}_{c}$ is equal to $\{ W_{c}[m]|m\in\mathbb{N}\}$.
Then we have \[
\udim W_{c}[m]=md\delta,\]
if $c\in\mathbb{P}^{1}(k)_{d}$ with $d\in\mathbb{N}$. Next assume
that $c\in E_{\Delta}\cup\{\infty\}$, which depends on the value
of $r\in\{2,3\}$ in Figure \ref{quiver-dom-can}. Set $c(1)=\infty$,
$c(2)=0$ (and $c(3):=-1$ when $r=3$). Then for $i=1,\ldots,r$,
$\mathcal{T}_{c(i)}$ is a \emph{stable tube} of \emph{rank} $p(i)$,
namely, it is isomorphic to the translation quiver $\mathbb{Z}A_{\infty}/\langle\tau^{p(i)}\rangle$
(= $\mathbb{Z}A_{\infty}/p(i)$ in \cite[Chap. 3]{Ri84}). The simple
modules $S_{x_{i2}},\ldots,S_{x_{ip(i)}}$, and the module $W'_{c(i)}$
are the modules on the mouth of $\mathcal{T}_{c(i)}$, where $W'_{c(i)}:=M(\delta-\sum_{j=2}^{p(i)}e_{x_{ij}})$
(see Definition \ref{dfn:separating} for the notation), which is
possible because a direct calculation shows that $\delta-\sum_{j=2}^{p(i)}e_{x_{ij}}$
is a positive root of $\chi_{A}$. Each module in $\mathcal{T}_{c(i)}$
is uniquely determined by its quasi-length $m$ and the starting point
$W\in\{ S_{x_{i2}},\ldots,S_{x_{ip(i)}},W'_{c(i)}\}$ of the shortest
path from the mouth to it. Therefore we denote it by $W[m]$. The
set of modules in $\mathcal{T}_{c(i)}$ is equal to $\{ W[m]|W\in\{ S_{x_{i2}},\ldots,S_{x_{ip(i)}},W'_{c(i)}\},m\in\mathbb{N}\}$.
The modules $X_{ij}$, $j\in\{1,\ldots,p(i)\}$ defined in Sect.$\:$
\ref{sub:inds-of-dimvec-delta} are in $\mathcal{T}_{c(i)}$ and of
quasi-length $p(i)$, thus $X_{ij}=W[p(i)]$ for some $W\in\{ S_{x_{i2}},\ldots,S_{x_{ip(i)}},W'_{c(i)}\}$.
By the additivity of dimension vectors on exact sequences, we easily
see that \[
\udim W[dp(i)]=d\delta\]
 for all $d\in\mathbb{N}$, and $\udim W[m]\in\mathbb{Z}\delta$ if
and only if $m\in\mathbb{Z}p(i)$. In particular, we have $\{ c\in\mathbb{P}^{1}(k)|\mathcal{T}_{c}\textrm{ contains a module of dimension vector }\delta\}=\mathbb{P}^{1}(k)_{1}=\mathbb{P}_{k}^{1}$.
We call $\mathcal{T}_{c(i)}$, $i\in\{1,\ldots,r\}$ \emph{non-homogeneous}
tubes.

\subsection{$\mathbf{\tau}$-orbits in the preprojective component}

Set $\Phi:=-C^{-T}C$ to be the \emph{Coxeter} \emph{matrix} of $A$.
For later use we give an explicit form of $\Phi^{-1}$ when $A$ is
not of type $A_{n}$:

\[ \Phi^{-1}=\left(\begin{array}{c|cccc|cccc|c|c} 
0 & 0 & -1 & \cdots & -1 & 0 & -1 & \cdots & -1 & 0 & -1\\ 
\hline
0 & 0 & 1 & 0 & \cdots & 0 & 0 & \cdots & 0 & 0 & 0\\ 
\vdots & \vdots & \ddots & \ddots & \ddots & \vdots & \vdots & \ddots & \vdots & \vdots & \vdots\\ 
0 & 0 & \cdots & 0 & 1 & 0 & 0 & \cdots & 0 & 0 & 0\\ 
1 & 0 & \cdots & 0 & 0 & 1 & 1 & \cdots & 1 & 1 & 1\\ 
\hline

0 & 0 & \cdots & 0 & 0 & 0 & 1 & 0 & \cdots & 0 & 0\\ 
\vdots & \vdots & \ddots & \vdots & \vdots & \vdots & \ddots & \ddots & \ddots & \vdots & \vdots\\ 
0 & 0 & \cdots & 0 & 0 & 0 & \cdots & 0 & 1 & 0 & 0\\ 
1 & 1 & \cdots & 1 & 1 & 0 & \cdots & 0 & 0 & 1 & 1\\ 
\hline

1 & 1 & \cdots & 1 & 1 & 1 & \cdots & 1 & 1 & 0 & 1\\ 
\hline
-2 & -1 & \cdots & -1 & -1 & -1 & \cdots & -1 & -1 & -1 & -1\end{array}\right),\]

\begin{comment}
\[
\Phi^{-1}=\left(\begin{array}{ccccccccccc}
0 & 0 & -1 & \cdots & -1 & 0 & -1 & \cdots & -1 & 0 & -1\\
0 & 0 & 1 & 0 & \cdots & 0 & 0 & \cdots & 0 & 0 & 0\\
\vdots & \vdots & \ddots & \ddots & \ddots & \vdots & \vdots & \ddots & \vdots & \vdots & \vdots\\
0 & 0 & \cdots & 0 & 1 & 0 & 0 & \cdots & 0 & 0 & 0\\
1 & 0 & \cdots & 0 & 0 & 1 & 1 & \cdots & 1 & 1 & 1\\
0 & 0 & \cdots & 0 & 0 & 0 & 1 & 0 & \cdots & 0 & 0\\
\vdots & \vdots & \ddots & \vdots & \vdots & \vdots & \ddots & \ddots & \ddots & \vdots & \vdots\\
0 & 0 & \cdots & 0 & 0 & 0 & \cdots & 0 & 1 & 0 & 0\\
1 & 1 & \cdots & 1 & 1 & 0 & \cdots & 0 & 0 & 1 & 1\\
1 & 1 & \cdots & 1 & 1 & 1 & \cdots & 1 & 1 & 0 & 1\\
-2 & -1 & \cdots & -1 & -1 & -1 & \cdots & -1 & -1 & -1 & -1\end{array}\right),\]

\end{comment}
\noindent where the rows and columns are ordered by the sequence
$(1,x_{12}$, $\ldots$, $x_{1p(1)}$,$x_{22}$, $\ldots$, $x_{2p(2)}$,
$x_{32}$, $\infty$), and in the first row the two zeros between
entries with value $-1$ correspond to $x_{22}$ and $x_{32}$ , whereas
in the first column the 1 between zeros corresponds to $x_{1p(1)}$.
In many cases $\Phi^{-1}$ can be seen as a {}``shadow'' of $\tau^{-1}$
in $K_{0}(A)$ as the following statement shows (see \cite[2.4 (4*)]{Ri84}
for the proof).

\begin{lem}
\label{lem:Coxeter}Let \emph{}$M$ \emph{}be an \emph{$A$}-module.
If \emph{}$\injdim M\leq1$ \emph{}and \emph{}$\Hom_{A}(D(A_{A}),M)=0$,
then
\end{lem}
\[
\udim\tau^{-1}M=(\udim M)\Phi^{-1}.\]
\hfil\qed

In order to check that the injective dimension is at most 1 we cite
the following lemma from \cite{Ri84}.

\begin{lem}
\label{lem:injdim-atmost1}Let $M$ be an $A$-module. Then $\injdim M\le1$
if and only if $\Hom_{A}(\tau^{-1}M,A)=0$. In particular $\injdim M\le1$
holds if $M$is an indecomposable module such that $\tau^{-1}M$ is
not a predecessor of any projective indecomposable $A$-module.\hfil\qed
\end{lem}
Note that if $M$ is an indecomposable module that \emph{}is a successor
of a complete slice of the preprojective component, then $\tau^{-1}M$
cannot be a predecessor of any projective indecomposable $A$-module,
and hence $\injdim M\leq1$.

Direct calculation shows the following.

\begin{lem}
\emph{\label{lem:preservedby-Phi}$\delta\Phi^{-1}=\delta$ and $\Phi^{-1}\rho^{T}=\rho^{T}$.\hfil\qed}
\end{lem}
On the set of dimension vectors of indecomposable preprojective $A$-modules
there are two natural partitions: the $\tau^{-1}$-orbit decomposition
and the coset decomposition modulo $\delta$. The following gives
a relationship between them, which was obtained in answering a question
by A.$\,$Hubery.

\begin{prop}
\label{pro:tau-delta}If $M$ is an indecomposable preprojective $A$-module
such that $\tau^{-1}M$ is not a predecessor of any projectives in
$\Gamma_{A}$, then there exist $t,m\in\mathbb{N}$ such that\[
\udim\tau^{-t}M=\udim M+m\delta.\]

\end{prop}
\begin{proof}
Let $\mathcal{P}$ be the set of vertices of the preprojective component
of $\Gamma_{A}$. For each $r\in\{1,2,\ldots,6\}$ set

\begin{eqnarray*}
\mathcal{P}_{r} & := & \{ X\in\mathcal{P}\mid\tau^{-1}X\textrm{ is not a predecessor of any projectives},\rank X=r\}\\
\udim\mathcal{P}_{r} & := & \{\udim X\mid X\in\mathcal{P}_{r}\}.\end{eqnarray*}
 Define an equivalence relation $\sim$ on $\udim\mathcal{P}_{r}$
by $v\sim w$ if and only if $v-w\in\mathbb{Z}\delta$ for all $v,w\in\udim\mathcal{P}_{r}$.
Since $\udim X-\dim X(\infty)\delta$ is a root of $\chi_{A}$ for
each $X\in\mathcal{P}_{r}$, there exists an indecomposable $A^{l}$-module
$Y$ such that $\udim X-\dim X(\infty)\delta=\udim Y$. This shows
that the quotient set $(\udim\mathcal{P}_{r})/\sim$ is finite because
$A^{l}$ is representation-finite. We show that $\Phi^{-1}$ acts
on the finite set $(\udim\mathcal{P}_{r})/\sim$. Let $X\in\mathcal{P}_{r}$.
Then clearly $\tau^{-2}X$ is not a predecessor of any projectives,
either, and $\rank\tau^{-1}X=(\udim X)\Phi^{-1}\rho^{T}=\rank X=r$
by Lemmas \ref{lem:Coxeter} and \ref{lem:preservedby-Phi}. Hence
$\tau^{-1}$ induces an injective map $\mathcal{P}_{r}\to\mathcal{P}_{r}$.
Thus by Lemma \ref{lem:Coxeter} the right multiplication by $\Phi^{-1}$
induces an injective map $\udim\mathcal{P}_{r}\to\udim\mathcal{P}_{r}$,
and by Lemma \ref{lem:preservedby-Phi} it also induces an injective
map $(\udim\mathcal{P}_{r})/\sim\to(\udim\mathcal{P}_{r})/\sim$ ,
which is a bijection because the set $(\udim\mathcal{P}_{r})/\sim$
is finite. Now let $M$ be as in the assertion and put $r:=\rank M$.
Then $\udim M\in\udim\mathcal{P}_{r}$ and $(\udim M)\Phi^{-t}\sim\udim M$
for some $t\in\mathbb{N}$, which means that $\udim\tau^{-t}M=\udim M+m\delta$
for some $m\in\mathbb{Z}$. If $m=0$, then $\udim\tau^{-t}M=\udim M$,
and we have $\tau^{-t}M\cong M$, a contradiction. Thus $m\ne0$.
If $m<0$, then there exists some $s\in\mathbb{N}$ such that $\udim\tau^{-st}M=\udim M+sm\delta<0$,
a contradiction. Hence $m\in\mathbb{N}$.
\end{proof}
\begin{rem}
If the assumption that $\tau^{-1}M$ is not a predecessor of any projectives
in $\Gamma_{A}$ is dropped or is replaced with the weaker condition
that $\injdim M\le1$, then there is a counter-example, e.g., in the
case where $A$ is of type $E_{6}$ and $M=\tau^{-1}(A\mathbf{e}_{x_{32}})$.
There is also an example for which the smallest value of $m$ is not
equal to 1 (see Sect.$\:$\ref{sub:exm}).
\end{rem}
Recall that an $A$-module $M$ is called \emph{sincere} if $\mathbf{e}_{x}M\ne0$
for all $x\in Q_{0}$. We say that $M$ is \emph{non-sincere} if it
is not sincere. The following well-known fact follows also from the
proposition above (see \cite{HV}, \cite{Bo84} for general results).

\begin{cor}
\label{cor:min-rep-inf}$A$ is a minimal representation-infinite
algebra, i.e., $A/A\mathbf{e}_{x}A$ is repre\-senta\-tion-finite
for all $x\in Q_{0}$.
\end{cor}
\begin{proof}
It is enough to show that the Auslander-Reiten quiver of $A$ contains
only a finite number of non-sincere indecomposable $A$-modules. By
the previous proposition it is immediate that the preprojective component
contains only a finite number of non-sincere indecomposable $A$-modules.
Dually the preinjective component has the same property. In the tubular
family any non-sincere indecomposable $A$-module $X$ lies in a non-homogeneous
tube $\mathcal{T}$ and has quasi-length less than the rank of $\mathcal{T}$.
Thus also there are only a finite number of non-sincere indecomposables
in the tubular family.
\end{proof}
\begin{prop}
\label{lem:non-sincere}If an indecomposable $A$-module $M$ is not
sincere, then $\mathbf{u}_{[M]}\in L(A)_{1}$.
\end{prop}
\begin{proof}
Let $M$ be a non-sincere indecomposable $A$-module, and set $B:=\supp M$.
Then by Corollary \ref{cor:min-rep-inf}, $B$ is representation-finite.
We regard $\mod B$ as a full subcategory of $\mod A$ by the canonical
embedding. Then $M$ is regarded as a sincere indecomposable $B$-module,
and by the formula (\ref{eq:Hall-coefficient}) we see that $L(B)_{1}\subseteq L(A)_{1}$.
Therefore it is enough to show that $\mathbf{u}_{[M]}\in L(B)_{1}$.
If $B$ is hereditary, then this follows by Lemma \ref{lem:rep-fin-hered}.
If $M$ is a preprojective (resp.~preinjective) $A$-module, then
$M(\infty)=0$ (resp.~$M(1)=0$) by Corollary \ref{sub:min-dim}
because $M$ is not a sincere $A$-module, and then $B$ turns out
to be hereditary. Thus in this case the assertion holds. Hence we
may assume that $M$ is a regular $A$-module. Since $M$ is not a
sincere $A$-module, $M$ is in a non-homogeneous tube of a rank $p>1$
and with quasi-length less than $p$. Thus, in particular, $\udim M$
consists of 0 and 1. We show that $\mathbf{u}_{[M]}\in L(A)_{1}$
by induction on $\dim M$. If $\dim M=1$, then $M$ is simple and
the assertion is trivial by definition. Now assume $\dim M>1$. Then
the form of $\udim M$ shows that there exists an exact sequence of
the form\[
0\to N\to M\to S\to0\textrm{ \: or\: }0\to S\to M\to N\to0\]
with $S$ a simple $A$-module and $N$ an indecomposable $A$-module,
and that\[
(F_{S^{K}N^{K}}^{M^{K}},F_{N^{K}S^{K}}^{M^{K}})=(1,0)\;\textrm{or}\;(0,1),\]
 respectively, for all $K\in\Omega_{A}$. Thus we have $\mathbf{u}_{[M]}=\pm[\mathbf{u}_{[S]},\mathbf{u}_{[N]}]\in L(A)_{1}$
because by induction hypothesis both $\mathbf{u}_{[S]}$ and $\mathbf{u}_{[N]}$
are in $L(A)_{1}$.
\end{proof}

\section{\textbf{Realization of simple Lie algebras}}

In this section we state our main theorems realizing simple Lie algebras
and their root spaces, and give a precise form of Chevalley generators
of our realization.

\subsection{Main results}

\begin{defn}
When $\Delta\ne A_{1}$ (resp. $\Delta=A_{1}$) we set $I(A)$ to
be the ideal of $L(A)_{1}$ (resp. of $L(A)_{1}^{\mathbb{Z}[2^{-1}]}$)
generated by the set\[
\{\mathbf{u}_{m(e_{x}+\delta)}-\mathbf{u}_{m(e_{x})}\mid x\in Q_{0}\}.\]
 For each $u\in L(A)_{1}^{\mathbb{C}}$ we denote by $\overline{u}$
the coset of $u$ in $L(A)_{1}^{\mathbb{C}}/I(A)^{\mathbb{C}}$.
\end{defn}
\begin{rem}
By Lemmas \ref{lem:rg-prpl}, \ref{lem:end-prpl}, and \ref{lem:end-prpl-A1}
that will be proved in the next section we see that $\mathbf{u}_{m(e_{x}+\delta)}\in L(A)_{1}$
(resp. $\mathbf{u}_{m(e_{x}+\delta)}\in L(A)_{1}^{\mathbb{Z}[2^{-1}]}$)
for all $x\in Q_{0}$ if $\Delta\ne A_{1}$ (resp. $\Delta=A_{1}$). 
\end{rem}
\begin{notation}
\label{ntn:Chevalley-gen}(1) For each $x\in Q_{0}$, the vector $\delta-e_{x}$
is a root of $\chi_{A}$, which enables us to consider the indecomposable
$A$-module $T_{x}:=M(\delta-e_{x})$. By Lemma \ref{lem:non-sincere}
we have $\mathbf{u}_{[T_{x}]}\in L(A)_{1}$.

(2) For each $x\in\Delta_{0}$, we set\[
\varepsilon_{x}:=\mathbf{\bar{u}}_{[S_{x}]},\zeta_{x}:=\begin{cases}
-\mathbf{\bar{u}}_{T_{1}} & \textrm{if $x=1$},\\
\mathbf{\bar{u}}_{T_{x}} & \textrm{otherwise,}\end{cases}\textrm{and }\eta_{x}:=[\varepsilon_{x},\zeta_{x}].\]
Note that all of these are in $L(A)_{1}^{\mathbb{C}}/I(A)^{\mathbb{C}}$.

(3) Let $\{ E_{x},F_{x},H_{x}\}_{x\in\Delta_{0}}$ be Chevalley generators
of $\mathfrak{g}(\Delta)$, where $\{ H_{x}\}_{x\in\Delta_{0}}$ is
a basis of the Cartan subalgebra $\mathfrak{h}$ of $\mathfrak{g}(\Delta)$,
$\{ E_{x}\}_{x\in\Delta_{0}}$ (resp.\ $\{ F_{x}\}_{x\in\Delta_{0}}$)
are generators of the positive (resp. negative) part $\mathfrak{n}_{+}$
(resp.\ $\mathfrak{n}_{-}$) of $\mathfrak{g}(\Delta)$. 
\end{notation}
We are now in a position to state our main theorem.

\begin{thm}
\label{thm:realization}Let $\Delta$ be a simply-laced Dynkin diagram,
$A$ a canonical algebra of type $\Delta$ and $\mathfrak{g}(\Delta)$
the complex simple Lie algebra of type $\Delta$\emph{.} Further let
$\{ E_{x},F_{x},H_{x}\}_{x\in\Delta_{0}}\subseteq\mathfrak{g}(\Delta)$
and $\{\varepsilon_{x},\zeta_{x},\eta_{x}\}_{x\in\Delta_{0}}\subseteq L(A)_{1}^{\mathbb{C}}/I(A)^{\mathbb{C}}$
\emph{}be as in Notation \emph{\ref{ntn:Chevalley-gen}}. Then there
is an isomorphism\[
\phi\colon\mathfrak{g}(\Delta)\isoto L(A)_{1}^{\mathbb{C}}/I(A)^{\mathbb{C}}\]
such that $\phi(E_{x})=\varepsilon_{x}$, $\phi(F_{x})=\zeta_{x}$
and $\phi(H_{x})=\eta_{x}$ for all $x\in\Delta_{0}$.
\end{thm}
The proof is given in Sect.$\:$\ref{sec:Proof-of-Theorem}.

\begin{defn}
\label{def:degree}(1) Let $v\in K_{0}(A)$. Then we set $\deg v:=v-v_{\infty}\delta$,
which we regard as an element of $K_{0}(kQ^{l})$ because $(\deg v)_{\infty}=0$.
This defines a linear map $\deg\colon K_{0}(A)\to K_{0}(kQ^{l})$,
$v\mapsto\deg v$. Obviously this is surjective and $\Ker\deg=\mathbb{Z}\delta$.
For an indecomposable $A$-module $M$, we set $\deg M:=\deg(\udim M)$
and call it the \emph{degree} of $M$.

(2) Let $v$ be a positive root of $\chi_{A}$, and $v_{i}$ ($i\in Q_{0}$)
the smallest entry of $v$. Then we define $\org v:=v-v_{i}\delta$.
Since $v\not\in\mathbb{Z}\delta$, we have $\org v>0$. Note that
$\org v$ is also a positive root of $\chi_{A}$ with $M(\org v)$
non-sincere by Corollary \ref{sub:delta-strings} and that $\org v=\deg v$
if $\rank v>0$.

(3) We denote by $\Rt(\chi_{A})$ (resp. $\Rt_{+}(\chi_{A})$) the
set of all roots (resp. positive roots) of $\chi_{A}$, and by $\Rt(\mathfrak{g}(\Delta))$
(resp. $\Rt_{+}(\mathfrak{g}(\Delta))$) the set of all roots (resp.
all positive roots) of $\mathfrak{g}(\Delta)$, and set $\Rt_{-}(\blank):=-\Rt_{+}(\blank)$.
Then $\Rt_{+}(\chi_{kQ^{l}})$ is regarded as $\Rt_{+}(\mathfrak{g}(\Delta))$
by identifying $e_{x}$ with $E_{x}$ for all $x\in Q_{0}^{l}=\Delta_{0}$
by Gabriel's Theorem. Thus we have \[
\Rt(\mathfrak{g}(\Delta))=\Rt_{+}(\chi_{kQ^{l}})\cup(\Rt_{-}(\chi_{kQ^{l}}))=\Rt(\chi_{kQ^{l}})\subseteq K_{0}(kQ^{l}).\]

\end{defn}
\begin{lem}
\label{lem:deg-root}The map $\deg$ induces a surjective linear map
$\Rt_{+}(\chi_{A})\to\Rt(\mathfrak{g}(\Delta))$. In particular, for
each indecomposable $A$-module $M$ with $\udim M\not\in\mathbb{Z}\delta$,
$\deg M$ is a root of $\mathfrak{g}(\Delta)$.
\end{lem}
\begin{proof}
For each $v\in\Rt_{+}(\chi_{A})$ we have $\deg v\in\Rt(\chi_{A})$
by Corollary \ref{sub:delta-strings}, and hence $\deg v\in\Rt(\chi_{kQ^{l}})$=$\Rt(\mathfrak{g}(\Delta))$.
Therefore $\deg(\Rt_{+}(\chi_{A}))\subseteq\Rt(\mathfrak{g}(\Delta))$.
Conversely, for each $w\in\Rt(\mathfrak{g}(\Delta))=\Rt(\chi_{kQ^{l}})\subseteq\Rt(\chi_{A})$
there exists some $t\in\mathbb{N}$ such that $v:=w+t\delta>0$. Then
again by Corollary \ref{sub:delta-strings} $v\in\Rt(\chi_{A})$,
and hence $v\in\Rt_{+}(\chi_{A})$. Here it is obvious that $\deg v=w$.
Therefore we have $\Rt(\mathfrak{g}(\Delta))\subseteq\deg(\Rt_{+}(\chi_{A}))$,
and hence $\deg(\Rt_{+}(\chi_{A}))=\Rt(\mathfrak{g}(\Delta))$.
\end{proof}
\begin{rem}
\label{rem:pos-neg}In the above, set $M:=M(v)$. Then
\begin{enumerate}
\item If $M$ is preprojective, then $\deg v>0$ by Corollary \ref{sub:min-dim}.
\item If $M$ is preinjective, then $\deg v<0$ by Corollary \ref{sub:min-dim}.
\item If $M$ is regular, then $\deg v=\pm\sum_{j=s}^{t}e_{x_{ij}}$ for
some $i\in\{1,\ldots,r\}$ and some $s,t$ with $2\le s\le t\le p(i)$.
\end{enumerate}
\end{rem}
For a root $\alpha$ of $\mathfrak{g}(\Delta)$ we denote by $\mathfrak{g}(\Delta)_{\alpha}$
the root space of $\mathfrak{g}(\Delta)$ with root $\alpha$.

\begin{prop}
\label{prp:root-space}Let $v$ be a positive root of \emph{$\chi_{A}$}.
Then
\begin{enumerate}
\item $0\ne\mathbf{\overline{u}}_{m(v)}\in L(A)_{1}^{\mathbb{C}}/I(A)^{\mathbb{C}}$;
\item $\phi^{-1}(\mathbf{\overline{u}}_{m(v)})\in\mathfrak{g}(\Delta)_{\deg v}$;
and
\item $\mathbb{C}\mathbf{\bar{u}}_{m(v+\delta)}=\mathbb{C}\mathbf{\bar{u}}_{m(v)}$.
\end{enumerate}
\end{prop}
The proof is given in Sect.$\:$\ref{sec:root-spaces}. This immediately
yields the following.

\begin{thm}
\label{thm:root-space}Let \emph{}$\phi$ \emph{}be as in Theorem
\emph{\ref{thm:realization}}, \emph{}$\alpha$ \emph{}a root of \emph{}$\mathfrak{g}(\Delta)$
\emph{}and \emph{}$v$ \emph{}a positive root of \emph{}$\chi_{A}$
\emph{}with \emph{}$v-\alpha\in\mathbb{Z}\delta$. Then $\alpha=\deg v$
and the restriction of \emph{}$\phi$ \emph{}induces an isomorphism
from the root space \emph{}$\mathfrak{g}(\Delta)_{\alpha}$ \emph{}to
\emph{}$\mathbb{C}\mathbf{\bar{u}}_{m(v)}$.
\end{thm}

\subsection{Basis of the Cartan subalgebra\label{sub:basis-Cartan}}

We now give the precise forms of $\eta_{x}$'s using the list in Proposition
\ref{ind-delta}. Noting that $[u_{S_{1}},-u_{T_{1}}]=[u_{T_{1}},u_{S_{1}}]=u_{T_{1}}u_{S_{1}}$we
have \[
[u_{S_{1}},-u_{T_{1}}]=\sum_{c\in k}u_{W_{c}}+u_{X_{11}}\]
 in the Lie algebra $\overline{L}(A)_{(q-1)}$. Then in $L(A)_{1}$
we have

\begin{equation}
[\mathbf{u}_{S_{1}},-\mathbf{u}_{T_{1}}]=\left(\sum_{c\in K}u_{W_{c}(K)}\right)_{K\in\Omega}+\mathbf{u}_{X_{11}}=:h_{1}\label{eq:h-1}\end{equation}
Hence

\begin{equation}
\eta_{1}=\left\{ \begin{array}{ll}
\overline{\left(\sum_{c\in K^{\times}}u_{W_{c}(K)}\right)}_{K\in\Omega}+\overline{\mathbf{u}}_{X_{11}}+\overline{\mathbf{u}}_{X_{21}} & \textrm{if $A$ is of type $A_{n}$;}\\
\\\overline{\left(\sum_{c\in K\setminus\{0,-1\}}u_{W_{c}(K)}\right)}_{K\in\Omega}+\overline{\mathbf{u}}_{X_{11}}+\overline{\mathbf{u}}_{X_{21}}+\overline{\mathbf{u}}_{X_{31}} & \textrm{otherwise.}\end{array}\right.\label{eq:et-1}\end{equation}
 For each $x_{ij}\in\Delta_{0}\setminus\{1\}$ we have \begin{equation}
[\mathbf{u}_{S_{x_{ij}}},\mathbf{u}_{T_{x_{ij}}}]=\mathbf{u}_{X_{ij}}-\mathbf{u}_{X_{i,j-1}}=:h_{x_{ij}}\label{eq:h-ij}\end{equation}
 in $L(A)_{1}$ , and hence

\begin{equation}
\eta_{x_{ij}}=\overline{h}_{x_{ij}}=\overline{\mathbf{u}}_{X_{ij}}-\overline{\mathbf{u}}_{X_{i,j-1}}.\label{eq:et-ij}\end{equation}

\section{\textbf{Preparations of proof of Main Theorem}}

\subsection{Non-homogeneous tubes}

Throughout this section $\mathcal{T}$ is a non-homogeneous tube of
rank $p>1$ in $\Gamma_{A}$. 

We recall some fundamental facts on nilpotent representations of cyclic
quivers (\cite{Asa02}). Consider the cyclic quiver with $p$ vertices

$$
R:=\xymatrix@1{
1 \ar[r]^{\al} & 2 \ar[r]^{\al} &\cdots
\ar[r]^{\al} &n\ar[r]^{\al}& p,
\ar@<0.5ex>@/^1pc/[llll]^{\al}}$$and set $\Lambda$ to be the path-algebra $kR$, which is infinite-dimensional
over $k$. Let $J$ be the ideal of $\Lambda$ generated by all arrows.
A finite-dimensional $\Lambda$-module $M$ is called \emph{nilpotent}
if $J^{m}M=0$ for some $m\in\mathbb{N}$. Denote by $\mod_{0}\Lambda$
(resp. $\ind_{0}\Lambda$) the full subcategory of $\mod\Lambda$
(resp. $\ind\Lambda$) consisting of nilpotent $\Lambda$-modules.
Then $\mod_{0}\Lambda$ has almost split sequences, its Auslander-Reiten
quiver is isomorphic to $\mathcal{T}$, and $\ind_{0}\Lambda$ is
standard, i.e. $\ind_{0}\Lambda\cong k(\mathcal{T})$ (\cite[3.6(6)]{Ri84}).
On the other hand $\langle\mathcal{T}\rangle$ is also standard by
Theorem \ref{thm:AR-quiver}. Therefore we have \begin{equation}
\langle\mathcal{T}\rangle\cong\ind_{0}\Lambda.\label{eq:tube-cyclic-qvalg}\end{equation}
Of course, if $X\in\mathcal{T}$ is corresponding to $Y\in\ind_{0}\Lambda$
under this isomorphism, then the quasi-length of $X$ is equal to
$l(Y)$. In particular, the modules on the mouth of $\mathcal{T}$
are corresponding to simple modules in $\ind_{0}\Lambda$.

\begin{rem}
\label{rmk:shape-of-T}The shape of $\mathcal{T}$ does not depend
on $k$. In particular, $\ind_{0}\Lambda^{K}=\{ X^{K}|X\in\ind_{0}\Lambda\}$
as sets for all $K\in\Omega$.
\end{rem}
\begin{lem}
For each $M\in\ind_{0}\Lambda$, we have $\Hom_{\Lambda}(M,\tau M)=0$
if and only if $l(M)<p$.
\end{lem}
\begin{proof}
Since $\tau M(i,j)=M(i+1,j)$ for all $(i,j)\in R_{0}\times\mathbb{N}$,
where $R_{0}$ is the set of vertices of $R$ (cf.\ e.g., \cite[Proof of Theorem 2.1]{ARS}),
the assertion is easily verified.
\end{proof}
\begin{prop}
\label{sub:Positions-of-exceptional}Let $X\in\ind A$. Then $X$
is exceptional if and only if
\begin{enumerate}
\item $X$ is preprojective;
\item $X$ is preinjective; or
\item $X$ is in a non-homogeneous tube with quasi-length less than the
rank of the tube.
\end{enumerate}
\end{prop}
\begin{proof}
If $X$ is preprojective or preinjective, then $\Hom_{A}(X,\tau X)=0$.
Hence by the Auslander-Reiten formula\[
\Ext_{A}^{1}(X,X)\cong D(\overline{\Hom}_{A}(X,\tau X))\]
 (see e.g., \cite[Proposition 2.3]{Ga831}) we see that $X$ is exceptional,
where for each $M,N\in\mod A$, $\overline{\Hom}_{A}(M,N)$ is the
factor space of $\Hom_{A}(M,N)$ by the subspace consisting of homomorphisms
factoring through an injective module. Now let $X$ be regular, then
we have $\Ext_{A}^{1}(X,X)\cong D(\Hom_{A}(X,\tau X))$ because $\Hom_{A}(\mathcal{I},\mathcal{R})=0$.
If $X$ is in a homogeneous tube, then $X\cong\tau X$ , and hence
$X$ cannot be exceptional. Assume that $X$ is in a non-homogeneous
tube of rank $p>1$. Then by the lemma above and by (\ref{eq:tube-cyclic-qvalg})
we have $\Hom_{A}(X,\tau X)=0$ if and only if the quasi-length of
$X$ is less than $p$. Hence $X$ is exceptional if and only if the
quasi-length of $X$ is less than $p$.
\end{proof}
\begin{notation}
Since $\End_{A}(X)$ is a finite-dimensional local $k$-algebra for
each $X\in\ind A$, the factor algebra $\End_{A}(X)/\rad\End_{A}(X)$
is a field, which we denote by $F_{A}(X)$.
\end{notation}
\begin{prop}
\label{sub:Mods-in-non-homog}Let $X$ be a module in the non-homogeneous
tube $\mathcal{T}$. Then $F_{A}(X)=k$, and hence $X$ is $\Omega$-indecomposable.
Thus we may write $\mathbf{u}_{[X]}=(u_{[X^{K}]})_{K\in\Omega}$.
\end{prop}
\begin{proof}
Let\[
Y=Y_{1}\to Y_{2}\to\cdots\to Y_{m}=X\]
 be the shortest path in $\mathcal{T}$ with $Y$ an indecomposable
module on the mouth of $\mathcal{T}$. Since $Y$ is an exceptional
$A$-module by Proposition \ref{sub:Positions-of-exceptional}, we
have $\End_{A}(Y)\cong k$ by Lemma \ref{sub:Endo-algebras-of}. Let
$\rad_{A}$ be the radical of the category $\mod A$ (\cite[p.\ 53]{Ri84}).
For $A$-modules $M,N$ set $\Irr(M,N):=\rad_{A}(M,N)/\rad_{A}^{2}(M,N)$
to be the $(F_{A}(N),F_{A}(M))$-bimodule of irreducible maps from
$M$ to $N$ (\cite[p.\ 55]{Ri84}). Then it follows from the shape
of $\mathcal{T}$ that\[
\dim_{F_{A}(Y_{i+1})}\Irr(Y_{i},Y_{i+1})=\dim_{k}\Irr(Y_{i},Y_{i+1})=\dim\Irr(Y_{i},Y_{i+1})_{F_{A}(Y_{i})}=1\]
 for all $i\in\{1,\ldots,m-1\}$. This shows that $\mathcal{E}(A)\ni k\cong F_{A}(Y)\cong F_{A}(Y_{1})\cong\cdots\cong F_{A}(Y_{m})\cong F_{A}(X_{t})$.

Now for each $K\in\Omega$, $\End_{A}(X^{K})/\rad\End_{A}(X^{K})\cong F_{A}(X)^{K}$
is a field, and hence $X^{K}\in\ind A^{K}$. 
\end{proof}
\begin{defn}
\label{sub:mods-in-non-homg-ql}(1) Set $\overline{L}(\mathcal{T}):=\bigoplus_{[X]\in\mathcal{T}}\mathbb{Z}u_{]X]}$.
Then $\overline{L}(\mathcal{T})/(q-1)$ is a Lie subalgebra of the
Lie algebra $\overline{L}(A)/(q-1)$ because $\mathcal{T}$ is closed
under extensions by \cite[3.1(3)]{Ri84} (or by the fact that $\mathcal{R}$
is separating $\mathcal{P}$ from $\mathcal{I}$).
\end{defn}
(2) Set $\overline{L}_{0}(\Lambda):=\bigoplus_{[X]\in[\ind_{0}\Lambda]}\mathbb{Z}u_{[X]}$.
Then $\overline{L}_{0}(\Lambda)/(q-1)$ is a Lie algebra with Hall
commutator the Lie bracket as in Sect.$\:$\ref{Hall}. 

\begin{lem}
\label{lem:Lie-T-iso-Lie-cyc-modulo}For each $K\in\Omega$ let $\mathcal{T}(A^{K})$
be the non-homogeneous tube of $\Gamma_{A^{K}}$ corresponding to
$\mathcal{T}$. Then we have\[
\overline{L}(\mathcal{T}(A^{K}))/(|K|-1)\cong\overline{L}_{0}(\Lambda^{K})/(|K|-1).\]
 
\end{lem}
\begin{proof}
This follows from the isomorphism (\ref{eq:tube-cyclic-qvalg}) by
the formula (\ref{eq:Hall-coefficient}). 
\end{proof}
\begin{defn}
Since $\mod_{0}\Lambda$ has Hall polynomials by \cite[Theorem 2.7]{Guo},
we can define a Lie algebra $\overline{L}_{0}(\Lambda)_{1}:=\bigoplus_{[X]\in[\mod_{0}\Lambda]}\mathbb{Z}u_{[X]}$
over $\mathbb{Z}$ with the Lie bracket defined by\[
[u_{[X]},u_{[Y]}]:=\sum_{[Z]\in[\ind_{0}\Lambda]}(\varphi_{[X][Y]}^{[Z]}(1)-\varphi_{[Y][X]}^{[Z]}(1))u_{[Z]}\]
for all $[X],[Y]\in[\ind_{0}\Lambda]$ using Hall polynomials $\varphi_{[X][Y]}^{[Z]}$.
\end{defn}
\begin{lem}
We have isomorphisms\begin{equation}
\overline{L}(\mathcal{T}(A^{K}))/(|K|-1)\cong\overline{L}_{0}(\Lambda)_{1}/(|K|-1)\label{eq:tube-cyclic-Liealg}\end{equation}
for all $K\in\Omega$.
\end{lem}
\begin{proof}
Since $F_{[X^{K}][Y^{K}]}^{[Z^{K}]}=\varphi_{[X][Y]}^{[Z]}(1)$ in
$\mathbb{Z}/(|K|-1)\mathbb{Z}$ for all $X,Y,Z\in\ind_{0}\Lambda$
and all $K\in\Omega$, we have $\overline{L}_{0}(\Lambda^{K})/(|K|-1)\cong\overline{L}_{0}(\Lambda)_{1}/(|K|-1)$
(here notations $F_{[X^{K}][Y^{K}]}^{[Z^{K}]}$ are for $\Lambda$-modules)
by Remark \ref{rmk:shape-of-T}. Hence the assertion follows by Lemma
\ref{lem:Lie-T-iso-Lie-cyc-modulo}.
\end{proof}
The Lie bracket of $\overline{L}_{0}(\Lambda)_{1}$ is easily described
as follows. First we have a bijection $[\ind_{0}\Lambda]\to R_{0}\times\mathbb{N}$
defined by $[M]\mapsto(i,j)$ with $\top M\cong S_{i}$ and $l(M)=j$
because all modules in $\ind_{0}\Lambda$ are uniserial. We identify
$R_{0}$ with $\mathbb{Z}/\mathbb{Z}p$, and for each $(i,j)\in R_{0}\times\mathbb{N}$
denote by $m(i,j)$ the isoclass of the indecomposable modules in
$\ind_{0}\Lambda$ corresponding to $(i,j)$. We choose a representative
$M(i,j)\in m(i,j)$ for all $(i,j)\in R_{0}\times\mathbb{N}$. Then
as calculated in \cite[1.2]{Asa02} we have the following: For each
$m(i,j),m(f,g)\in[\ind_{0}\Lambda]$

\begin{equation}
[u_{m(i,j)},u_{m(f,g)}]=\delta_{(i+j),f}u_{m(i,j+g)}-\delta_{(f+g),i}u_{m(f,j+g)}\label{eq:Lie-bracket-cyclic}\end{equation}
in $\overline{L}_{0}(\Lambda)_{1}$.

\begin{defn}
Define $L_{0}(\Lambda)_{1}$ to be the Lie subalgebra of $\overline{L}_{0}(\Lambda)_{1}$
generated by all $u_{[S]}$ with $S$ simple modules in $\ind_{0}\Lambda$. 
\end{defn}
From the formula (\ref{eq:Lie-bracket-cyclic}) we obtain the following.

\begin{lem}
Let \emph{}$X\in\ind_{0}\Lambda$ \emph{}and assume that \emph{}$l:=l(X)\not\in\mathbb{Z}p$.
\emph{}If \emph{}$p>2$ \emph{}$($resp. if \emph{}$p=2$$)$, \emph{}then
\emph{}$u_{[X]}$ \emph{}$($resp. \emph{}$2^{m}u_{[X]}$ \emph{}with
\emph{}$m:=(l-1)/2$$)$ \emph{}is obtained from \emph{}$\{ u_{[S]}\mid S\textrm{ simple in }\ind_{0}\Lambda\}$
by Lie brackets in \emph{}$\overline{L}_{0}(\Lambda)_{1}$, in particular,
we have \emph{}$u_{[X]}\in L_{0}(\Lambda)_{1}^{\mathbb{Q}}$.
\end{lem}
\begin{proof}
We put $L$ to be the subset of $\overline{L}_{0}(\Lambda)_{1}$ consisting
of elements obtained from $u_{[S]}$ with $S$ simple modules in $\ind_{0}\Lambda$
by Lie brackets in $\overline{L}_{0}(\Lambda)_{1}$. There exist a
unique $m\in\mathbb{N}_{0}$ and a unique $r\in\{1,\ldots,p-1\}$
such that $l=mp+r$. We show the assertion by induction on $l\ge1$.
If $l=1$, then $X$ is simple and there is nothing to show. Assume
$l>1$. Without loss of generality we may assume that $X=M(1,l)$. 

Case 1. $p>2$, $r>1$. In this case if we put $t:=(m-1)p+(r-1)$,
then $l=t+(p+1)$, $t,p+1<l$, and $t,p+1\not\in\mathbb{Z}p$, which
implies $u_{m(1,t)},u_{m(r,p+1)}\in L$ by induction hypothesis. Since
$r+1\ne1$ in $\mathbb{Z}/\mathbb{Z}p$, we have $[u_{m(1,t)},u_{m(r,p+1)}]=u_{[X]}$
by the formula (\ref{eq:Lie-bracket-cyclic}). Hence $u_{[X]}\in L$.

Case 2. $p>2,$ $r=1$. In this case if we put $t:=(m-1)p+(p-1)$,
then $l=t+2$, $t,2<l$, and $t,2\not\in\mathbb{Z}p$, which implies
$u_{m(t,t)},u_{m(p,2)}\in L$ by induction hypothesis. Since $1\ne2$
in $\mathbb{Z}/\mathbb{Z}p$ we have $[u_{m(1,t)},u_{m(p,2)}]=u_{[X]}$
by the formula (\ref{eq:Lie-bracket-cyclic}). Hence $u_{[X]}\in L$.

Case 3. $p=2$. In this case by induction hypothesis we have $2^{m-1}u_{m(1,l-2)}\in L$.
Then $[[u_{m(1,1)},u_{m(2,1)}],u_{m(1,l-2)}]=[u_{m(1,2)}-u_{m(2,2)},u_{m(1,l-2)}]=2u_{[X]}$
shows that $2^{m}u_{[X]}\in L$.
\end{proof}
By (\ref{eq:tube-cyclic-qvalg}) and (\ref{eq:tube-cyclic-Liealg})
the lemma above implies the following:

\begin{prop}
\label{prp:mod-in-nonhom-tube}Let \emph{}$X\in\mathcal{T}$ \emph{}and
assume that the quasi-length $l$ of \emph{}$X$ \emph{}is not a multiple
of \emph{}$p$. If \emph{}$p>2$ $($resp. if \emph{}$p=2$$)$, then
\emph{}$\mathbf{u}_{[X]}$ $($resp. \emph{}$2^{m}\mathbf{u}_{[X]}$
\emph{}with \emph{}$m:=(l-1)/2$$)$ \emph{}is obtained from $\{\mathbf{u}_{[Y]}\mid Y\textrm{ modules on the mouth of }\mathcal{T}\}$
\emph{}by Lie brackets in the Lie algebra \emph{}$\Pi$. In particular,
\emph{}$\mathbf{u}_{[X]}\in L(A)_{1}^{\mathbb{Q}}$ \emph{}by Lemma
\emph{\ref{lem:non-sincere}.}
\end{prop}
\emph{\qed}

\begin{rem}
\label{rem:rg-cmp}In other words, the last statement above is stated
as follows: Let $v$ be a positive root of $\chi_{A}$. If $M(v)$
is regular, then $\mathbf{u}_{m(v)}\in L(A)_{1}^{\mathbb{Q}}$. In
particular, if $M(v)$ is exceptional, then $\mathbf{u}_{m(v)}\in L(A)_{1}$.
The last statement also follows by Proposition \ref{lem:non-sincere}
because in this case $M(v)$ is not sincere.
\end{rem}
The following is obvious by the formula (\ref{eq:Lie-bracket-cyclic}):

\begin{lem}
\label{sub:inds-dimvec-delta-in-honhmg}In \emph{$\overline{L}_{0}(\Lambda)_{1}$}
we have \emph{}$[u_{m(i,sp)},u_{m(j,tp)}]=0$ for all $i,j\in R_{0}$
and for all $s,t\in\mathbb{N}$.\emph{\hfil\qed}
\end{lem}
\begin{prop}
\label{prp:im-rt-commute}Let \emph{}$\mathcal{T}$ \emph{}be a non-homogeneous
tube of \emph{}$\Gamma_{A}$\emph{, and let} $X,Y\in\mathcal{T}$.
If both $\udim X$ and $\udim Y$ are in \emph{}$\mathbb{Z}\delta$,
then \emph{}$[\mathbf{u}_{[X]},\mathbf{u}_{[Y]}]=0$ \emph{}in \emph{}$\Pi$.
\end{prop}
\begin{proof}
Both the quasi-lengths of $X$ and $Y$ are multiples of $p$ by Sect.$\:$\ref{sub:tubular-family}.
Hence the assertion follows from the lemma above by (\ref{eq:tube-cyclic-Liealg}).
\end{proof}

\subsection{Regular root modules modulo $\boldsymbol{{\delta}}$}

\begin{lem}
\label{lem:rg-prpl}Let $v$ be a positive root of $\chi_{A}$with
$M(v)$ a regular module. Then there exists an $x\in Q_{0}\setminus\{1,\infty\}$
which is determined by $v$ modulo $\mathbb{Z}\delta$ such that $[\mathbf{u}_{m(v)},h_{x}]=\mathbf{u}_{m(v+\delta)}$
in $\Pi$.
\end{lem}
\begin{proof}
Set $M:=M(v)$. Then $M$ is in a non-homogeneous tube $\mathcal{T}$
of rank $p>1$. Let $\{ S_{i}\mid i\in\mathbb{Z}/p\mathbb{Z}\}$ be
the set of quasi-simples in $\mathcal{T}$ with $S_{i+1}=\tau S_{i}\ (i\in\mathbb{Z}/p\mathbb{Z})$.
Then $M=S_{i}[t]$ for some $i\in\mathbb{Z}/p\mathbb{Z}$ and $t\in\mathbb{N}$
with $p\nmid t$. Express $t$ as $t=mp+r$ with $m\ge0$ and $1\le r\le p-1$.
Then $\mathbf{u}_{[S_{i+r}[p]]}-\mathbf{u}_{[S_{i+r+1}[p]]}=h_{x}$
for a unique $x\in Q_{o}\setminus\{1,\infty\}$ ($x$ is determined
by $i+r$) and we have \[
[\mathbf{u}_{m(v)},h_{x}]=[\mathbf{u}_{m(v)},\mathbf{u}_{[S_{i+r}[p]]}-\mathbf{u}_{[S_{i+r+1}[p]]}]=\mathbf{u}_{[S_{i}[t+p]]}=\mathbf{u}_{m(v+\delta)}.\]
Next we show that the vertex $x\in Q_{0}$ is determined by $v$ modulo
$\mathbb{Z}\delta$. Let $v'$ be another positive root of $\chi_{A}$
with $M(v')$ regular and assume that $v-v'\in\mathbb{Z}\delta$.
Then $M(\delta')$ is also in $\mathcal{T}$ and $M(\delta')=S_{i'}[t']$
for some $i'\in\mathbb{Z}/p\mathbb{Z}$ and $t'\in\mathbb{N}$ with
$p\nmid t'$. Let $t'=m'p+r'$ with $m'\ge0$ and $1\le r'\le p-1$.
Then since $v-v'\in\mathbb{Z}\delta$, we have $i=i'$ and $t-t'\in p\mathbb{Z}$.
Hence $r=r'$ and thus $i+r=i'+r'$ as desired.
\end{proof}
\begin{rem}
Let $v$ be a positive root of $\chi_{A}$ with $M(v)$ a regular
module. In this case $v':=\org v$ (see Definition \ref{def:degree})
is the unique positive root of $\chi_{A}$ such that $M(v')$ exceptional
and $v-v'\in\mathbb{Z}\delta$. As easily seen it is given by \[
\org v=\begin{cases}
\deg v & \textrm{if }\deg v>0;\\
\deg v+\delta & \textrm{otherwise.}\end{cases}\]
 
\end{rem}
\begin{prop}
\label{prp:dif-rg}Let $v$ be a positive root of $\chi_{A}$ with
$M(v)$ a regular module. Then $\mathbf{u}_{m(v)}-\mathbf{u}_{m(\org v)}\in I(A)$.
\end{prop}
\begin{proof}
Set $v':=\org v$ and let $m$ be the non-negative integer such that
$v=v'+m\delta$. It is enough to show that\[
\mathbf{u}_{m(v-t\delta)}-\mathbf{u}_{m(v-(t+1)\delta)}\in I(A)\]
for all $t=0,1,\ldots,m-1$. By Lemma \ref{lem:rg-prpl} we have\[
\mathbf{u}_{m(v-t\delta)}-\mathbf{u}_{m(v-(t+1)\delta)}=[\mathbf{u}_{m(v-(t+1)\delta)}-\mathbf{u}_{m(v-(t+2)\delta)},h_{x}]\]
for all $t=0,1,\ldots,m-2$. Hence it is enough to show that\[
\mathbf{u}_{m(v'+\delta)}-\mathbf{u}_{m(v')}\in I(A).\]
Let $\mathcal{T}$ be the non-homogeneous tube containing $M(v')$.
Then the set of quasi-simple modules in $\mathcal{T}$ is equal to
$\mathcal{S}:=\{ M(e_{i2}),M(e_{i3}),\ldots,M(e_{ip(i)}),M(v_{i})\}$
for some $i\in\{1,2,3\}$, where $v_{i}:=\delta-\sum_{j=2}^{p(i)}e_{ij}$.
Set $p:=p(i)$, $e_{j}:=e_{ij}\quad(j=2,\ldots,p)$ for short. And
for each $j\in\mathbb{Z}/p\mathbb{Z}$, we set $m_{j}:=\left\{ \begin{array}{ll}
m(v_{i}) & \text{if }j=1;\\
m(e_{j}) & \text{otherwise.}\end{array}\right.$ Let $t$ be the quasi-length of $M(v')$. Then $1\le t\le p-1$ and
there exists some $j\in\mathbb{Z}/p\mathbb{Z}$ such that\[
\mathbf{u}_{m(v')}=[\mathbf{u}_{m_{j}},\mathbf{u}_{m_{j+1}},\ldots,\mathbf{u}_{m_{j+t-1}}].\]

Case (i). If $\{ j,j+1,\ldots,j+t-1\}\ne\{1\}$, then there exists
some $s\in\{ j,j+1,\ldots,j+t-1\}\setminus\{1\}$ such that\[
\mathbf{u}_{m(v'+\delta)}=[\mathbf{u}_{m_{j}},\ldots,\mathbf{u}_{m_{s-1}},\mathbf{u}_{m(e_{s}+\delta)},\mathbf{u}_{m_{s+1}},\ldots,\mathbf{u}_{m_{j+t-1}}].\]
Hence we have\[
\mathbf{u}_{m(v'+\delta)}-\mathbf{u}_{m(v')}=[\mathbf{u}_{m_{j}},\ldots,\mathbf{u}_{m_{s-1}},\mathbf{u}_{m(e_{s}+\delta)}-\mathbf{u}_{m(e_{s})},\mathbf{u}_{m_{s+1}},\ldots,\mathbf{u}_{m_{j+t-1}}],\]
which is in $I(A)$ because $\mathbf{u}_{m(e_{s}+\delta)}-\mathbf{u}_{m(e_{s})}\in I(A)$.

Case (ii). Otherwise, we have $v'=v_{i}$. Then $\mathbf{u}_{m(v')}=[\mathbf{u}_{m(v_{i}-e_{1})},\mathbf{u}_{m(e_{1})}]$
and $\mathbf{u}_{m(v'+\delta)}=[\mathbf{u}_{m(v_{i}-e_{1})},\mathbf{u}_{m(e_{1}+\delta)}]$.
Since $\mathbf{u}_{m(v_{i}-e_{1})}\in L(A)_{1}$ and $\mathbf{u}_{m(e_{1}+\delta)}-\mathbf{u}_{m(e_{1})}\in I(A)$
we have\[
\mathbf{u}_{m(v'+\delta)}-\mathbf{u}_{m(v')}=[\mathbf{u}_{m(v_{i}-e_{1})},\mathbf{u}_{m(e_{1}+\delta)}-\mathbf{u}_{m(e_{1})}]\in I(A).\]

\end{proof}

\subsection{Simple modules modulo $\boldsymbol{{\delta}}$}

\begin{lem}
\label{lem:end-prpl}Assume that $\Delta\ne A_{1}$. Then for each
$t\in\mathbb{N}_{0}$ we have the following.
\end{lem}
\begin{enumerate}
\item $[\mathbf{u}_{m(e_{1}+t\delta)},h_{12}]=\mathbf{u}_{m(e_{1}+(t+1)\delta)}$;
and
\item $[\mathbf{u}_{m(e_{\infty}+t\delta)},h_{1p(1)}]=\mathbf{u}_{m(e_{\infty}+(t+1)\delta)}$.
\end{enumerate}
\begin{proof}
We only show the statement (2), the first one is shown similarly.
Now $h_{1p(1)}=\mathbf{u}_{[X_{1p(1)}]}-\mathbf{u}_{[X_{1,p(1)-1}]}\in L(A)_{1}$,
and hence it is enough to show the following two equalities:\begin{eqnarray}
[\mathbf{u}_{m(e_{\infty}+t\delta)},\mathbf{u}_{[X_{1p(1)}]}] & = & \mathbf{u}_{m(e_{\infty}+(t+1)\delta)};\label{eq:aa}\\
{}[\mathbf{u}_{m(e_{\infty}+t\delta)},\mathbf{u}_{[X_{1,p(1)-1}]}] & = & 0.\label{eq:bb}\end{eqnarray}
The equality (\ref{eq:aa}) holds if and only if for all $K\in\Omega$,
we have\[
[u_{[M(e_{\infty}+t\delta)^{K}]},u_{[X_{1p(1)}^{K}]}]\equiv u_{[M(e_{\infty}+(t+1)\delta)^{K}]}\quad(\mod(|K|-1)).\]
Since $F_{X_{1,pI1)}^{K},M(e_{\infty}+t\delta)^{K}}^{M(e_{\infty}+(t+1)\delta)^{K}}=0$,
this holds if $F_{M(e_{\infty}+t\delta)^{K},X_{1,p(1)}^{K}}^{M(e_{\infty}+(t+1)\delta)^{K}}=1$.
To this end we only have to show the following:\begin{equation}
F_{M(e_{\infty}+t\delta),X_{1,p(1)}}^{M(e_{\infty}+(t+1)\delta)}=1\label{eq:cc}\end{equation}
because the value of the right hand side does not depend on $q$.
Similarly the equality (\ref{eq:bb}) holds if the following holds:\begin{equation}
F_{M(e_{\infty}+t\delta),X_{1,p(1)-1}}^{M(e_{\infty}+(t+1)\delta)}=0.\label{eq:dd}\end{equation}
First we show (\ref{eq:dd}). Let $f\in\Hom_{A}(X_{1,p(1)-1},M(e_{\infty}+(t+1)\delta))$.
Then we have a commutative diagram $$
\begin{CD}
k @<0<< k\\
@V{f_{1,p(1)-1}}VV @VV{f_{1p(1)}}V\\
k^{t+1} @<<{\id}< k^{t+1},
\end{CD}
$$which shows that $f_{1,p(1)}=0$, and that $f$ is not a monomorphism.
Therefore there is no exact sequence of the form\[
0\to X_{1,p(1)-1}\to M(e_{\infty}+(t+1)\delta)\to M(e_{\infty}+t\delta)\to0,\]
which proves the equality (\ref{eq:dd}). Next we show the equality
(\ref{eq:cc}). Assume that we have an exact sequence $$
\begin{CD}
0 @>>> X_{1,p(1)} @>f>> M(e_{\infty}+(t+1)\delta) @>g>> M(e_{\infty}+ t\delta) @>>> 0.
\end{CD}
$$

Then the direct calculation shows that\[
\begin{array}{cc}
f_{\infty}=\left[\begin{array}{c}
0\\
\vdots\\
0\\
a\end{array}\right]\in\Mat_{(t+2)\times1}(k), & g_{\infty}=\left[\begin{array}{cccc}
 &  &  & 0\\
 & g_{1} &  & \vdots\\
 &  &  & 0\\
* & \cdots & * & c\end{array}\right]\in\Mat_{(t+1)\times(t+2)}(k)\\
f_{1}=\left[\begin{array}{c}
0\\
\vdots\\
0\\
a\end{array}\right]\in\Mat_{(t+1)\times1}(k), & g_{\infty}=\left[\begin{array}{cccc}
b & * & \cdots & *\\
0\\
\vdots &  & g_{1}\\
0\end{array}\right]\in\Mat_{(t+1)\times(t+2)}(k)\end{array}\]
for some $a,b,c\in k$. Therefore we have\[
\begin{array}{cc}
g_{\infty}=\left[\begin{array}{cccc}
b & c &  & \textrm{\Huge0}\\
 & b & \ddots\\
 &  & \ddots & c\\
\textrm{\Huge0} &  &  & b\end{array}\right]\in\Mat_{(t+1)\times(t+2)}(k), & g_{1}=\left[\begin{array}{cccc}
b & c &  & \textrm{\Huge0}\\
 & b & \ddots\\
 &  & \ddots & c\\
\textrm{\Huge0} &  &  & b\end{array}\right]\in\Mat_{t\times(t+1)}(k)\end{array}.\]
 Then $f$ is a monomorphism, $g$ is an epimorphism and $gf=0$ if
and only if $a\ne0,c=0$ and $b\ne0$. Hence we have $|W_{M(e_{\infty}+t\delta),X_{1p(1)}}^{M(e_{\infty}+(t+1)\delta)}|=(q-1)^{2}$,
which shows the equality (\ref{eq:cc}) (see Sect.$\,$\ref{sub:Hall-numbers}).
\end{proof}
To deal with the case that $\Delta=A_{1}$ we need the following formula:

\begin{lem}
\label{lem:Szanto}Assume that $\Delta=A_{1}$. Let $K\in\Omega$,
$l,m\in\mathbb{N}_{0}$ and $X\in\ind A^{K}$. Then we have\[
F_{X,M(e_{1}+l\delta)^{K}}^{M(e_{1}+m\delta)^{K}}=F_{M(e_{\infty}+l\delta)^{K},X}^{M(e_{\infty}+m\delta)^{K}}=\begin{cases}
1 & \textrm{if $l<m,\:\udim X=(m-l)\delta$}\\
0 & \textrm{otherwise.}\end{cases}\]

\end{lem}
\begin{proof}
This follows from a direct calculation, or from Sz\'{a}nt\'{o} \cite[Lemma 1.3]{Sz}.
\end{proof}
Using the lemma above we obtain the statement corresponding to Lemma
\ref{lem:end-prpl} in the $A_{1}$ case.

\begin{lem}
\label{lem:end-prpl-A1}Assume that $\Delta=A_{1}$. Then for each
$t\in\mathbb{N}_{0}$ we have the following.
\end{lem}
\begin{enumerate}
\item $[h_{1},\mathbf{u}_{m(e_{1}+t\delta)}]=2\mathbf{u}_{m(e_{1}+(t+1)\delta)}$;
and
\item $[\mathbf{u}_{m(e_{\infty}+t\delta)},h_{1}]=2\mathbf{u}_{m(e_{\infty}+(t+1)\delta)}$.
\end{enumerate}
\begin{proof}
(1) It is enough to show that $[\sum_{c\in K}u_{W_{c}(K)}+u_{X_{11}},u_{m(e_{1}+t\delta)^{K}}]=2u_{m(e_{1}+(t+1)\delta)^{K}}$
in $\overline{L}(A)_{1}/(|K|-1)$ for each $K\in\Omega$. By Lemma
\ref{lem:Szanto} the left hand side is equal to $(|K|+1)u_{m(e_{1}+(t+1)\delta)^{K}}$,
which is equal to the right hand side in $\overline{L}(A)_{1}/(|K|-1)$.

(2) This is shown similarly.
\end{proof}
\begin{prop}
\label{prp:dif-smp}For each $x\in Q_{0}$ and $t\in\mathbb{N}$,
we have\[
\mathbf{u}_{m(e_{x}+t\delta)}-\mathbf{u}_{m(e_{x})}\in\begin{cases}
I(A) & \textrm{if $\Delta\ne A_{1}$};\\
I(A)^{\mathbb{Z}[2^{-1}]} & \textrm{if $\Delta=A_{1}$}.\end{cases}\]

\end{prop}
\begin{proof}
If $x\ne1,\infty$, then $v=e_{x}+t\delta$ is a positive root with
$M(v)$ a regular module and $\deg v=e_{x}$, and hence the statement
holds by Proposition \ref{prp:dif-rg}. Now let $x=1,\infty$. First
we consider the case that $\Delta\ne A_{1}$. It is enough to show
that $\mathbf{u}_{m(e_{x}+(t+1)\delta)}-\mathbf{u}_{m(e_{x}+t\delta)}\in I(A)$
by induction on $t\in\mathbb{N}_{0}$. This holds for $t=0$ by definition
of $I(A)$. Let $t\ge1$. By setting $u:=\begin{cases}
h_{12} & \textrm{if }x=1\\
h_{1p(1)} & \textrm{if }x=\infty\end{cases}$ we have $\mathbf{u}_{m(e_{x}+(t+1)\delta)}-\mathbf{u}_{m(e_{x}+t\delta)}=[\mathbf{u}_{m(e_{x}+t\delta)}-\mathbf{u}_{m(e_{x}+(t-1)\delta)},u]$
by Lemma \ref{lem:end-prpl}. Since $u\in L(A)_{1}$, the right hand
side is in $I(A)$ by induction hypothesis. Next consider the case
that $\Delta=A_{1}$. Then by setting $u:=\begin{cases}
-h_{1} & \textrm{if }x=1\\
h_{1} & \textrm{if }x=\infty\end{cases}$ we have $2(\mathbf{u}_{m(e_{x}+(t+1)\delta)}-\mathbf{u}_{m(e_{x}+t\delta)})=[\mathbf{u}_{m(e_{x}+t\delta)}-\mathbf{u}_{m(e_{x}+(t-1)\delta)},u]$
for each $t\ge1$ by Lemma \ref{lem:end-prpl-A1}. This proves the
remaining statement.
\end{proof}

\subsection{Stability of Hall numbers}

Let $K\in\Omega$ and $X$, $Y\in\ind A^{K}$. If $v:=\udim X+\udim Y$
is a root of $\chi$ , then we set\[
b_{[X],[Y]}:=(F_{XY}^{M(v)^{K}}-F_{YX}^{M(v)^{K}})+(|K|-1)\mathbb{Z}\in\mathbb{Z}/(|K|-1)\mathbb{Z}.\]
Note that in this case we have $[u_{[X]},u_{[Y]}]=b_{[X],[Y]}u_{m(v)^{K}}\in\overline{L}(A^{K})/(|K|-1)$
and $b_{[X],[Y]}$ is uniquely determined by this property.

\begin{prop}
\label{pro:img-stb}Assume that $\Delta\ne A_{1}$. Let $K\in\Omega$,
$X\in\ind A^{K}$, and $x\in Q_{0}$. If $\udim X\in\mathbb{Z}\delta$,
then for any $t\in\mathbb{N}$ we have\[
b_{[X],m(e_{x}+t\delta)^{K}}=b_{[X],m(e_{x})^{K}}.\]

\end{prop}
\begin{proof}
We have $\udim X=d\delta$ for some $d\in\mathbb{N}$. By Lemma \ref{lem:rg-prpl}
and Lemma \ref{lem:end-prpl} there exists some $x_{ij}\in Q_{0}\setminus\{1,\infty\}$
such that $u:=u_{X_{ij}^{K}}-u_{X_{i,j-1}^{K}}$ satisfies $u_{m(e_{x}+t\delta)^{K}}=[u_{m(e_{x}+(t-1)\delta)^{K}},u]$
for any $t\in\mathbb{N}$. By Proposition \ref{prp:im-rt-commute}
we have $[u_{[X]},u]=0$. Then\begin{eqnarray*}
[u_{[X]},u_{m(e_{x}+t\delta)^{K}}] & = & [u_{[X]},[u_{m(e_{x}+(t-1)\delta)^{K}},u]]\\
 & = & [[u_{[X]},u_{m(e_{x}+(t-1)\delta)^{K}}],u]+[u_{m(e_{x}+(t-1)\delta)^{K}},[u_{[X]},u]]\\
 & = & b_{[X],m(e_{x}+(t-1)\delta)^{K}}[u_{m(e_{x}+(t-1+d)\delta)^{K}},u]\\
 & = & b_{[X],m(e_{x}+(t-1)\delta)^{K}}u_{m(e_{x}+(t+d)\delta)^{K}},\end{eqnarray*}
which shows $b_{[X],m(e_{x}+t\delta)^{K}}=b_{[X],m(e_{x}+(t-1)\delta)^{K}}$.
By repeating this we obtain the assertion.
\end{proof}
When $\Delta=A_{1}$, we have the following by Lemma \ref{lem:Szanto}:

\begin{prop}
\label{pro:Krn-stb}Assume that $\Delta=A_{1}$. Let $m\in\mathbb{N}$,
$K\in\Omega$, and $X\in\ind A^{K}$. If $\udim X=m\delta$, then
for each $x\in Q_{0}=\{1,\infty\}$ and each $t\in\mathbb{N}$ we
have $F_{X,M(e_{x})^{K}}^{M(e_{x}+m\delta)^{K}}=F_{X,M(e_{x}+t\delta)^{K}}^{M(e_{x}+(t+m)\delta)^{K}}(=1)$.
Thus

\[
b_{[X],m(e_{x}+t\delta)^{K}}=b_{[X],m(e_{x})^{K}}.\]

\end{prop}

\subsection{Hall polynomials}

\begin{lem}
\label{lem:Hall-sub-factor}If $M$ and $N$ are preprojective $A$-modules,
then there exists polynomials $\varphi_{*,N}^{M}(T),\varphi_{N,*}^{M}\in\mathbb{Z}[T]$
such that\[
F_{*,N^{K}}^{M^{K}}=\varphi_{*,N}^{M}(|K|)\textrm{ and $F_{N^{K},*}^{M^{K}}=\varphi_{N,*}^{M}(|K|)$}\]
for all $K\in\Omega$.
\end{lem}
\begin{proof}
Since the class of preprojective $A$-modules are closed under submodules,
the images of homomorphisms between preprojective $A$-modules are
again preprojective. Noting this the statement can be shown by the
same argument as in the step (4) of the proof of \cite[Theorem 1]{Ri90-Banach}
if we replace the whole Auslander-Reiten quiver of $A$ by its preprojective
component. The details are left to the reader.
\end{proof}
\begin{prop}
\label{prp:Hall-preproj-smp}If $M$ and $N$ are preprojective $A$-modules
and $S$ is a simple $A$-module, then there exist Hall polynomials
$\varphi_{SN}^{M}$, $\varphi_{NS}^{M}$. 
\end{prop}
\begin{proof}
If $\udim M-\udim N\ne\udim S$, then we can take $\varphi_{SN}^{M}=0$,
$\varphi_{NS}^{M}=0$. Assume that $\udim M-\udim N=\udim S$. Then
since any $A$-module having dimension vector $\udim S$ is isomorphic
to $S$, we can take $\varphi_{SN}^{M}=\varphi_{*N}^{M}$ and $\varphi_{NS}^{M}=\varphi_{N*}^{M}$
by using Lemma \ref{lem:Hall-sub-factor}.
\end{proof}
Recall that an algebra is called \emph{representation-directed} if
it is representation-finite and its Auslander-Reiten quiver does not
contain oriented cycles (\cite{Ri90-Banach}). By Corollary \ref{cor:min-rep-inf}
an $A$-module $X$ is non-sincere if and only if $\supp X$ is representation-directed
because the latter is equivalent to saying that $\supp X$ is representation-finite
in our case. Note that regular exceptional $A$-modules are non-sincere,
for which we apply the following later.

\begin{prop}
\label{prp:Hall-preproj-nonsincere}Let $M$ and $N$ be preprojective
$A$-modules. If $X$ is a non-sincere $A$-module, then there exists
a Hall polynomial $\varphi_{XN}^{M}$.
\end{prop}
\begin{proof}
Since $X$ is non-sincere, $\supp X$ is representation-directed as
explained above. Noting that any $A$-module with dimension vector
$\udim X$ is a module over $\supp X$, the assertion can be shown
by the same argument as in the last step of the proof of \cite[Theorem 1]{Ri90-Banach}.
For the benefit of the reader we outline the proof.

If $\udim M\ne\udim X+\udim N$, then we can take $\varphi_{XN}^{M}=0$.
Therefore we may assume that $\udim M=\udim X+\udim N$. We define
$\varphi_{XN}^{M}$ by induction on $\udim X\in K_{0}(A)$ (here $K_{0}(A)$
is regarded as a poset by the order defined in Definition \ref{sub:K0}).
If $\udim X=0$, then we may take $\varphi_{XN}^{M}=1$. Assume $\udim X\ne0$.
Let $X=\bigoplus_{i=1}^{m}X_{i}^{(t_{i})}$ be a direct sum decomposition
of $X$ into pairwise non-isomorphic indecomposable $A$-modules $X_{i}$.
If $m=1$, $X$ is called \emph{homogeneous}. Note that all $X_{i}$
are modules over $\supp X$ and regarded as vertices in the AR-quiver
$\Gamma$ of $\supp X$. Since $\supp X$ is representation-directed,
we can regard the set $\Gamma_{0}$ of vertices of $\Gamma$ as a
poset by defining an order as follows: For $x,y\in\Gamma$, $x$ is
smaller than $y$ if and only if there exists a path from $x$ to
$y$ in $\Gamma$. We may assume that $X_{1}$ is minimal among all
$X_{i}$ in $\Gamma_{0}$.

Case 1. $X$ is non-homogeneous (i.e., $m>1$). Let $X';=X_{1}^{(t_{1})}$
and $X'':=\bigoplus_{i\ne1}X_{i}^{(t_{i})}$. Then $X=X'\oplus X''$,
$\Hom_{A}(X'',X')=0$ and $\Ext_{A}^{1}(X',X'')=0$. We define $\varphi_{XN}^{M}$
as follows.\[
\varphi_{XN}^{M}:=\sum_{V\in\mathcal{V}}\varphi_{X'V}^{M}\varphi_{X''N}^{V},\]
where $\mathcal{V}$ is a complete set of representatives of isoclasses
of submodules $V$ of $M$ such that $\udim V=\udim X''+\udim N$.
Then all the terms on the right hand side are already defined by the
induction hypothesis because $N,V,M$ are preprojective and $X',X''$
are modules over $\supp X$. Here we can show that $\varphi_{XN}^{M}$
is a Hall polynomial, i.e., that $\varphi_{XN}^{M}(|K|)=F_{X^{K}N^{K}}^{M^{K}}$
for all $K\in\Omega$, by using the associativity of Hall multiplication
and the facts that $\Hom_{A}(X'',X')=0$ and $\Ext_{A}^{1}(X',X'')=0$.

Case 2. $X$ is homogeneous. In this case we define $\varphi_{XN}^{M}$
as follows.\[
\varphi_{XN}^{M}:=\varphi_{*,N}^{M}-\sum_{U\in\mathcal{U}}\varphi_{UN}^{M},\]
where $\varphi_{*,N}^{M}$ is the polynomial defined in Lemma \ref{lem:Hall-sub-factor},
and $\mathcal{U}$ is a complete set of representatives of isoclasses
of modules $U$ over $\supp X$ such that $\udim U=\udim X$ and $U\not\cong X$.
Here we see that each $U\in\mathcal{U}$ cannot be homogeneous because
if $U$ is homogeneous, the condition that $\udim U=\udim X$ implies
$U\cong X$. Hence the right hand side is defined by Case 1, and it
is easy to see that this $\varphi_{XN}^{M}$ is a Hall polynomial.
\end{proof}

Dually we have the following.

\begin{prop}
\label{pro:Hall-preinj-nonsincere}Let $M$ and $N$ be preinjective
$A$-modules. If $X$ is a non-sincere $A$-module, then there exists
a Hall polynomial $\varphi_{NX}^{M}$.\hfil\qed.
\end{prop}
This proposition proves the following.

\begin{cor}
\label{cor:prin-Hall}Let $(x_{1},x_{2},\dots,x_{n})$ be a permutation
of elements in $\Delta_{0}$ with $x_{1}=1$ such that $0\ne[F_{x_{1}},F_{x_{2}},\dots,F_{x_{n}}]\in\mathfrak{g}(\Delta)$.
Set $f_{i}:=\delta-e_{x_{i}}$ for all $i=1,\dots,n$. Then there
exist Hall polynomials $\varphi_{M(f_{1}+\cdots+f_{i-1}),M(f_{i})}^{M(f_{1}+\cdots+f_{i-1}+f_{i})}$
for all $i=2,\dots,n$.\hfil\qed
\end{cor}

\section{\textbf{Proof of Main Theorem\label{sec:Proof-of-Theorem}}}

In this section we prove our main result Theorem \ref{thm:realization}.
First we show the following.

\subsection{Claim}

\emph{There exists a homomorphism} \[
\phi:\mathfrak{g}(\Delta)\rightarrow L(A)_{1}^{\mathbb{C}}/I(A)^{\mathbb{C}}\]
 \emph{such that} $\phi(E_{x})=\varepsilon_{x},\phi(F_{x})=\zeta_{x},\phi(H_{x})=\eta_{x}$
\emph{for all} $x\in\Delta_{0}$.

Indeed, it is enough to verify the following equations for all $x,y\in\Delta_{0}$:\begin{eqnarray}
[\eta_{x},\eta_{y}] & = & 0\label{et-et}\\
{[\varepsilon_{x},\zeta_{x}]} & = & \eta_{x}\label{ep-ze1}\\
{[\varepsilon_{x},\zeta_{y}]} & = & 0,\quad\textrm{if }x\neq y\label{ep-ze2}\\
{[\eta_{x},\varepsilon_{y}]} & = & a_{xy}\varepsilon_{y}\label{et-ep}\\
{[\eta_{x},\zeta_{y}]} & = & -a_{xy}\zeta_{y}\label{et-ze}\\
(\ad\varepsilon_{x})^{1-a_{xy}}\varepsilon_{y} & = & 0,\quad\textrm{if }x\neq y\label{ep-ep}\\
(\ad\zeta_{x})^{1-a_{xy}}\zeta_{y} & = & 0,\quad\textrm{if }x\neq y\label{ze-ze}\end{eqnarray}

(see (\ref{eq:a-xy}) for $a_{xy}$).

\subsection*{Verification of (\ref{ep-ze1})}

This follows from the construction of $\eta_{x}$'s.

\subsection*{Verification of \emph{}(\ref{et-et})}

\begin{lem}
\label{lem:distinct-tubes}Let $K\in\Omega$ and $(\mathcal{T}_{\rho})_{\rho\in\mathbb{P}^{1}(K)}$
the tubular family in $\mod A^{K}$. If $X$ and $Y$ are indecomposable
$A^{K}$-modules contained in tubes $\mathcal{T}_{\rho}$ and $\mathcal{T}_{\sigma}$,
respectively with $\rho\ne\sigma$, then $[u_{X},u_{Y}]=0$ in $\overline{L}(A^{K})/(|K|-1)$.
\end{lem}
\begin{proof}
First note that there are no nonzero homomorphisms between indecomposable
$A^{K}$-modules contained in distinct tubes because the tubular family
$(\mathcal{T}_{\rho})_{\rho\in\mathbb{P}^{1}(K)}$ is standard (Definition
\ref{dfn:separating} (2)). If there exists an exact sequence of the
form\[
0\to Y\ya{f}M\ya{g}X\to0\]
in $\mod A$ with $M$ indecomposable, then $\rank M=\rank X+\rank Y=0$
shows that $M$ is contained in a tube $\mathcal{T}_{\lambda}$ ($\lambda\in\mathbb{P}^{1}(K)$).
But $f\ne0$ and $g\ne0$ shows that $\rho=\lambda$ and $\lambda=\sigma$,
thus $\rho=\sigma$, a contradiction. Hence there are no exact sequence
of this form. By the symmetry of the argument the same statement still
holds even if we exchange $X$ and $Y$. Hence $[u_{X},u_{Y}]=0$
in $\overline{L}(A^{K})/(|K|-1)$.
\end{proof}
For each $K\in\Omega$ the lemma above shows that $[u_{W_{c}(K)},u_{W_{c'}(K)}]=0$,
$[u_{W_{c}(K)},u_{X_{ij}^{K}}]=0$, $[u_{X_{ij}^{K}},u_{X_{st}^{K}}]=0$
for all $c,c'\in K\setminus E_{\Delta}$, $\alpha_{ij},\alpha_{st}\in Q_{1}$
with $c\neq c'$ and $i\neq s$ (see (\ref{eq:zyogai})). It is trivial
that $[u_{W_{c}(K)},u_{W_{c}(K)}]=0$ for all $c\in K\setminus E_{\Delta}$,
and also by Proposition \ref{prp:im-rt-commute} we have $[u_{X_{ij}^{K}},u_{X_{it}^{K}}]=0$
for all $\alpha_{ij},\alpha_{it}\in Q_{1}$ . Therefore we have $[u_{X},u_{Y}]=0$
for all indecomposable $A^{K}$-modules $X$, $Y$ with dimension
vector $\delta$ and for all $K\in\Omega$. Hence by (\ref{eq:et-1})
and (\ref{eq:et-ij}) we obtain the equation (\ref{et-et}).

\subsection*{Verification of (\ref{ep-ze2})}

Let $x,y\in\Delta_{0}$ with $x\neq y$. Assume that there exists
an indecomposable $A$-module $M$ with $\udim M=\udim S_{x}+\udim T_{y}=e_{x}+\delta-e_{y}$.
Then $M$ is an indecomposable module over its support algebra $B:=\supp M=A/A\mathbf{e}_{y}A$
with dimension vector $w$, where $w_{z}=1$ if $z\neq x$; and $w_{z}=2$
if $z=x$ for all $z\in\Delta_{0}\setminus\{ y\}$. If $y=1$, then
$B$ is a hereditary algebra of Dynkin type $\Delta$. Since $w$
is a dimension vector of an indecomposable $B$-module, we must have
$x=\infty$, a contradiction. Hence if $y=1$, then there exists no
indecomposable $A$-modules of dimension vector $\udim S_{x}+\udim T_{y}$
, and we see that $[\varepsilon_{x},\eta_{y}]=0$. Assume next that
$y\neq1$. Then since $w$ is a dimension vector of an indecomposable
$B$-module, $\Delta$ is not of type $A_{n}$ and also $x=1$. In
this case we see $[u_{S_{x}},u_{T_{y}}]=(q-1)u_{M}$. Thus if we replace
$k$ with any $K\in\Omega$, we have $[u_{S_{x}},u_{T_{y}}]=(|K|-1)u_{M}=0$
in $\overline{L}(A^{K})/(|K|-1)$. Hence in any case we have the equation
(\ref{ep-ze2}).

\subsection*{Verification of (\ref{ep-ep})}

Let $x,y\in\Delta_{0}$ with $x\neq y$.

\subsubsection*{Case 1. $x,y$ are not neighbors in $\Delta$.}

In this case $1-a_{xy}=1$, and we have to show $[\varepsilon_{x},\varepsilon_{y}]=0$.
Now since $x,y$ are not neighbors in $\Delta$, there exist no indecomposable
$A$-modules of dimension vector $\udim S_{x}+\udim S_{y}$. This
shows $[\varepsilon_{x},\varepsilon_{y}]=0$.

\subsubsection*{Case 2. $x,y$ are neighbors in $\Delta$.}

In this case $1-a_{xy}=2$, and we have to show $[\varepsilon_{x},[\varepsilon_{x},\varepsilon_{y}]]=0$.
If this is nonzero, then there exists an indecomposable $A$-module
$M$ with $\udim M=2\udim S_{x}+\udim S_{y}$. But the support algebra
of $M$ is the algebra given by the full subquiver of $Q$ consisting
of the vertices $x,y$, which is of type $A_{2}$. Over this algebra
$M$ is still indecomposable but $\udim M=(2,1)$, which is impossible.

\subsection*{Verification of (\ref{ze-ze})}

Let $x,y\in\Delta_{0}$ with $x\neq y$.

\subsubsection*{Case 1. $x,y$ are not neighbors in $\Delta$. }

By the formula (\ref{eq:B-delta-ex}) we see

\[
\chi(\udim T_{x}+\udim T_{y})=\chi(2\delta-e_{x}-e_{y})=2.\]
Hence there exist no indecomposable $A$-modules of dimension vector
$\udim T_{x}+\udim T_{y}$, which shows that $[\zeta_{x},\zeta_{y}]=0$,
as desired.

\subsubsection*{Case 2. $x,y$ are neighbors in $\Delta$.}

Again by the formula (\ref{eq:B-delta-ex}) we see

\[
\chi(2\udim T_{x}+\udim T_{y})=\chi(3\delta-2e_{x}-e_{y})=7.\]
Hence similarly we have $[\zeta_{x},[\zeta_{x},\zeta_{y}]]=0$, as
desired.

\subsection*{Verification of (\ref{et-ep})}

Let $x,y\in\Delta_{0}$. Since $\chi(\delta+e_{y})=1$, we have an
indecomposable $A$-module $M=M(\delta+e_{y})$. By Proposition \ref{prp:dif-smp}
we have $\overline{\mathbf{u}}_{M}=\varepsilon_{y}$.

\subsubsection*{Case 1. $y\neq1$, say $y=x_{st}$ for some $s,t$ with $t>1$.}

In this case we may assume that $M=(M(z),M(\alpha))_{z\in Q_{0},\alpha\in Q_{1}}$
has the following structure:

\[
M(z)=\left\{ \begin{array}{ll}
k^{2} & \textrm{if $z=x_{st}$; }\\
k & \textrm{otherwise,}\end{array}\right.\]

\[
M(\alpha_{ij})=\left\{ \begin{array}{ll}
(0,\id) & \textrm{if $(i,j)=(s,t-1)$;}\\
{{\id \choose 0}} & \textrm{if $(i,j)=(s,t)$;}\\
-\id & \textrm{if $(i,s,t)\in\{(1,3,1),(2,3,1),(3,2,1)\}$;}\\
\id & \textrm{otherwise}\end{array}\right.\]
 because this gives an indecomposable $A$-module of dimension vector
$\delta+e_{y}$. Now since $\soc M\cong S_{1}\oplus S_{y}$, and $\top M\cong S_{y}\oplus S_{\infty}$,
we have

\[
\mathcal{F}_{S_{y},*}^{M}=\{ N\}=\mathcal{F}_{S_{y},X_{s,t-1}}^{M}\textrm{, and }\mathcal{F}_{*,S_{y}}^{M}=\{ S\}=\mathcal{F}_{X_{st},S_{y}}^{M}\]
for some $N\cong X_{s,t-1}$ and $S\cong S_{y}$. Therefore we have
the following formula: For any $K\in\Omega$,

\begin{eqnarray}
{[u_{X_{ij}^{K}},u_{S_{x_{st}}^{K}}]} & = & \left\{ \begin{array}{ll}
u_{M^{K}} & \textrm{if $(i,j)=(s,t)$;}\\
-u_{M^{K}} & \textrm{if $(i,j)=(s,t-1)$;}\\
0 & \textrm{otherwise; and}\end{array}\right.\label{eq:u-u}\\
{[u_{W_{c}(K)},u_{S_{x_{st}}^{K}}]} & = & 0\textrm{ for all $c\in K\setminus E_{\Delta}$}\nonumber \end{eqnarray}
in $\overline{L}(A^{K})/(|K|-1)$, where $E_{\Delta}$ is as in (\ref{eq:zyogai}).

\paragraph*{Case 1.1. $x\neq1$, say $x=x_{ij}$ for some $i,j$.}

In this case $\eta_{x}$ has the form (\ref{eq:et-ij}).

\subparagraph*{Case 1.1.1. $x=y$.}

In this case $a_{xy}=2$. By the formula (\ref{eq:u-u}) we have

\begin{eqnarray*}
{[u_{X_{ij}},u_{S_{x_{ij}}}]} & = & u_{M}\\
{[u_{X_{i,j-1}},u_{S_{X_{ij}}}]} & = & -u_{M}\end{eqnarray*}
Therefore $[u_{X_{ij}}-u_{X_{i,j-1}},u_{S_{x_{ij}}}]=2u_{M}$. This
shows that $[\eta_{x},\varepsilon_{y}]=2\varepsilon_{y}=a_{xy}\varepsilon_{y}$.

\subparagraph*{Case 1.1.2. $x\neq y$ and $x,y$ are not neighbors in $\Delta$.}

In this case $a_{xy}=0$. Again by (\ref{eq:u-u}) we have $[u_{X_{ij}},u_{X_{y}}]=0=[u_{X_{i,j-1}},u_{S_{y}}]$.
Thus $[\eta_{x},\varepsilon_{y}]=0=a_{xy}\varepsilon_{y}$.

\subparagraph*{Case 1.1.3. $x\neq y$ and $x,y$ are neighbors in $\Delta$.}

In this case $a_{xy}=-1$, and $x\in\{ x_{s,t-1},x_{s,t+1}\}$. When
$x=x_{s,t-1}$, it follows from $[u_{X_{s,t-1}},u_{S_{x_{st}}}]=-u_{M}$
and $[u_{X_{s,t-2}},u_{S_{x_{st}}}]=0$ (by (\ref{eq:u-u})) that
$[\eta_{x},\varepsilon_{y}]=-\varepsilon_{y}=a_{xy}\varepsilon_{y}$.
When $x=x_{s,t+1}$, it follows from $[u_{X_{s,t+1}},u_{S_{x_{st}}}]=0$
and $[u_{X_{st}},u_{S_{x_{st}}}]=u_{M}$ that $[\eta_{x},\varepsilon_{y}]=-\varepsilon_{y}=a_{xy}\varepsilon_{y}$.

\paragraph*{Case 1.2. $x=1$.}

In this case $x\neq y$ and $\eta_{x}$ has the form (\ref{eq:et-1}).

\subparagraph*{Case 1.2.1. $x,y$ are not neighbors in $\Delta$.}

Then $a_{xy}=0$ and $t\geq3$. By (\ref{eq:u-u}) we have $[\eta_{x},\varepsilon_{y}]=0=a_{xy}\varepsilon_{y}$.

\subparagraph*{Case 1.2.2. $x,y$ are neighbors in $\Delta$.}

Then $a_{xy}=-1$ and $t=2$. By (\ref{eq:u-u}) we have $[\eta_{x},\varepsilon_{y}]=-\varepsilon_{y}=a_{xy}\varepsilon_{y}$.

\subsubsection*{Case 2. $y=1$.}

In this case we may assume that $M$ has the following structure:

\[
M(z)=\left\{ \begin{array}{ll}
k^{2} & \textrm{if $z=1$; }\\
k & \textrm{otherwise,}\end{array}\right.\]

\[
M(\alpha_{ij})=\left\{ \begin{array}{ll}
{{\id \choose 0}} & \textrm{if $(i,j)=(1,1)$;}\\
{{0 \choose \id}} & \textrm{if $(i,j)=(2,1)$;}\\
{-{{\id \choose \id}}} & \textrm{if $(i,j)=(3,1)$;}\\
\id & \textrm{otherwise.}\end{array}\right.\]
Then as easily seen for any $K\in\Omega$ we have

\begin{eqnarray}
F_{X,S_{1}^{K}}^{M^{K}} & = & \left\{ \begin{array}{ll}
1 & \textrm{if $X\cong W_{c}(K)$ (for some $c\in K)$ or $X_{11}^{K}$;}\\
0 & \textrm{otherwise.}\end{array}\right.\nonumber \\
F_{S_{1}^{K},X}^{M^{K}} & = & 0\textrm{\quad for all incecomposable }X\textrm{ with $\udim X=\delta$}.\label{eq:F-MXS}\end{eqnarray}

\paragraph*{Case 2.1. $x\neq1$, say $x=x_{ij}$.}

Then $x\neq y$ and $\eta_{x}$ has the form (\ref{eq:et-ij}).

\subparagraph*{Case 2.1.1. $x,y$ are not neighbors in $\Delta$.}

In this case $a_{xy}=0$ and $j\geq3$. By (\ref{eq:F-MXS}) we have
$[\eta_{x},\varepsilon_{y}]=0=a_{xy}\varepsilon_{y}$.

\subparagraph*{Case 2.1.2. $x,y$ are neighbors in $\Delta$.}

In this case $a_{xy}=-1$ and $x=x_{i2}$. By (\ref{eq:F-MXS}) we
have $[u_{X_{i2}},u_{S_{1}}]=0$ and $[u_{X_{i1}},u_{S_{1}}]=u_{M}$
and hence $[\eta_{x},\varepsilon_{y}]=-\varepsilon_{y}=a_{xy}\varepsilon_{y}$.

\paragraph*{Case 2.2. $x=1$.}

Then $x=y=1$ and $a_{xy}=2$. By (\ref{eq:F-MXS}) we have for any
$K\in\Omega$

\[
[\sum_{c\in K}u_{W_{c}(K)}+u_{X_{11}^{K}},u_{S_{1}^{K}}]=(|K|+1)u_{M^{K}}=2u_{M^{K}}\]
in $\overline{L}(A^{K})/(|K|-1)$. Therefore $[\eta_{1},\varepsilon_{1}]=2\varepsilon_{1}=a_{11}\varepsilon_{1}$.

\subsection*{Verification of (\ref{et-ze})}

Set $f_{x}:=\delta-e_{x}$ for all $x\in Q_{0}$. First we show the
following.

\begin{lem}
\label{lem:fx-delta}For each $x\in\Delta_{0}$ we have $\overline{\mathbf{u}}_{m(f_{x}+\delta)}=\overline{\mathbf{u}}_{m(f_{x})}$.
\end{lem}
\begin{proof}
If $x\ne1$, then $f_{x}+\delta$ is a positive root of $\chi_{A}$
with $M(f_{x}+\delta)$ a regular module, and $\org(f_{x}+\delta)=f_{x}$.
Hence the assertion holds by Proposition \ref{prp:dif-rg}. Now let
$x=1$. First we assume that $\Delta$ is not of type $A_{n}$. We
may assume that the module $M:=M(f_{1}+\delta)$ has the following
structure:\begin{eqnarray}
M(z) & = & \begin{cases}
k & \textrm{if }z=1;\\
k^{2} & \textrm{otherwise.}\end{cases}\nonumber \\
M(\alpha_{ij}) & = & \begin{cases}
(1,0) & \textrm{if }(i,j)=(1,1);\\
(0,1) & \textrm{if }(i,j)=(2,1);\\
-(1,1) & \textrm{if }(i,j)=(3,1);\\
\id_{k^{2}} & \textrm{otherwise.}\end{cases}\label{eq:str-of-f1+delta}\end{eqnarray}
By looking at this structure we easily see that the following equalities
hold:\begin{multline*}
\mathbf{u}_{m(f_{1}+\delta)}=[\mathbf{u}_{m(e_{\infty}+\delta)},\mathbf{u}_{m(e_{x_{32}})},\mathbf{u}_{m(e_{x_{2,p(2)}})},\mathbf{u}_{m(e_{x_{2,p(2)-1}})},\ldots,\mathbf{u}_{m(e_{x_{22}})},\\
\mathbf{u}_{m(e_{x_{1,p(1)}})},\mathbf{u}_{m(e_{x_{1,p(1)-1}})},\ldots,\mathbf{u}_{m(e_{x_{12}})}];\textrm{ and}\end{multline*}
\begin{multline*}
\mathbf{u}_{m(f_{1})}=[\mathbf{u}_{m(e_{\infty})},\mathbf{u}_{m(e_{x_{32}})},\mathbf{u}_{m(e_{x_{2,p(2)}})},\mathbf{u}_{m(e_{x_{2,p(2)-1}})},\ldots,\mathbf{u}_{m(e_{x_{22}})},\\
\mathbf{u}_{m(e_{x_{1,p(1)}})},\mathbf{u}_{m(e_{x_{1,p(1)-1}})},\ldots,\mathbf{u}_{m(e_{x_{12}})}].\end{multline*}
Hence we have\begin{multline*}
\mathbf{u}_{m(f_{1}+\delta)}-\mathbf{u}_{m(f_{1})}=[\mathbf{u}_{m(e_{\infty}+\delta)}-\mathbf{u}_{m(e_{\infty})},\mathbf{u}_{m(e_{x_{32}})},\mathbf{u}_{m(e_{x_{2,p(2)}})},\mathbf{u}_{m(e_{x_{2,p(2)-1}})},\ldots,\mathbf{u}_{m(e_{x_{22}})},\\
\mathbf{u}_{m(e_{x_{1,p(1)}})},\mathbf{u}_{m(e_{x_{1,p(1)-1}})},\ldots,\mathbf{u}_{m(e_{x_{12}})}]\end{multline*}
is in $I(A)^{\mathbb{C}}$ because so is $\mathbf{u}_{m(e_{\infty}+\delta)}-\mathbf{u}_{m(e_{\infty})}$
by Proposition \ref{prp:dif-smp}.

Finally a similar argument works also in the $A_{n}$ case.
\end{proof}

\subsubsection*{Case 1. $y\ne1$, say $y=x_{ij}$ $(i\in\{1,\ldots,r\},j\in\{2,\ldots,p(i)\})$.}

In this case $\zeta_{y}=\overline{\mathbf{u}}_{m(f_{y})}$ and $M(f_{y})$
is a regular exceptional module contained in a non-homogeneous tube
$\mathcal{T}_{\rho}$. Note that for each $s\in\{1,\ldots,r\}$ and
$t\in\{2,\ldots,p(s)\}$ we have $X_{st}\in\mathcal{T}_{\rho}$ if
and only if $s=i$. By Lemma \ref{lem:distinct-tubes} we have\[
\begin{cases}
[u_{W_{c}(K)},u_{m(f_{y})^{K}}]=0 & \forall c\in K\setminus E_{\Delta}\textrm{; and}\\
{}[u_{X_{xt}^{K}},u_{m(f_{y})^{K}}]=0 & \forall s\in\{1,\ldots,r\}\setminus\{ i\},\forall t\in\{2,\ldots,p(s)\}.\end{cases}\]
If $s=i$, then we can calculate the bracket $[u_{X_{it}^{K}},u_{m(f_{y})^{K}}]$
in $\mathcal{T}_{\rho}$ using (\ref{eq:Lie-bracket-cyclic}) as follows.\[
[u_{X_{it}^{K}},u_{m(f_{y})^{K}}]=\begin{cases}
-u_{m(f_{y}+\delta)^{K}} & \textrm{if }j=t;\\
u_{m(f_{y}+\delta)^{K}} & \textrm{if }j=t+1;\textrm{ and}\\
0 & \textrm{otherwise.}\end{cases}\]
Using this we now calculate $[\eta_{x},\zeta_{y}]$.

\paragraph*{Case 1.1. $x=1$.}

For any $K\in\Omega$, we have\begin{eqnarray*}
[\sum_{c\in K\setminus E_{\Delta}}u_{W_{c}(K)}+u_{X_{11}^{K}}+\cdots+u_{X_{r1}^{K}},u_{m(f_{y})^{K}}] & = & [u_{X_{i1}^{K}},u_{m(f_{y})^{K}}]\\
 & = & \begin{cases}
-u_{m(f_{y}+\delta)^{K}} & \textrm{if }j=1\textrm{ (i.e., }y=1,\textrm{ impossible})\\
u_{m(f_{y}+\delta)^{K}} & \textrm{if }j=2\textrm{ (i.e., }y=x_{i2})\\
0 & \textrm{otherwise (i.e., }y=x_{ij},j\ge3)\end{cases}\\
 & = & -a_{xy}u_{m(f_{y}+\delta)^{K}}.\end{eqnarray*}
Hence we have $[\eta_{x},\zeta_{y}]=-a_{xy}\zeta_{y}$ by Lemma \ref{lem:fx-delta}.

\paragraph*{Case 1.2. $x\ne1$, say $x=x_{st}$$(s\in\{1,\ldots,r\}$ and $t\in\{2,\ldots,p(s)\})$.}

Let $K\in\Omega$. If $s\ne i$, then $a_{xy}=0$ and $[u_{X_{st}^{K}}-u_{X_{s,t-1}^{K}},u_{m(f_{y})^{K}}]=0=-a_{xy}u_{m(f_{y})^{K}}$.
Thus in this case we have $[\eta_{x},\zeta_{y}]=-a_{xy}\zeta_{y}$.
Therefore we may assume that $s=i$.\begin{eqnarray*}
[u_{X_{st}^{K}}-u_{X_{s,t-1}^{K}},u_{m(f_{y})^{K}}] & = & \begin{cases}
-u_{m(f_{y}+\delta)^{K}} & \textrm{if }j=t\\
u_{m(f_{y}+\delta)^{K}} & \textrm{if }j=t+1\\
0 & \textrm{otherwise}\end{cases}\\
 & + & \begin{cases}
-u_{m(f_{y}+\delta)^{K}} & \textrm{if }j=t-1\\
u_{m(f_{y}+\delta)^{K}} & \textrm{if }j=t\\
0 & \textrm{otherwise}\end{cases}\\
 & = & -a_{xy}u_{m(f_{y}+\delta)^{K}}.\end{eqnarray*}
Hence also in this case we have $[\eta_{x},\zeta_{y}]=-a_{xy}\zeta_{y}$
by Lemma \ref{lem:fx-delta}.

\subsubsection*{Case 2. $y=1$.}

In this case $\zeta_{1}=-\overline{\mathbf{u}}_{m(f_{1})}$. Using
the structure of $M(f_{1}+\delta)$ described in (\ref{eq:str-of-f1+delta})
we can calculate Hall numbers as follows: For any $K\in\Omega$\[
F_{M(f_{1})^{K}X_{ij}^{K}}^{M(f_{1}+\delta)^{K}}=\begin{cases}
1 & \textrm{if }j=1;\\
0 & \textrm{otherwise.}\end{cases}\]
\[
F_{M(f_{1})^{K}W_{c}(K)}^{M(f_{1}+\delta)^{K}}=1\quad\textrm{ for all }c\in K.\]
Further since there are no nonzero homomorphisms from preinjectives
to regulars, we have \[
F_{X,M(f_{1})^{K}}^{M(f_{1}+\delta)^{K}}=0\quad\textrm{ for all indecomposable }X\textrm{ with }\udim X=\delta.\]
These enable us to calculate $[\eta_{x},\zeta_{1}]$.

\paragraph*{Case 2.1. $x=1$.}

In this case $a_{11}=2$. For any $K\in\Omega$ we have\[
[\sum_{c\in K}u_{W_{c}(K)}+u_{X_{11}^{K}},u_{m(f_{1})^{K}}]=-(|K|+1)u_{m(f_{1}+\delta)^{K}}=-2u_{m(f_{1}+\delta)^{K}}\]
in $\overline{L}(A^{K})/(|K|-1)$. Hence by Lemma \ref{lem:fx-delta}
we have $[\eta_{1},\zeta_{1}]=-a_{11}\zeta_{1}$.

\paragraph*{Case 2.2. $x,y$ are neighbors in $\Delta$, i.e., $x=x_{ij}$ with
$j=2$.}

In this case $a_{x1}=-1$. For any $K\in\Omega$ we have\[
[u_{X_{i2}^{K}}-u_{X_{i1}^{K}},u_{m(f_{1})^{K}}]=u_{m(f_{1}+\delta)^{K}}.\]
Hence by Lemma \ref{lem:fx-delta} we have $[\eta_{x},\zeta_{1}]=-a_{x1}\zeta_{1}$.

\paragraph*{Case 2.3. $x,y$ are not neighbors in $\Delta$, i.e., $x=x_{ij}$
with $j\ge3$.}

In this case $a_{x1}=0$. For any $K\in\Omega$ we have\[
[u_{X_{ij}^{K}}-u_{X_{i,j-1}^{K}},u_{m(f_{1})^{K}}]=0.\]
Hence by Lemma \ref{lem:fx-delta} we have $[\eta_{x},\zeta_{1}]=-a_{x1}\zeta_{1}$.
This finishes the proof of the claim.

\subsection{Injectivity of $\mathbf{\phi}$\label{sub:Injectivity-of-phi}}

We next show that $\phi:\mathfrak{g}(\Delta)\to L(A)_{1}^{\mathbb{C}}/I(A)^{\mathbb{C}}$
is injective. Since $\mathfrak{g}(\Delta)$ is simple, it is enough
to show that $\Im\phi\ne0$. First we consider the case that $\Delta\ne A_{1}$.
In this case $x_{12}$ exists, and we set $e_{2}:=e_{x_{12}}$, $E_{2}:=E_{x_{12}}$
for short. Then it is enough to show that $\phi(E_{2})\ne0$, i.e.,
that $\mathbf{u}_{m(e_{2})}\otimes1_{\mathbb{C}}\not\in I(A)^{\mathbb{C}}$,
where $1_{\mathbb{C}}$ stands for the identity $1\in\mathbb{C}$.
Denote by $1_{\mathbb{Q}}$ the identity $1\in\mathbb{Q}$. Since
the canonical isomorphism $(L(A)_{1}^{\mathbb{Q}}/I(A)^{\mathbb{Q}})^{\mathbb{C}}\isoto L(A)_{1}^{\mathbb{C}}/I(A)^{\mathbb{C}}$
sends the coset of $\mathbf{u}_{m(e_{2})}\otimes1_{\mathbb{Q}}$ to
that of $\mathbf{u}_{m(e_{2})}\otimes1_{\mathbb{C}}$, we have only
to show that $\mathbf{u}_{m(e_{2})}\otimes1_{\mathbb{Q}}\not\in I(A)^{\mathbb{Q}}$.
Assume that $\mathbf{u}_{m(e_{2})}=\mathbf{u}_{m(e_{2})}\otimes1_{\mathbb{Q}}\in I(A)^{\mathbb{Q}}$.
Then by definition of $I(A)$ there exist a finite set $J\subseteq\bigcup_{n\in\mathbb{N}}Q_{0}^{(n)}$and
an $(a_{i})_{i\in J}\in\mathbb{Q}^{J}$ such that $\mathbf{u}_{m(e_{2})}$
is expressed as a linear combination

\[
\mathbf{u}_{m(e_{2})}=\sum_{i\in J}a_{i}[\mathbf{u}_{m(e_{i(1)})},\mathbf{u}_{m(e_{i(2)})},\ldots,\mathbf{u}_{m(e_{i(t_{i}-1)})},\mathbf{u}_{m(e_{i(t_{i})}+\delta)}-\mathbf{u}_{m(e_{i(t_{i})})}],\]
where we put $i=(i(1),\ldots,i(t_{i}))$ for all $i\in J$. Take an
$a\in\mathbb{N}$ such that $aa_{i}\in\mathbb{Z}$ for all $i\in J$.
By renaming $aa_{i}$ as $a_{i}$, we have\[
a\mathbf{u}_{m(e_{2})}=\sum_{i\in J}a_{i}[\mathbf{u}_{m(e_{i(1)})},\mathbf{u}_{m(e_{i(2)})},\ldots,\mathbf{u}_{m(e_{i(t_{i}-1)})},\mathbf{u}_{m(e_{i(t_{i})}+\delta)}-\mathbf{u}_{m(e_{i(t_{i})})}],\]
with $a_{i}\in\mathbb{Z}$ for all $i\in J$. We put

\[
d_{i}:=\sum_{j=1}^{t_{i}}e_{i(j)}\quad\textrm{and}\quad e(i):=e_{i(t_{i})}\]
for all $i\in J$. Then for each $K\in\Omega$ we have\begin{align*}
au_{m(e_{2})^{K}}= & \sum_{i\in J}a_{i}[u_{m(e_{i(1)})^{K}},u_{m(e_{i(2)})^{K}},\ldots,u_{m(e_{i(t_{i}-1)})^{K}},u_{m(e(i)+\delta)^{K}}-u_{m(e(i))^{K}}]\\
= & \sum_{i\in J_{1}}a_{i}[u_{m(e_{i(1)})^{K}},u_{m(e_{i(2)})^{K}},\ldots,u_{m(e_{i(t_{i}-1)})^{K}},u_{m(e(i)+\delta)^{K}}-u_{m(e(i))^{K}}]\\
+ & \sum_{i\in J_{2}}a_{i}[u_{m(e_{i(1)})^{K}},u_{m(e_{i(2)})^{K}},\ldots,u_{m(e_{i(t_{i}-1)})^{K}},u_{m(e(i)+\delta)^{K}}-u_{m(e(i))^{K}}],\end{align*}
in $\overline{L}(A^{K})/(|K|-1)$ where $J_{1}:=\{ i\in J\mid d_{i}\in\mathbb{Z}\delta+e_{2}\}$
and $J_{2}:=\{ i\in J\mid d_{i}\not\in\mathbb{Z}\delta+e_{2}\}$.
Here since $\overline{L}(A^{K})/(|K|-1)$ is a free $\mathbb{Z}/(|K|-1)$-module
with basis $[\ind A^{K}]$ which is a disjoint union of the subsets
$\{[X]\in[\ind A^{K}]\mid\udim X\in\mathbb{Z}\delta+e_{2}\}$ and
$\{[X]\in[\ind A^{K}]\mid\udim X\not\in\mathbb{Z}\delta+e_{2}\}$,
we have\[
au_{m(e_{2})^{K}}=\sum_{i\in J_{1}}a_{i}[u_{m(e_{i(1)})^{K}},u_{m(e_{i(2)})^{K}},\ldots,u_{m(e_{i(t_{i}-1)})^{K}},u_{m(e(i)+\delta)^{K}}-u_{m(e(i))^{K}}].\]
We can write $d_{i}=s_{i}\delta+e_{2}$ for some $s_{i}\in\mathbb{Z}$.
Now we have\[
[u_{m(e_{i(1)})^{K}},u_{m(e_{i(2)})^{K}},\ldots,u_{m(e_{i(t_{i}-1)})^{K}}]\begin{cases}
\in(\mathbb{Z}/(|K|-1)\mathbb{Z})u_{m(d_{i}-e(i))} & \textrm{if}\:\chi(d_{i}-e(i))\in\{0,1\}\\
=0 & \textrm{otherwise.}\end{cases}\]
Here, \begin{align*}
\chi(d_{i}-e(i)) & =B(s_{i}\delta+e_{2}-e(i),s_{i}\delta+e_{2}-e(i))\\
 & =s_{i}^{2}B(\delta,\delta)+s_{i}\{ B(\delta,e_{2}-e(i))+B(e_{2}-e(i),\delta)\}+B(e_{2}-e(i),e_{2}-e(i))\\
 & =B(e_{2},e_{2})-\{ B(e_{2},e(i))+B(e(i),e_{2})\}+B(e(i),e(e))\\
 & \in2-\{2,0,-1\}=\{0,2,3\}.\end{align*}
Hence $\chi(d_{i}-e(i))\in\{0,1\}$ $\Leftrightarrow$ $\chi(d_{i}-e(i))=0$
$\Leftrightarrow$$B(e_{2},e(i))+B(e(i),e_{2})=2$ $\Leftrightarrow$
$e(i)=e_{2}$. Therefore by putting $J_{0}:=\{ i\in J_{1}\mid e(i)=e_{2}\}=\{ i\in J\mid d_{i}\in\mathbb{Z}\delta+e_{2},e(i)=e_{2}\}$
we have\[
au_{m(e_{2})^{K}}=\sum_{i\in J_{0}}a_{i}[u_{m(e_{i(1)})^{K}},u_{m(e_{i(2)})^{K}},\ldots,u_{m(e_{i(t_{i}-1)})^{K}},u_{m(e_{2}+\delta)^{K}}-u_{m(e_{2})^{K}}].\]
For each $i\in J_{0}$ set $I_{i}:=\{[X]\in[\ind A^{K}]\mid\udim X=s_{i}\delta\}$.
Then noting that $d_{i}-e_{2}=s_{i}\delta$ for all $i\in J_{0}$,
we see for each $i\in J_{0}$ and for each $[X]\in I_{i}$ there exists
some $b_{i,[X]}\in\mathbb{Z}/(|K|-1)\mathbb{Z}$ such that\[
[u_{m(e_{i(1)})^{K}},u_{m(e_{i(2)})^{K}},\ldots,u_{m(e_{i(t_{i}-1)})^{K}}]=\sum_{[X]\in I_{i}}b_{i,[X]}u_{[X]}.\]
Then\begin{align*}
[u_{m(e_{i(1)})^{K}},u_{m(e_{i(2)})^{K}},\ldots,u_{m(e_{i(t_{i}-1)})^{K}},u_{m(e_{2}+\delta)^{K}}] & =\sum_{[X]\in I_{i}}b_{i,[X]}[u_{[X]},u_{m(e_{2}+\delta)^{K}}]\\
 & =\sum_{[X]\in I_{i}}b_{i,[X]}b_{[X],m(e_{2}+\delta)^{K}}u_{m(e_{2}+(s_{i}+1)\delta)^{K}},\end{align*}
and\begin{align*}
[u_{m(e_{i(1)})^{K}},u_{m(e_{i(2)})^{K}},\ldots,u_{m(e_{i(t_{i}-1)})^{K}},u_{m(e_{2})^{K}}] & =\sum_{[X]\in I_{i}}b_{i,[X]}[u_{[X]},u_{m(e_{2})^{K}}]\\
 & =\sum_{[X]\in I_{i}}b_{i,[X]}b_{[X],m(e_{2})^{K}}u_{m(e_{2}+s_{i}\delta)^{K}}.\end{align*}
Here by Proposition \ref{pro:img-stb}, we have $b_{[X],m(e_{2}+\delta)^{K}}=b_{[X],m(e_{2})^{K}}$
for all $[X]\in I_{i}$. Hence we obtain\[
[u_{m(e_{i(1)})^{K}},u_{m(e_{i(2)})^{K}},\ldots,u_{m(e_{i(t_{i}-1)})^{K}},u_{m(e_{2}+\delta)^{K}}-u_{m(e_{2})^{K}}]=c_{i}(K)(u_{m(e_{2}+(s_{i}+1)\delta)^{K}}-u_{m(e_{2}+s_{i}\delta)^{K}}),\]
where we put $c_{i}(K):=\sum_{[X]\in I_{i}}b_{i,[X]}b_{[X],m(e_{2})^{K}}\in\mathbb{Z}/(|K|-1)\mathbb{Z}$.
As a consequence, we have\[
au_{m(e_{2})^{K}}=\sum_{i\in J_{0}}a_{i}c_{i}(K)(u_{m(e_{2}+(s_{i}+1)\delta)^{K}}-u_{m(e_{2}+s_{i}\delta)^{K}}).\]
From this formula, it is easy to see that $a=0$ in $\mathbb{Z}/(|K|-1)\mathbb{Z}$
for all $K\in\Omega$. Hence $a=0$ in $\mathbb{Z}$, a contradiction.
Hence we must have $\mathbf{u}_{m(e_{2})}\not\in I(A)^{\mathbb{Q}}$,
and hence $\Im\phi\ne$0.

Also in the case that $\Delta=A_{1}$, a similar argument works by
Proposition \ref{pro:Krn-stb} to show that $\Im\phi\ne$0.

\subsection{Surjectivity of $\mathbf{\phi}$}

We finally show that $\phi$ is surjective. It is enough to show that
$\varepsilon_{\infty}\in\Im\phi$ because $L(A)_{1}^{\mathbb{C}}/I(A)^{\mathbb{C}}$
is generated by $\{\varepsilon_{x}\mid x\in Q_{0}\}$ and we already
know that $\{\varepsilon_{x}\mid x\in Q_{0}\setminus\{\infty\}\}\subseteq\Im\phi$
by definition of $\phi$. There exists a permutation $(x_{1},\ldots,x_{n})$
of $\Delta_{0}$ such that $[F_{x_{1}},\cdots,F_{x_{n}}]\neq0$ in
$\mathfrak{g}(\Delta)$ (see \ref{sub:notations}). Thus by \ref{sub:Injectivity-of-phi}
we have $[\zeta_{x_{1}},\ldots,\zeta_{x_{n}}]\neq0$ in $L(A)_{1}^{\mathbb{C}}/I(A)^{\mathbb{C}}$.
Set $f_{x}:=\delta-e_{x}=\udim T_{x}$ for all $x\in\Delta_{0}$ and
note that

\begin{equation}
\sum_{x\in\Delta_{0}}f_{x}=\sum_{x\in\Delta_{0}}(\delta-e_{x})=(n-1)\delta+e_{\infty}.\label{eq:sum-T-x}\end{equation}
Since $B(\delta,e_{\infty})=1$, $B(e_{\infty},\delta)=-1$ and $B(e_{\infty},e_{\infty})=1$
we have $\chi((n-1)\delta+e_{\infty})=1$. Hence there exists a unique
indecomposable $A$-module $M$ with $\udim M=(n-1)\delta+e_{\infty}$
up to isomorphisms. Hence we have\[
[\zeta_{x_{1}},\ldots,\zeta_{x_{n}}]=\overline{(c(K)u_{[M^{K}]})_{K\in\Omega}}\]
for some $c(K)\in\mathbb{Z}/(|K|-1)\mathbb{Z}$ for each $K\in\Omega$.
By Corollary \ref{cor:prin-Hall}, there exist Hall polynomials\[
\varphi_{M(f_{x_{1}})M(f_{x_{2}})}^{M(f_{x_{1}}+f_{x_{2}})},\varphi_{M(f_{x_{3}})M(f_{x_{1}}+f_{x_{2}})}^{M(f_{x_{1}}+f_{x_{2}}+f_{x_{3}})},\ldots,\varphi_{M(f_{x_{n}})M(\sum_{i=1}^{n-1}f_{x_{i}})}^{M}.\]
Therefore there is some $c\in\mathbb{Z}$ such that $c\equiv c(K)\pmod{|K|-1}$
for all $K\in\Omega$. Thus $0\ne[\zeta_{x_{1}},\ldots,\zeta_{x_{n}}]=c\overline{\mathbf{u}}_{M}$
and we have $c\ne0$. By Proposition \ref{prp:dif-smp}, we have $\overline{\mathbf{u}}_{M}=\overline{\mathbf{u}}_{m(e_{\infty})}=\varepsilon_{\infty}$.
Hence $\varepsilon_{\infty}=\frac{1}{c}[\zeta_{x_{1}},\ldots,\zeta_{x_{n}}]\in\Im\phi$.

As a consequence, $\phi\colon\mathfrak{g}(\Delta)\to L(A)_{1}^{\mathbb{C}}/I(A)^{\mathbb{C}}$
is an isomorphism. \hfil\qed

\section{\textbf{Root spaces}\label{sec:root-spaces} \textbf{}}

In this section we prove Proposition \ref{prp:root-space}.

\subsection{Gabriel-Roiter submodules}

We first recall the definitions of the Gabriel-Roiter measure and
of Gabriel-Roiter submodules (see Ringel \cite{Ri-GR} for details).

\begin{defn}
Let $M\in\mod A$ and $l$ the length of $M$. 

(1) The \emph{Gabriel-Roiter measure} $\mu(M)\in\mathbb{Q}$ of $M$
is defined by induction on $l$ as follows. If $l=0$, then $\mu(M):=0$.
If $l>0$, then\[
\mu(M):=\max_{M'<M}\mu(M')+\begin{cases}
2^{-l} & \textrm{if $M$ is indecomposable;}\\
0 & \textrm{otherwise,}\end{cases}\]
where $M'$ runs through all proper submodules of $M$ (for the existence
of this maximum see \cite[Section 1]{Ri-GR}).

(2) If $M$ is indecomposable and $M'$ is an indecomposable submodule
of $M$ with $\mu(M')$ maximal, then we call $M'$ a \emph{Gabriel-Roiter}
\emph{submodule} (\emph{GR-submodule} for short) of $M$ and the embedding
$M'\hookrightarrow M$ a \emph{Gabriel-Roiter inclusion} (\emph{GR-inclusion}
for short).

(3) A monomorphism $u$ is called \emph{mono-irreducible} if (i) $u$
is not a section, and (ii) for every factorization $u=u''u'$ with
$u''$ a monomorphism, either $u'$ is a section or $u''$ is an isomorphism.
\end{defn}
We cite the following from \cite[Section 2]{Ri-GR}.

\begin{prop}
\label{prp:GR-monoirr} $(1)$ GR-inclusions are mono-irreducible.

$(2)$ The cokernel of a mono-irreducible monomorphism between indecomposable
$A$-modules is indecomposable.

$(3)$ Let $Y$ be an indecomposable $A$-module with a GR-submodule
$X$ and let $U$ be a submodule of $Y$ isomorphic to $X$. Then
$U$ is also a GR-submodule of $Y$, and hence $Y/U$ is again indecomposable
by $(1)$ and $(2)$ above.
\end{prop}

\begin{lem}
Let $Y$ be an indecomposable $A$-module with a GR-submodule $X$.
If $\udim Y/X\not\in\mathbb{Z}\delta$, then $\mathcal{F}_{*,X}^{Y}=\mathcal{F}_{Y/X,X}^{Y}$.
\end{lem}
\begin{proof}
It is enough to show that $\mathcal{F}_{*,X}^{Y}\subseteq\mathcal{F}_{Y/X,X}^{Y}$.
Let $U\in\mathcal{F}_{*,X}^{Y}$. Then since $Y\ge U\cong X$, both
$Y/X$ and $Y/U$ are indecomposable by the proposition above. But
since $\udim Y/U=\udim Y/X\not\in\mathbb{Z}\delta$, we have $Y/U\cong Y/X$.
Hence $U\in\mathcal{F}_{Y/X,X}^{Y}$.
\end{proof}
\begin{prop}
\label{pro:preproj-GR-Hall}Let $Y$ be a preprojective indecomposable
$A$-module with a GR-submodule $X$. If $\udim Y/X\not\in\mathbb{Z}\delta$,
then there exists a Hall polynomial $\varphi_{Y/X,X}^{Y}=\varphi_{*,X}^{Y}$.
\end{prop}
\begin{proof}
Since $Y$ is preprojective, so is $X$. Then by Lemma \ref{lem:Hall-sub-factor}
the polynomial $\varphi_{*,X}^{Y}$ exists. By the lemma above we
have $\varphi_{Y/X,X}^{Y}=\varphi_{*,X}^{Y}$.
\end{proof}
\begin{prop}
\label{pro:GR-M}Let $X$ and $Y$ be indecomposable preprojective
$A$-modules. Assume that $X$ is a GR-submodule of $Y$. If $\rank Y\ge2$,
then
\end{prop}
\begin{enumerate}
\item $\udim Y/X\not\in\mathbb{Z}\delta$; and
\item There exists a Hall polynomial $\varphi_{Y/X,X}^{Y}$.
\end{enumerate}
\begin{proof}
Assume that $\rank Y\ge2$. Then $\Delta$ is not of type $A_{n}$. 

(1) Assume that $\udim Y/X\in\mathbb{Z}\delta$ and set $v:=\udim X-\dim X(\infty)\delta$.
Then $v_{\infty}=0$, $X=M(v+s\delta)$ and $Y=M(v+t\delta)$ for
some $s<t$ in $\mathbb{N}$, and $\rank X=v_{1}=\rank Y$. For each
$r\in\mathbb{N}$ we may assume that $M:=M(v+r\delta)$ has the following
structure by \cite[Theorems 2, 3]{KM}: $M(x_{ij})=k^{v_{x_{ij}}+r}$
for all $x_{ij}\in Q_{0}$; $M(\alpha_{1j})$ has the form $\left[\begin{array}{c}
\id_{v_{x_{1,j+1}}+r}\\
0\end{array}\right]$ for all $1\le j\le p(1)$; $M(\alpha_{2j})$ has the form $\left[\begin{array}{c}
0\\
\id_{v_{x_{2,j+1}}+r}\end{array}\right]$ for all $1\le j\le p(2)$; $M(\alpha_{32})$ has the form $\left[\begin{array}{c}
0\\
\id_{r}\end{array}\right]$; and $M(\alpha_{31})=-Z_{r}$, where $Z_{r}$ is the $r$-th \emph{enlargement}
(see \cite[Section 2]{KM} for the definition) of a matrix $Z$ listed
in \cite[Theorem 3, Table 1]{KM} that is determined by $v$ not depending
on $r$ (only here we use the assumption that $\chr k\ne2$). For
all $r<r'$ in $\mathbb{N}$ we can define a monomorphism $f:M(v+r\delta)\to M(v+r'\delta)$
by setting $f_{x}:=\left[\begin{array}{c}
\id_{v_{x}+r}\\
0\end{array}\right]$ for all $x\in Q_{0}$, which we can regard the inclusion $M(v+r\delta)\hookrightarrow M(v+r'\delta)$.
Now if $t-s>1$, then we have strict inclusions of indecomposable
modules $X=M(v+s\delta)\hookrightarrow M(v+(s+1)\delta)\hookrightarrow M(v+t\delta)=Y$,
which contradicts the fact that the inclusion $X\hookrightarrow Y$
is mono-irreducible (Proposition \ref{prp:GR-monoirr}). Hence we
must have $t=s+1$. Then a direct calculation shows that $\Cok f=Y/X$
has the following structure: $(Y/X)(x)=k$ for all $x\in Q_{0}$;
and \[
(Y/X)(\alpha_{ij})=\begin{cases}
0 & \textrm{if $i=1$ and $v_{1j}=v_{1,j+1}$};\\
-\id & \textrm{if $(i,j)=(3,1)$};\\
\id & \textrm{otherwise}\end{cases}\]
for all $\alpha_{ij}\in Q_{1}$. Since $v_{1}=\rank Y\ge2$, we see
$(Y/X)(\alpha_{1j})=0$ for at least two distinct values of $j$,
which shows that $Y/X$ is decomposable, a contradiction to Proposition
\ref{prp:GR-monoirr}. Hence we must have $\udim Y/X\not\in\mathbb{Z}\delta$.

(2) This follows from (1) by Proposition \ref{pro:preproj-GR-Hall}.
\end{proof}

\subsection{Proof of Proposition \ref{prp:root-space}}

We will make full use of the following fundamental facts on simple
Lie algebras below: Let $0\ne x\in\mathfrak{g}(\Delta)_{\alpha}$
and $0\ne y\in\mathfrak{g}(\Delta)_{\beta}$ for some roots $\alpha,\beta$
of $\mathfrak{g}(\Delta)$, and assume that $\alpha+\beta$ is a root
of $\mathfrak{g}(\Delta)$. Then $0\ne[x,y]\in\mathfrak{g}(\Delta)_{\alpha+\beta}$.

First we show that Proposition \ref{prp:root-space} has a slightly
stronger form for a positive root $v$ of $\chi_{A}$ with $M(v)$
a regular module.

\begin{lem}
\label{lem:non-sincere-root}Let $M$ be a non-sincere indecomposable
$A$-module. Then $\mathbf{u}_{[M]}\in L(A)_{1}$ and $0\ne\phi^{-1}(\overline{\mathbf{u}}_{[M]})\in\mathfrak{g}(\Delta)_{\deg M}$.
\end{lem}
\begin{proof}
We already know that $\mathbf{u}_{[M]}\in L(A)_{1}$ by Proposition
\ref{lem:non-sincere}. By induction on $\dim M$ we show that $0\ne\phi^{-1}(\overline{\mathbf{u}}_{[M]})\in\mathfrak{g}(\Delta)_{\deg M}$.
Assume first that $\dim M=1$. Then $\overline{\mathbf{u}}_{[M]}=\varepsilon_{x}$
for some $x\in Q_{0}$. If $x\ne\infty$ then $0\ne\phi^{-1}(\overline{\mathbf{u}}_{[M]})=E_{x}\in\mathfrak{g}(\Delta)_{e_{x}}=\mathfrak{g}(\Delta)_{\deg M}$
and the assertion holds. If $x=\infty$, then we know that $\varepsilon_{\infty}=\frac{1}{c}[\zeta_{x_{1}},\ldots,\zeta_{x_{n}}]$
for some $c\in\mathbb{Z}^{\times}$ and for some permutation $(x_{1},\ldots,x_{n})$
of $\Delta_{0}$ as in the proof of surjectivity of $\phi$. Hence
\[
\phi^{-1}(\overline{\mathbf{u}}_{[M]})=\frac{1}{c}[F_{x_{1}},\ldots,F_{x_{n}}]\in\mathfrak{g}(\Delta)_{e_{\infty}-\delta}\setminus\{0\}=\mathfrak{g}(\Delta)_{\deg M}\setminus\{0\},\]
and the assertion holds in this case. Assume next that $\dim M>1$.
Then as in the proof of Proposition \ref{lem:non-sincere} there is
a non-sincere indecomposable $A$-module $N$ (with $\dim N=\dim M-1$)
and a simple $A$-module $S$ such that $\mathbf{\overline{u}}_{[M]}=\pm[\mathbf{\overline{u}}_{[S]},\mathbf{\overline{u}}_{[N]}]$
in any case. Here $0\ne\phi^{-1}(\overline{\mathbf{u}}_{[S]})\in\mathfrak{g}(\Delta)_{\deg S}$
, and by the induction hypothesis $0\ne\phi^{-1}(\overline{\mathbf{u}}_{[N]})\in\mathfrak{g}(\Delta)_{\deg N}$.
Since $\deg S+\deg N=\deg M$ is a root of $\mathfrak{g}(\Delta)$
by Lemma \ref{lem:deg-root}, we have $0\ne\phi^{-1}(\overline{\mathbf{u}}_{[M]})=\pm[\phi^{-1}(\overline{\mathbf{u}}_{[S]}),\phi^{-1}(\overline{\mathbf{u}}_{[N]})]\in\mathfrak{g}(\Delta)_{\deg M}$.
\end{proof}
\begin{prop}
Let $v$ be a positive root of $\chi_{A}$. If $M(v)$ is regular,
then
\begin{enumerate}
\item $\mathbf{u}_{m(v)}\in L(A)_{1}\setminus I(A)$; 
\item $\phi^{-1}(\overline{\mathbf{u}}_{m(v)})\in\mathfrak{g}(\Delta)_{\deg v}$;
and
\item $\overline{\mathbf{u}}_{m(v+\delta)}=\overline{\mathbf{u}}_{m(v)}$.
\end{enumerate}
\end{prop}
\begin{proof}
There exists a unique regular exceptional module $X$ such that $v':=\udim X$
has the property that $v-v'\in\mathbb{Z}\delta$. Then $\deg v=\deg v'=\deg(v+\delta)$.
By Proposition \ref{prp:dif-rg} both $\mathbf{u}_{m(v)}-\mathbf{u}_{m(v')}$
and $\mathbf{u}_{m(v+\delta)}-\mathbf{u}_{m(v')}$ are in $I(A)$.
By Remark \ref{rem:rg-cmp} we have $\mathbf{u}_{m(v')}\in L(A)_{1}$
because $M(v')$ is non-sincere. Hence $\mathbf{u}_{m(v)},\mathbf{u}_{m(v')},\mathbf{u}_{m(v+\delta)}\in L(A)_{1}$
and $\overline{\mathbf{u}}_{m(v+\delta)}=\overline{\mathbf{u}}_{m(v)}=\overline{\mathbf{u}}_{m(v')}$.
By the lemma above we have $\mathbf{u}_{m(v)}\not\in I(A)$ and $\phi^{-1}(\overline{\mathbf{u}}_{m(v)})=\phi^{-1}(\overline{\mathbf{u}}_{m(v')})\in\mathfrak{g}(\Delta)_{\deg v'}=\mathfrak{g}(\Delta)_{\deg v}$.
\end{proof}
\begin{rem}
The statements above clearly hold also for $v=e_{x}+t\delta$ for
all $x\in Q_{0}$ and $t\in\mathbb{N}_{0}$.
\end{rem}

\subsection*{Proof of Proposition \ref{prp:root-space} in general.}

Let $v$ be a positive root of $\chi_{A}$. We have to prove the following:

\begin{enumerate}
\item $\mathbf{\overline{u}}_{m(v)}\in L(A)_{1}^{\mathbb{C}}/I(A)^{\mathbb{C}}$;
\item $0\ne\phi^{-1}(\mathbf{\overline{u}}_{m(v)})\in\mathfrak{g}(\Delta)_{\deg v}$;
and
\item $\mathbb{C}\mathbf{\bar{u}}_{m(v+\delta)}=\mathbb{C}\mathbf{\bar{u}}_{m(v)}$.
\end{enumerate}
If both (1) and (2) are shown, then we have $\mathbb{C}\mathbf{\overline{u}}_{m(v)}=\phi(\mathfrak{g}(\Delta)_{\deg v})$.
Then the equality $\deg(v+\delta)=\deg v$ proves the statement (3).
Hence it is enough to show the statements (1) and (2) by induction
on $\dim M(v)$. Set $M:=M(v)$. If $\dim M=1$, then $M$ is non-sincere
and both (1) and (2) hold by Lemma \ref{lem:non-sincere-root}. Suppose
next that $\dim M>1$. Assume that both (1) and (2) hold for all positive
root $w$ of $\chi_{A}$ with $\dim M(w)<\dim M$. 

Case 1. $M$ is regular. In this case the assertion is already proved
in the previous proposition.

Case 2. $M$ is preprojective. If $\rank M=1$, then by looking at
the structure of $M$ described in \cite{KM} it is easy to see that
there exists an indecomposable maximal submodule $X$ of $M$. Set
$S:=M/X$. Then $X$ is also preprojective and $\rank X=1$. By setting
$v':=\udim X$ we may write $X=M(v')$. Since $\rank S=\rank M-\rank X=0$,
$S$ is a regular simple $A$-module, and has the form $S=M(e_{x})$
for some $x\in Q_{0}\setminus\{1,\infty\}$. A direct calculation
shows that $\mathbf{u}_{[M]}=[\mathbf{u}_{[S]},\mathbf{u}_{[X]}]$.
Thus $\mathbf{\overline{u}}_{m(v)}=[\mathbf{\overline{u}}_{m(e_{x})},\mathbf{\overline{u}}_{m(v')}]$.
By the induction hypothesis we have $\mathbf{\overline{u}}_{m(e_{x})},\mathbf{\overline{u}}_{m(v')}\in L(A)_{1}^{\mathbb{C}}/I(A)^{\mathbb{C}}$
and $0\ne\phi^{-1}(\mathbf{\overline{u}}_{m(e_{x})})\in\mathfrak{g}(\Delta)_{\deg e_{x}}$,
$0\ne\phi^{-1}(\mathbf{\overline{u}}_{m(v')})\in\mathfrak{g}(\Delta)_{\deg v'}$.
Hence $\mathbf{\overline{u}}_{m(v)}\in L(A)_{1}^{\mathbb{C}}/I(A)^{\mathbb{C}}$
and $\phi^{-1}(\overline{\mathbf{u}}_{m(v)})=[\phi^{-1}(\mathbf{\overline{u}}_{m(e_{x})}),\phi^{-1}(\mathbf{\overline{u}}_{m(v')})]\in\mathfrak{g}(\Delta)_{\deg v}\setminus\{0\}$
because $\deg e_{x}+\deg v'=\deg v$ is a root of $\mathfrak{g}(\Delta)$
by Proposition \ref{lem:deg-root}. Hence both (1) and (2) hold. Therefore
we may assume that $\rank M\ge2$.

Let $L$ be a GR-submodule of $M$ and $N:=M/L$. Then both $L$ and
$N$ are indecomposable. Set $v':=\udim L$, $v'':=\udim N$. In this
case $L$ is also preprojective and $L=M(v')$. By Proposition \ref{pro:GR-M}
we have $v''\not\in\mathbb{Z}\delta$ and there exists a Hall polynomial
$\varphi_{NL}^{M}$. Hence $a\mathbf{u}_{m(v)}=[\mathbf{u}_{m(v'')},\mathbf{u}_{m(v')}]$
for some $a\in\mathbb{Z}$. Here both $v'$ and $v''$ are positive
roots of $\chi_{A}$ with $\dim M(v')$, $\dim M(v'')<\dim M$. Therefore
by the induction hypothesis, both (1) and (2) hold for $v'$, $v''$.
Then by the statement (1) for $v'$, $v''$ we have $a\mathbf{\overline{u}}_{m(v)}=[\mathbf{\overline{u}}_{m(v'')},\mathbf{\overline{u}}_{m(v')}]\in L(A)_{1}^{\mathbb{C}}/I(A)^{\mathbb{C}}$,
and by (2) for $v'$, $v''$ we obtain $a\phi^{-1}(\mathbf{\overline{u}}_{m(v)})=[\phi^{-1}(\mathbf{\overline{u}}_{m(v'')}),\phi^{-1}(\mathbf{\overline{u}}_{m(v')})]\in\mathfrak{g}(\Delta)_{\deg v}\setminus\{0\}$
because $\deg v''+\deg v'=\deg v$ is a root of $\mathfrak{g}(\Delta)$
by Proposition \ref{lem:deg-root}. Thus $a\ne0$ and we finally have
both (1) and (2) for $v$.

Case 3. $M$ is preinjective. The dual argument works to show the
assertion.\hfil\qed

\section{\textbf{Example}}

\subsection{Basis vectors\label{sub:exm}}

For $\Delta=D_{5}$ we exhibit basis vectors of positive and negative
parts of $L(A)_{1}^{\mathbb{C}}/I(A)^{\mathbb{C}}\cong\mathfrak{g}(\Delta)$
in the Auslander-Reiten quiver of $A$. The positive part has 20 basis
vectors: 15 vectors are in the preprojective component (Fig.$\,$\ref{cap:preproj})
\begin{figure}
\includegraphics[%
  bb=50bp 470bp 596bp 742bp]{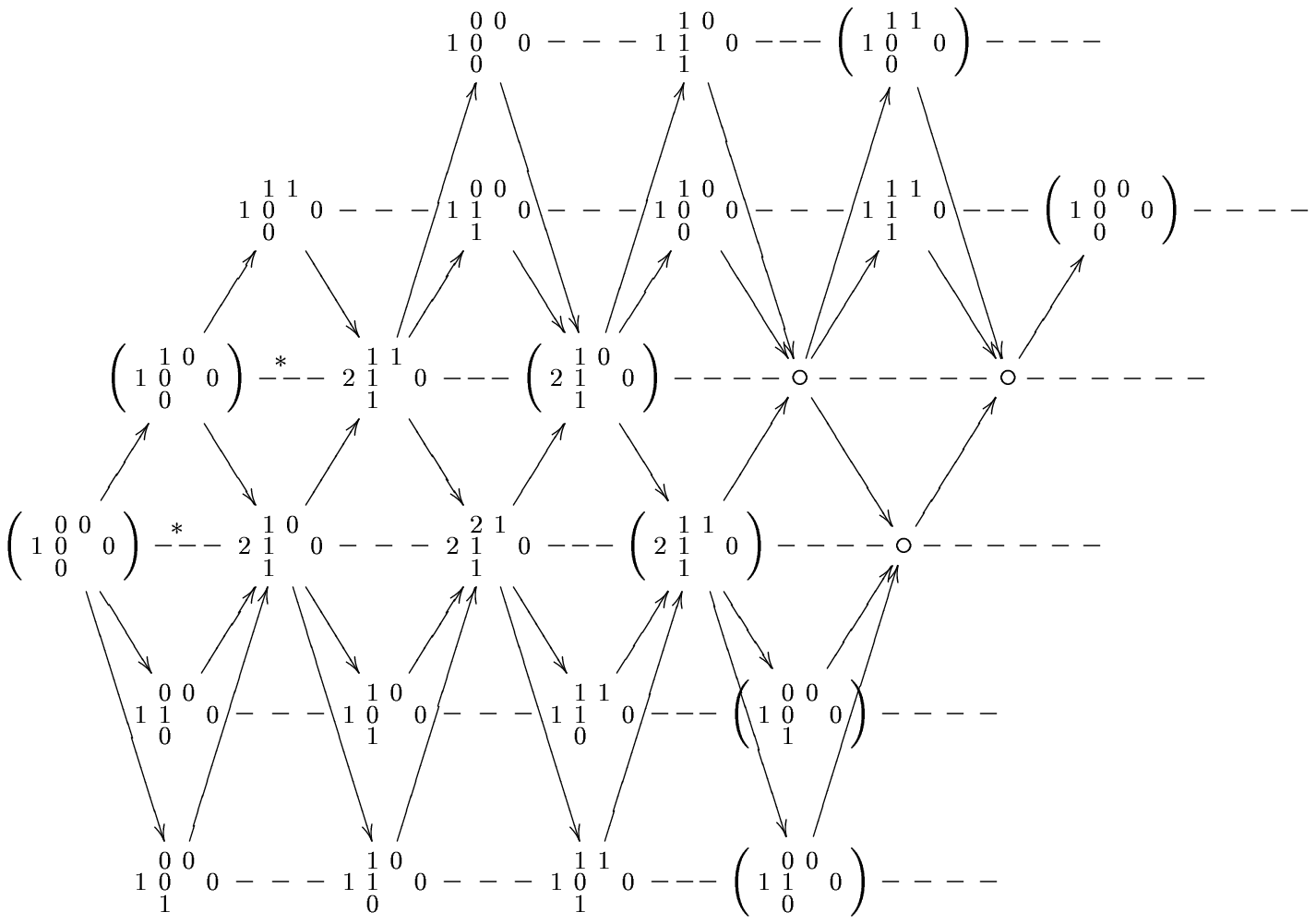}

\caption{\label{cap:preproj}15 basis vectors in the preprojective component}
\end{figure}
 and 5 vectors are in the non-homogeneous tubes (Figs.$\,$\ref{cap:tube3}
and \ref{cap:tube2}). Similarly the negative part has also 20 basis
vectors: 15 vectors are in the preinjective component (Fig.$\,$\ref{cap:preinj})
\begin{figure}
\includegraphics[%
  bb=140bp 463bp 500bp 750bp]{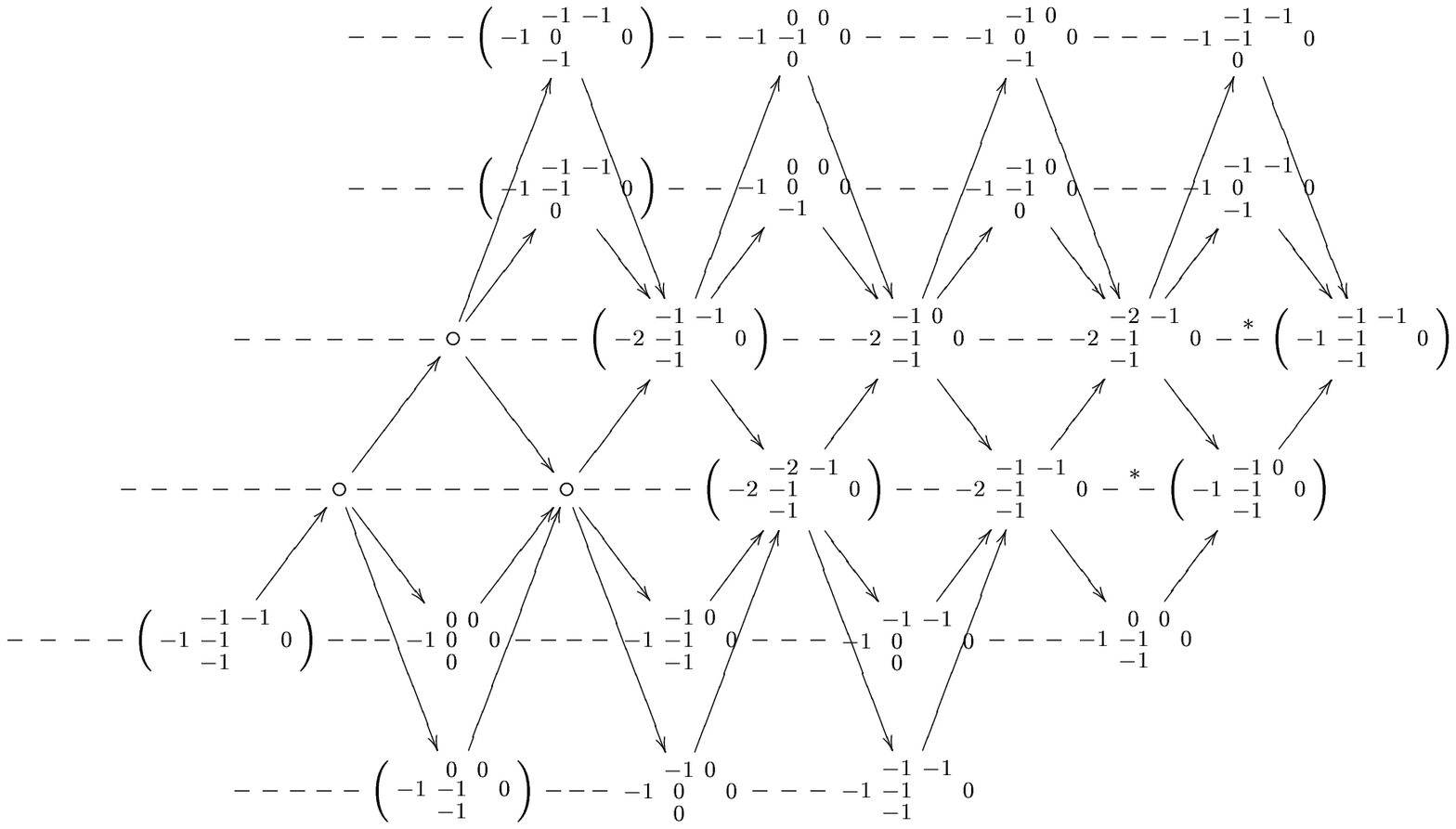}

\caption{\label{cap:preinj}15 basis vectors in the preinjective component}
\end{figure}
and 5 vectors are in the non-homogeneous tubes. In Fig$\,$\ref{cap:preproj}
vectors corresponding to indecomposable $A$-modules $M$ are given
by their degrees $\udim M-\dim M(\infty)\delta$ (see Definition \ref{def:degree}),
the broken lines stand for the Auslander-Reiten translation (from
the right to the left as usual); those with $*$ indicate that the
transformations from the left to the right are not given by the matrix
$\Phi^{-1}$, whereas those without $*$ indicate that the same transformations
are given by $\Phi^{-1}$. Vectors that are not chosen as representatives
of basis elements are written in parentheses. The vectors in parentheses
such that the same appear already on their left show us the action
of $\Phi^{-1}$ on $\udim\mathcal{P}_{r}$ (see \ref{pro:tau-delta}).
The dual remarks work for Fig.$\,$\ref{cap:preinj}. %
\begin{figure}
\includegraphics[%
  bb=0bp 590bp 596bp 729bp]{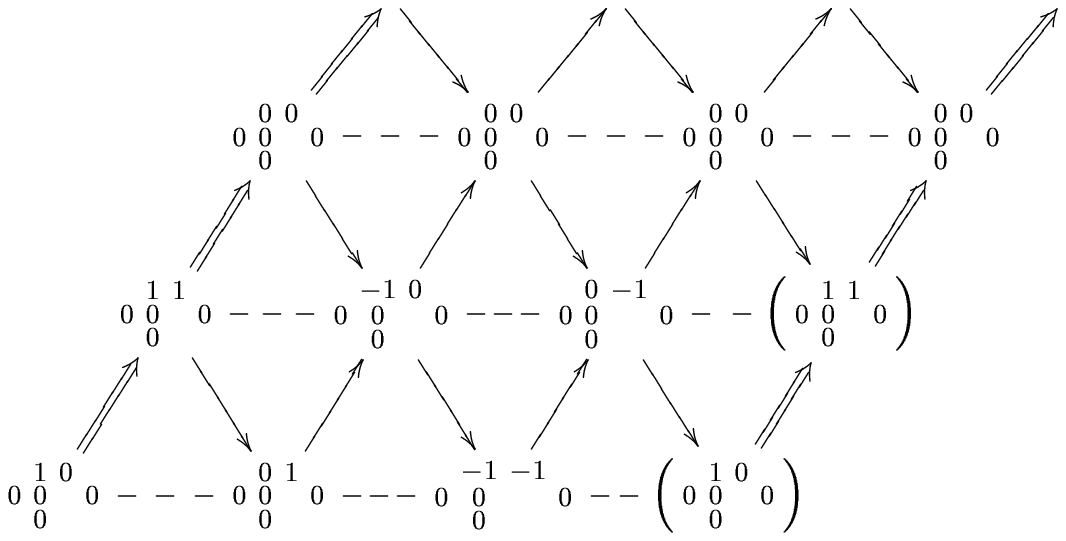}

\caption{\label{cap:tube3}Tube of rank 3}
\end{figure}
\begin{figure}
\includegraphics[%
  bb=80bp 650bp 596bp 750bp]{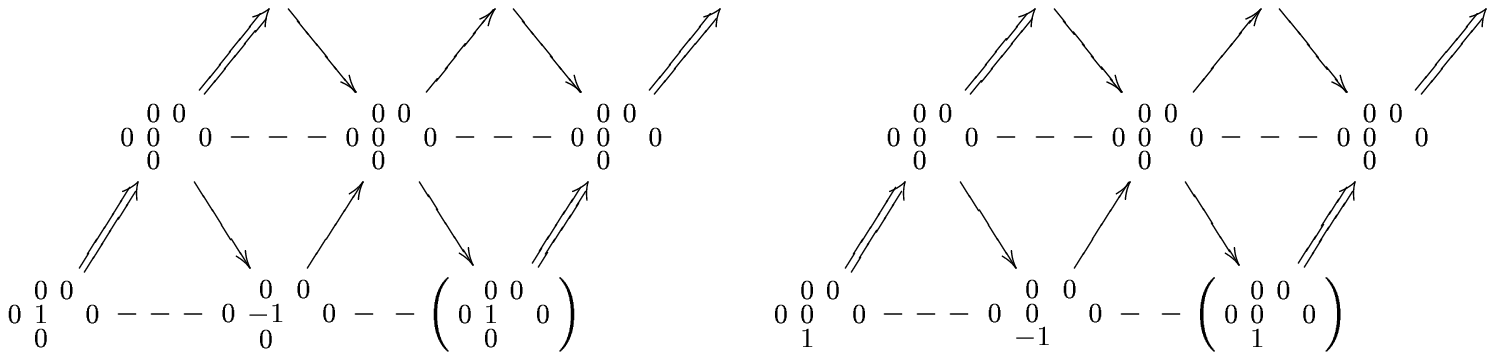}

\caption{\label{cap:tube2}Tubes of rank 2}
\end{figure}
 In Figures \ref{cap:tube3} and \ref{cap:tube2} the parallel arrows
drawn by double lines should be identified to form tubes.

\begin{rem}
\label{sub:Basis-vectors}In general the preprojective (resp.$\,$preinjective)
component over a domestic canonical algebra contains only basis vectors
of the positive (resp.$\,$negative) part because the dimension vector
of each preprojective (resp.$\,$preinjective) module takes the minimum
(resp.$\,$maximum) value at the vertex $\infty$ (see Remark \ref{rem:pos-neg}
for detail).
\end{rem}

\subsection{$E_{8}$ case}

By Proposition \ref{prp:root-space} (3) we see that if $v$ is a
positive root of $\chi_{A}$, then $\mathbf{\bar{u}}_{m(v+\delta)}=r_{v}\mathbf{\bar{u}}_{m(v)}$
for some $r_{v}\in\mathbb{C}^{\times}$. For a positive root $v$
of $\chi_{A}$ with $M(v)$ simple or regular we know that $r_{v}=1$.
Here we exhibit an example for $\Delta=E_{8}$ showing that this is
not always the case.

Let $v=\left[\begin{array}{cccccc}
 & 5 & 4 & 3 & 2\\
6 &  & 4 &  & 2 & 0\\
 &  & 3\end{array}\right]\in K_{0}(A)$. Then $v$ is a positive root of $\chi_{A}$ (with $M(v)$ exceptional),
and so are $v+e_{\infty}=\left[\begin{array}{cccccc}
 & 5 & 4 & 3 & 2\\
6 &  & 4 &  & 2 & 1\\
 &  & 3\end{array}\right]$ and $\deg(v+e_{\infty})=v+e_{\infty}-\delta=\left[\begin{array}{cccccc}
 & 4 & 3 & 2 & 1\\
5 &  & 3 &  & 1 & 0\\
 &  & 2\end{array}\right]$. A direct calculation shows that $[\mathbf{u}_{m(e_{\infty})},\mathbf{u}_{m(v)}]=\mathbf{u}_{m(v+e_{\infty})}$
but $[\mathbf{u}_{m(e_{\infty})},\mathbf{u}_{m(v+\delta)}]=-\mathbf{u}_{m(v+e_{\infty}+\delta)}$.
Hence we have $r_{v+e_{\infty}}=-r_{v}$, and at least one of these
cannot be 1.

In the first version of this paper we assumed that $I(A)$ contains
the differences $\mathbf{u}_{m(v+t\delta)}-\mathbf{u}_{m(v)}$ (namely
we assumed that $r_{v+(t-1)\delta}=1$) for all $v$ with $M(v)$
exceptional and $t\in\mathbb{N}$, and we found a serious error that
$L(A)_{1}^{\mathbb{C}}/I(A)^{\mathbb{C}}=0$ in this case. The present
version corrects this error.

\section*{\textbf{Acknowledgments}}

The first version of this work was done while I was visiting the University
of Bielefeld in 2003/2004. I would like to thank Professor Claus M.$\,$Ringel
for suggesting me the usefulness of canonical algebras, and for answering
my questions on their representations. I am grateful to Osaka City
University for the support of the visit. I would also like to thank
Bangming Deng, Christof Geiß, Andrew Hubery, and Dirk Kussin for helpful
discussions to prove Theorem 6.1 in the first version although that
is not used in this version. The first version was announced at ICRA
XI held in Patzcuaro, Mexico in 2004. An outline of the present version
was announced at ``Conference on Representation Theory and related
topics'' held at ICTP in Trieste, Italy in 2006.

\address{\noindent Department of Mathematics,\\
Faculty of Science,\\
Shizuoka University,\\
836 Ohya, Suruga-ku,\\
Shizuoka, 422-8529, Japan}
\end{document}